\documentclass[english]{smfart}

\usepackage{amsmath}
\usepackage{amssymb}
\usepackage{amscd}
\usepackage{latexsym}
\usepackage[latin1]{inputenc}

\usepackage{smfthm}
\newtheorem{theorem}[subsubsection]{Theorem}
\newtheorem{corollary}[subsubsection]{Corollary}
\newtheorem{lemma}[subsubsection]{Lemma}

\newtheorem{proposition}[subsubsection]{Proposition}

\numberwithin{equation}{section}

\author{Jean-Beno{î}t Bost}
\address{D{é}partement de Math{é}matiques, Universit{é}
Paris-Sud,
B{â}timent 425, 91405 Orsay cedex, France}
\email{jean-benoit.bost@math.u-psud.fr}
\author{Klaus K{ü}nnemann}
\address{Mathematik, 
Universit{ä}t Regensburg, 93040 Regensburg,
Germany}
\email{klaus.kuennemann@mathematik.uni-regensburg.de}

\title[The arithmetic Atiyah extension]
{Hermitian vector bundles and extension groups on arithmetic schemes II. \\
The arithmetic Atiyah extension}

\alttitle{Fibr\'es vectoriels hermitiens et groupes d'extensions sur les sch\'emas arithm\'etiques II. La classe d'Atiyah arithm\'etique} 

\dedicatory{Pour  Jean-Michel Bismut}

\begin{document}

\frontmatter

\begin{abstract}
In a previous paper, we have defined arithmetic
extension groups in the context of Arakelov geometry.
In the present one, we introduce an arithmetic analogue of the Atiyah extension that defines an element --- the arithmetic Atiyah class --- in a suitable arithmetic extension group.
Namely, if $\overline{E}$ is a hermitian vector bundle on an arithmetic scheme $X$,
its arithmetic Atiyah class 
$\widehat{{\rm at}}_{X/{\mathbb Z}}(\overline{E})$ lies in the group $\widehat{\rm Ext}^1_X(E,E\otimes\Omega_{X/{\mathbb Z}}^1)$, and 
is an obstruction to the algebraicity over $X$ of the unitary connection on the
vector bundle $E_{\mathbb C}$ over the complex manifold $X({\mathbb C})$
that is compatible with its holomorphic structure.

In the first sections of this article, we develop the basic properties of the arithmetic Atiyah class
which can be used to define characteristic classes in 
arithmetic Hodge cohomology.

Then we study the vanishing of the first Chern class 
$\hat c_1^H(\overline{L})$
of a hermitian line bundle $\overline{L}$
in the arithmetic Hodge cohomology group
$\widehat{\rm Ext}^1_X({\mathcal O}_X,\Omega_{X/{\mathbb Z}}^1)$.
This may be translated into a concrete problem of diophantine geometry, concerning
rational points of the
universal vector extension of the Picard variety of $X$. We investigate this problem, which was already considered and solved in some cases by Bertrand, by using a classical transcendence result of Schneider-Lang, and we derive a finiteness result for the
kernel of $\hat c_1^H$.

In the final section, we consider a geometric analog of our arithmetic
situation, namely a smooth, projective variety $X$
which is fibered on a curve $C$ defined over some field $k$ of 
characteristic zero. 
To any line 
bundle $L$ over $X$ is attached its relative Atiyah class ${\rm 
at}_{X/C}L$ in $H^1(X,\Omega^1_{X/C})$.
We describe precisely when ${\rm at}_{X/C}L$ vanishes. In particular, when 
the fixed part of the relative Picard variety of $X$ over $C$ is trivial, 
this holds iff some positive power of $L$ descends to a line bundle  
over $C$.
\end{abstract}

\begin{altabstract}
Dans un pr\'ec\'edent article, nous avons d\'efini des groupes d'extensions 
arithm\'eti\-ques dans le contexte de la g\'eom\'etrie d'Arakelov. 
Dans le pr\'esent travail, nous introduisons un analogue arithm\'etique de 
l'extension d'Atiyah; sa classe dans un groupe d'extensions arithm\'etiques 
convenable d\'efinit la classe d'Atiyah arithm\'etique. Plus pr\'ecis\'ement, 
pour tout fibr\'e vectoriel hermitien $\overline{E}$ sur un sch\'ema 
arithm\'etique $X$, sa classe d'Atiyah arithm\'etique 
$\widehat{{\rm at}}_{X/{\mathbb Z}}(\overline{E})$ appartient au groupe 
$\widehat{\rm Ext}^1_X(E,E\otimes\Omega_{X/{\mathbb Z}}^1)$ et constitue 
une obstruction
\`a l'alg\'ebricit\'e sur $X$ de l'unique connection unitaire sur la 
fibr\'e vectoriel $E_{\mathbb C}$ sur la vari\'et\'e complexe $X({\mathbb C})$ 
qui soit
compatible avec sa structure holomorphe.

Dans les premi\`eres sections de cet article, nous pr\'esentons la 
construction et les propri\'et\' es de base de la classe d'Atiyah, qui 
permettent notamment de d\'efinir des classes caract\'eristiques en 
cohomologie de Hodge arithm\'etique.

Nous \'etudions ensuite l'annulation de la premi\`ere classe de Chern
$\hat c_1^H(\overline{L})$
d'un fibr\'e en droites hermitien $\overline{L}$ dans le groupe de 
cohomologie de Hodge arithm\'etique
$\widehat{\rm Ext}^1_X({\mathcal O}_X,\Omega_{X/{\mathbb Z}}^1)$. 
La d\'etermination de tels fibr\'es en droites hermitiens se traduit en une 
question de g\'eom\'etrie diophantienne, concernant les points rationnels 
de l'extension vectorielle universelle de la vari\'et\'e de Picard de $X$. 
Nous \'etudions  ce probl\`eme --- qui a d\'ej\`a \'et\'e consid\'er\'e, 
et r\'esolu dans certains cas, par Bertrand --- au moyen d'un classique 
r\'esultat de transcendance d\^u \`a Schneider et Lang, et nous en 
d\'eduisons un th\'eor\`eme de finitude sur le noyau de 
 $\hat c_1^H$.
 
 Dans la derni\`ere section, nous \'etudions un analogue g\'eom\'etrique 
de la situation arithm\'etique pr\'ec\'edente. A savoir,  nous 
consid\'erons une vari\'et\'e projective lisse $X$ fibr\'ee sur une 
courbe $C$, au dessus d'un corps de base $k$ de caract\'eristique nulle 
et nous attachons \`a tout fibr\'e en droites $L$ sur $X$ sa classe d'Atiyah relative 
 ${\rm 
at}_{X/C}L$ dans $H^1(X,\Omega^1_{X/C})$. Nous d\'eterminons quand cette classe
 ${\rm at}_{X/C}L$ s'annule. Notamment, lorsque la vari\'et\'e de Picard 
relative de $X$ sur $C$ n'a pas de partie fixe, cela se produit 
pr\'ecis\'ement lorsque une puissance non-nulle de $L$ descend en un fibr\'e 
en droites sur $C$.
\end{altabstract}



\altkeywords{g\'eom\'etrie d'Arakelov, fibr\'es vectoriels hermitiens, groupes d'extensions, classe d'Atiyah, transcendance et groupes alg\'ebriques} 


%
\maketitle

\tableofcontents

\mainmatter

    \def\P{{\mathbb P}}
    \def\N{{\mathbb N}}
    \def\Z{{\mathbb Z}}
    \def\Q{{\mathbb Q}}
    \def\R{{\mathbb R}}
    \def\C{{\mathbb C}}
    \def\F{{\mathbb F}}
    \def\E{{\mathbb E}}
    \def\H{{\mathbb H}}
    \def\A{{\mathbb A}}
    \def\G{{\mathbb G}}

    \newcommand{\cA}{{\mathcal A}}
    \newcommand{\cC}{{\mathcal C}}
    \newcommand{\cE}{{\mathcal E}}
    \newcommand{\cI}{{\mathcal I}}
    \newcommand{\cO}{{\mathcal O}}
    \newcommand{\cU}{{\mathcal U}}    
    \newcommand{\cV}{{\mathcal V}}
    \newcommand{\cX}{{\mathcal X}}

    \renewcommand{\Re}{\,{\rm Re}\,}
    \renewcommand{\Im}{\,{\rm Im}\,}
    
      \newcommand{\Id}{{\rm Id}}
      \newcommand{\pr}{{\rm pr}}

    \newcommand{\ol}[1]{\overline{#1}}
    \newcommand{\cf}{\emph{cf.}\,}
    \newcommand{\ie}{\emph{i.e.}\,}
    \newcommand{\eg}{\emph{e.g.}\,}

    \newcommand{\at}{{\rm at}}
    \newcommand{\cat}{{\mathcal At}}
    \newcommand{\jet}{{\rm jet}}
    \newcommand{\cjet}{{\mathcal J\!et}}

    \newcommand{\Sm}{{\mathbf {SmPr}}}    
    \newcommand{\id}{\mbox{\rm id}}
    \newcommand{\rk}{{\rm{rk}}\,}
    \newcommand{\an}{{\rm hol}}
    \newcommand{\dolb}{{\overline{\partial}}}

    \newcommand{\Spec}{{\rm Spec}\,}
    \newcommand{\Proj}{{\rm Proj}}

    \newcommand{\Hom}{{\rm Hom}}
    \newcommand{\End}{{\rm End}}
    \newcommand{\Aut}{{\rm Aut}}
    \newcommand{\ihom}{{\mathcal Hom}}
    \newcommand{\iend}{{\mathcal End}}
    \newcommand{\im}{{{\rm im\, }}}
    \newcommand{\Coker}{{\rm Coker}}
    \newcommand{\Ker}{{\rm Ker}}
 \newcommand{\Lie}{{\rm Lie}\,}
 
    \newcommand{\cl}{{\rm cl}}
    \newcommand{\clar}{\widehat{{\rm cl}}}

    \newcommand{\NS}{{\rm{NS}}}
    \newcommand{\Pic}{{\rm{Pic}}}
    \newcommand{\Pico}{{\rm{Pic}}^0}

    \newcommand{\Ext}{\mbox{\rm Ext}}
    \newcommand{\Exthat}{\widehat{\rm Ext}}
    \newcommand{\dega}{\widehat{\rm deg}\,}

\setcounter{tocdepth}{1}

\setcounter{section}{-1}
\section{Introduction}

\noindent{\bf 0.1.}
This paper is a sequel to \cite{bostkuennemann1}, where we have defined and 
investigated arithmetic extensions on arithmetic schemes, and the groups they 
define.

Recall that if $X$ is a 
scheme over $\Spec \Z$, separated of finite type,  whose generic 
fiber $X_\Q$ is smooth, then an arithmetic extension of vector bundles 
over $X$ is the data $(\cE, s)$ of a short exact sequence of vector 
bundles (that is, of locally free coherent sheaves of $\cO_{X}$-modules) 
on the scheme $X,$
\begin{equation}\label{ext}\cE : 0 \longrightarrow G \stackrel{i}{\longrightarrow} E
\stackrel{p}{\longrightarrow} F \longrightarrow 0,
\end{equation}
and of a $\cC^\infty$-splitting
\[
s: F_\C \longrightarrow E_\C,
\]
invariant under complex conjugation, 
of the extension of $\cC^\infty$-complex
 vector bundles
\[
\cE_\C : 0 \longrightarrow G_\C \stackrel{i_\C}{\longrightarrow} E_\C
\stackrel{p_\C}{\longrightarrow} F_\C \longrightarrow 0
\]
on the complex manifold $X(\C)$, that is deduced from $\cE$ by the 
base change from $\Z$ to $\C$ and analytification.

For any two given vector bundles $F$ and $G$ over $X,$ the isomorphism 
classes of the so-defined arithmetic extensions of $F$ by $G$ 
constitute a set $\Exthat_{X}^1(F,G)$ that becomes an abelian group when 
equipped with the addition law defined by a variant of the classical 
construction of the Baer sum of $1$-extension of (sheaves of) 
modules\footnote{Consider indeed two arithmetic extensions of $F$ by $G$, say 
$\ol{\cE}_\alpha:=(\cE_\alpha, s_\alpha),$ $\alpha=1,2,$  defined by  extensions 
of vector bundles 
$\cE_\alpha : 0 \rightarrow G \stackrel{i_\alpha}{\rightarrow} E_\alpha
\stackrel{p_\alpha}{\rightarrow} F \rightarrow 0$
and $\cC^\infty$-splittings $s_\alpha: F_\C \rightarrow E_{\alpha,\C}$. 
We may define a vector bundle 
$E:=
\frac{\Ker(p_1-p_2:E_1\oplus E_2 { \rightarrow  } F)}
{{\rm Im}\,((i_1,-i_2): G {\rightarrow  } E_1\oplus E_2)}
$
over $X$. The Baer sum of $\ol{\cE}_1$ and $\ol{\cE}_2$ is the arithmetic 
extension $\ol{\cE}$ defined by the usual Baer sum of $\cE_1$ and $\cE_2$ --- namely
$\cE : 0 \rightarrow G \stackrel{i}{\rightarrow} E
\stackrel{p}{\rightarrow} F \rightarrow 0$
where the morphisms $i: G \rightarrow E$ and $p: E \rightarrow F$ are 
defined by $p([(g_1,g_2)]):=p_1(f_1)=p_2(f_2)$ and
$i(g):=[(i_1(g),0)]=[(0,i_2(g))]$ --- equipped with the $\cC^\infty$-splitting 
$s:F_\C \rightarrow E_\C$ defined by $s(e):=[(s_1(e),s_2(e))].$
}. 

Recall that a hermitian vector bundle $\ol{E}$ over $X$ is a pair
$(E,\|.\|)$ consisting of a vector bundle $E$ over $X$ and of a
$\cC^\infty$-hermitian metric, invariant under complex conjugation, on the
holomorphic vector bundle $E_{\C}$ over $X(\C).$ Examples of
arithmetic extensions in the above sense are provided by admissible
extensions  
\begin{equation}\label{admissible}
    \ol{\mathcal{E}}:\,0 \longrightarrow
\ol{G}\stackrel{i}{\longrightarrow}  \ol{E}
\stackrel{p}{\longrightarrow}
\ol{F}\longrightarrow  0
\end{equation}
of hermitian vector
bundles over $X$, namely from the
data of an extension
\[
\mathcal{E}:\,0\longrightarrow G \stackrel{i}{\longrightarrow}  E
\stackrel{p}{\longrightarrow}F
\longrightarrow  0
\]
of the underlying ${\mathcal O}_{X}$-modules such that the hermitian metrics
$\|.\|_{\ol{G}}$ and $\|.\|_{\ol{F}}$ on $G_{\C}$ and $F_{\C}$ are
induced (by restriction and quotients) by the metric $\|.\|_{\ol{E}}$
on
$E_{\C}$ (by means of the morphisms $i_{\C}$ and $p_{\C}$). Indeed, to
any such admissible extension is naturally attached its  orthogonal
splitting, namely the $\mathcal{C}^\infty$-splitting 
\[s_{\ol{\cE}}:F_{\C} \longrightarrow E_{\C}\]
that maps $F_{\C}$ isomorphically onto the orthogonal complement
$i_{\C}(G_{\C})^\perp$ of the image of $i_{\C}$ in $E_{\C}$. This
splitting is invariant under complex conjugation, and  
$(\mathcal{E},s_{\ol{\cE}})$ is an
arithmetic extension of $F$ by $G$. For any two hermitian vector
bundles $\ol{F}$ and $\ol{G}$ over $X,$ this construction establishes a
bijection from the set of isomorphism classes of admissible extension 
of  the form (\ref{admissible}) to the set $\Exthat^1_{X}(F,G)$.

In \cite{bostkuennemann1} we 
studied basic properties of the so-defined arithmetic extension groups. 
In particular, we introduced the
following natural morphisms of abelian groups:
\begin{itemize}
\item
the ``forgetful" morphism 
\[
\nu: \Exthat^1_{X}(F,G)
\longrightarrow \Ext^1_{\cO_X}(F,G),
\] 
which maps the class of an
arithmetic extension $(\cE,s)$ to the one of the underlying extension 
$\cE$ of $\cO_{X}$-modules;
\item
the morphism 
\[
b:{\rm Hom}_{\cC^\infty_{X(\C)}}(F_\C,G_\C)^{F_\infty} \longrightarrow
\widehat{\rm Ext}^1_X(F,G),
\]
defined on the real vector space $\Hom_{\cC^\infty_{X(\C)}}(F_\C,G_\C)^{F_\infty}$
of $\cC^\infty$-morphisms of vector bundles over $X(\C)$ from $F_{\C}$ to $G_{\C}$,
invariant under complex conjugation; it sends an element $T$ in this  space  to
the class of
the arithmetic
extension $(\mathcal{E},s)$ where $\mathcal{E}$ is the trivial
algebraic extension, defined by (\ref{ext}) with 
$E:=G\oplus F$ and $i$ and $p$ the obvious injection and projection
morphisms, and where $s$ is given by $s(f)=(T(f),f)$;
\item
the morphism 
\[
\iota:{\rm Hom}_{{\mathcal O}_X}(F,G)
\longrightarrow 
{\rm Hom}_{\cC^\infty_{X(\C)}}(F_\C,G_\C)^{F_\infty}
\] 
which sends a morphism  $\varphi: F \rightarrow G$ of vector 
bundles over $X$ to the  morphism of $\mathcal{C}^\infty$-complex vector bundles 
$\varphi_\C: F_\C \rightarrow G_\C$ deduced from $\varphi$ by base 
change from $\Z$ to $\C$ and analytification;
\item
the morphism 
\[
\Psi:\widehat{\rm Ext}_{X}^1(F,G)\longrightarrow
Z_{\overline{\partial}}^{0,1}(X_\R,F^\lor \otimes G), 
\]
that takes values in the real vector space
\[
Z_{\overline{\partial}}^{0,1}(X_\R,F^\lor \otimes G)
:=Z_{\overline{\partial}}^{0,1}(X(\C),F_\C^\lor \otimes G_\C)^{F_\infty}
\] 
of
$\ol{\partial}$-closed forms of type $(0,1)$ on $X(\C)$ with
coefficients in $F^\lor_\C\otimes G_\C$, invariant under complex
conjugation. It maps the class of an arithmetic extension $(\cE,s)$ to
its ``second fundamental form" $\Psi(\cE,s)$ defined by 
\[
i_{\C}\circ \Psi(\cE,s)
=\ol{\partial}_{F^\lor_\C\otimes G_\C}(s).
\]
\end{itemize}
We also established
the following basic exact sequence: 
\begin{equation}\label{eq:longexseq1}
{\rm Hom}_{{\mathcal O}_X}(F,G)
\stackrel{\iota}{\rightarrow }
{\rm Hom}_{\cC^\infty_{X(\C)}}(F_\C,G_\C)^{F_\infty}
\stackrel{b}{\rightarrow }
\widehat{\rm Ext}_X^1(F,G)
\stackrel{\nu}{\rightarrow } {\rm Ext}_{\cO_X}^1(F,G)\rightarrow 0,
\end{equation}
which displays the arithmetic extension group 
$\widehat{\rm Ext}_X^1(F,G)$ as an extension of the ``classical" extension group  
${\rm Ext}_{{\cO}_X}^1(F,G)$ by a group of analytic type. 

The sequel of \cite{bostkuennemann1} was devoted to the study of the
groups 
$\widehat{\rm Ext}_X^1(F,G)$ when the base scheme is an arithmetic curve, that is,
the spectrum $\Spec \cO_K$ of the ring of integers of some number 
field $K$. In this special case, these extension groups
appear as natural tools in geometry of numbers and reduction theory in
their modern guise, namely the study of hermitian vector bundles over 
arithmetic curves and their admissible extensions.

In the present paper, we focus on a natural construction of arithmetic
extensions attached to hermitian vector bundles over an arithmetic
scheme $X$ as above, their \emph{arithmetic Atiyah extensions}. In
contrast with the arithmetic extensions over arithmetic curves
investigated in \cite{bostkuennemann1}, for which the interpretation
as admissible extensions was crucial, the arithmetic Atiyah extensions
are genuine examples of arithmetic extensions constructed as pairs
$(\cE,s)$ --- where $s$ is a $\cC^\infty$-splitting of an extension
$\cE$ of vector bundles over $X$ --- and not derived from an
admissible extension.

\vspace{\baselineskip}

\noindent{\bf 0.2.}
Atiyah extensions of vector bundles were initially introduced 
by Atiyah \cite{atiyah57} in the context of complex analytic geometry.

Namely, for any holomorphic vector bundle $E$ over a complex 
manifold $X,$ Atiyah introduces the holomorphic vector bundle
$ P^1_{X}(E)$ of jets of order one of sections of $E$, whose fiber
$ P^1_{X}(E)_{x}$ at a point $x$ of $X$ 
is by definition the space of sections of $E$ over
the first order thickening $x_{1}:=\Spec \cO_{X,x}/\mathfrak{m}_{x}^2$ of 
$x$ in $X$. Here, as usual, $\cO_{X}$ denotes the sheaf of
holomorphic functions over $X$, and $\mathfrak{m}_{x}$ the maximal ideal
of its stalk $\cO_{X,x}$ at $x$. 

The vector bundle $ P^1_{X}(E)$ fits into a short exact sequence of
holomorphic vector bundles
\begin{equation}\label{AtiyahAnIntro}
    \cat_{X}E: 0 \longrightarrow E\otimes  \Omega^1_{X}
\stackrel{i}{\longrightarrow} P^1_{X}(E) \stackrel{p}{\longrightarrow}
E \longrightarrow 0,
\end{equation}
where the morphisms $i$ and $p$ are defined as follows: for any point $x$ in 
$X$, the map $i_{x}: E_x\otimes \Omega^1_{X,x} \rightarrow P^1_{X}(E)_{x}$
maps an element $v$ in $E_x \otimes \Omega^1_{X,x} \simeq
\Hom_{\C}(T_{X,x},E_{x})$ to the section of $E$ over $x_{1}$ that
vanishes at $x$ and admits $v$ as differential, while the map $p_{x}: P^1_{X}(E)_{x}
\rightarrow E_{x}$ is simply the evaluation at $x$.

The Atiyah extension of $E$ is precisely the extension $\cat_{X}E$ of
$E$ by $E\otimes \Omega^1_{X}$ so-defined. According to its very definition, its 
class ${\rm at}_{X}E$ in
the group ${\rm Ext}^1_{\cO_{X}}(E, E\otimes \Omega^1_{X})$ which
classifies  extensions of holomorphic vector bundles of $E$ by
$E\otimes \Omega^1_{X}$ is the
obstruction to the existence of a holomorphic connection
\[
\nabla : E \longrightarrow E\otimes \Omega^1_{X}
\]
on the vector bundle $E$.

The point of Atiyah's article \cite{atiyah57} is that the class ${\rm at}_{X}E$
also leads to a straightforward construction of characteristic classes 
of $E$ with values in the so-called Hodge cohomology groups of $X$
\begin{equation}\label{Hodgegroups}
    H^{p,p}(X):=H^p(X, \Omega^p_{X}).
\end{equation}
For instance, Atiyah defines a first Chern class $c_{1}^H(E)$ in
    $H^{1,1}(X)=H^1(X, \Omega^1_{X})$ as the image of ${\rm at}_{X}E$ 
    by the morphism
    \[
    \begin{array}{ccccc}
    {\rm Ext}^1_{\cO_{X}}(E,E\otimes\Omega^1_{X}) &\simeq & 
    {\rm Ext}^1_{\cO_{X}}(\cO_{X},{\iend}\,E\otimes \Omega^1_{X})&\\
    & & \hspace*{1.8cm}\downarrow 
    {\scriptstyle ({\rm Tr}_{E}\otimes {\rm id}_{\Omega^1_{X}} )\circ \,\_}&\\
    & & {\rm Ext}^1_{\cO_{X}}(\cO_{X}, \Omega^1_{X})& \simeq & H^1(X, \Omega^1_{X}) 
    \end{array}
    \]
    deduced from the canonical trace morphism
    \[
    \begin{array}{crcl}
        {\rm Tr}_{E}: & {\iend}\,E \simeq E^\lor \otimes E &
	\longrightarrow & \cO_{X},  \\
         & \lambda \otimes v & \mapsto & \lambda(v).
    \end{array}\]
    Higher degree characteristic classes are constructed by means of
    the successive powers $({\rm at}_{X}E)^p$ in ${\rm Ext}^p_{\cO_{X}}(\cO_{X},
    ({\iend}\,E)^{\otimes p}\otimes\Omega^p_{X}),$ where $p$ denotes a positive integer.
    For instance, the $p$-th Segre class, associated to the $p$-th
    Newton polynomial $X_{1}^p+\cdots +X_{\rk E}^p$, may be
    constructed in the Hodge cohomology group $H^p(X, \Omega^p_{X})$ as
    \[
    s_{p}^H(E):= ({\rm Tr}^p_{E}\otimes {\rm id}_{\Omega^p_{X}})\circ
    ({\rm at}_{X}E)^p,
    \]
    where 
    \[
    \begin{array}{crcl}
        {\rm Tr}^p_{E}: & ({\iend}\, E)^{\otimes p}  &
	\longrightarrow & \cO_{X},  \\
         & T_{1}\otimes \ldots \otimes T_{p}  & \mapsto &
	 {\rm Tr}_{E}(T_{1}\ldots T_{p}).
    \end{array}\]
    
    When the manifold $X$ is compact and K\"ahler (\emph{e.g.},
    projective), the Hodge cohomology group $H^p(X, \Omega^p_{X})$ may
    be identified with a subspace of the complex de Rham cohomology
    group $H^{2p}_{\rm{dR}}(X,\C)$ of $X$, and Atiyah's construction
    of characteristic classes is compatible (up to normalization) to
    classical topological constructions.
    
    The definition of the Atiyah class and the construction of the
    associated characteristic classes obviously make sense in a purely
    algebraic context, say over a base field $k$ of characteristic
    zero. If $X$ is a smooth algebraic scheme over $k$, for any vector
    bundle $E$ over $X$, its Atiyah class ${\rm at}_{X/k}E$ is
    constructed as above, \emph{mutatis mutandis}, as an element of
    the $k$-vector space ${\rm Ext}^1_{\cO_{X}}(E, E\otimes\Omega^1_{X/k})$, 
    and the associated characteristic classes are elements of the
    Hodge cohomology groups of $X$ defined similarly to (\ref{Hodgegroups}), but 
    now using the Zariski topology of 
    $X$ instead of the analytic one, and  the sheaf of
    K\"ahler differentials $\Omega^p_{X/k}$ instead of the
    holomorphic differential forms $\Omega^p_{X}$.
    
    These constructions are especially suited to smooth algebraic schemes $X$ that are 
    proper over $k$. In this case, the ``Hodge to de Rham" spectral
    sequence degenerates, and the Hodge group $H^{p,p}(X)$ gets
    identified to a 
    subquotient 
    of the Hodge filtration
    of the algebraic de Rham cohomology group 
    $H^{2p}_{\rm{dR}}(X/k):=
    H^{2p}(X,\Omega^{\cdot}_{X/k})$. 
    Moreover, when $X$
    is proper over
    $k=\C$, this algebraic construction is  compatible with the
    previous analytic one, as a consequence of the GAGA principle.
    
    This algebraic version of Atiyah's constructions has been
    considerably extended by Illusie \cite{illusie71}. Instead of a
    smooth algebraic scheme over a field $k$, he considers a suitable morphism 
    of ringed topoi $f: X \rightarrow S,$ 
and associates Atiyah classes and characteristic classes to
perfect complexes of sheaves of $\cO_{X}$-modules; his definitions
involve the cotangent complex $\mathbb{L}^\cdot_{X/S}$ of $X$ over $S$, which
in this general setting plays the role of the sheaf $\Omega^1_{X/k}$ attached
to a smooth scheme $X$ over the field $k$. 
Let us also mention the presentation of this ``algebraic" theory and of some 
of its developments in the monograph of  Ang\' eniol and Lejeune-Jalabert 
\cite{angenioletal89}, and the analytic construction of Buchweitz and 
Flenner \cite{buchweitzflenner98}, \cite{buchweitzflenner03}\footnote{These  authors work in 
an analytic context as 
the original article \cite{atiyah57}, but extend the construction of Atiyah 
classes to complex of coherent analytic sheaves over possibly singular 
complex spaces. Like Illusie's construction, this requires to deal 
with the cotangent complex, now in an analytic context.}.

\vspace{\baselineskip}

\noindent{\bf 0.3.} Let us briefly describe our construction of arithmetic 
Atiyah classes.

Let $\ol{E}:=(E,\|.\|_{E})$ be a hermitian 
vector bundle over a scheme $X$ which is separated and of finite type over $\Z$,
and which for 
simplicity will be assumed  smooth over $\Z$ in this introduction. 
The relative version of the exact sequence (\ref{AtiyahAnIntro})
defines the Atiyah extension of $E$ over $\Z$:
\begin{equation}\label{AtiyahZIntro}
    \cat_{X/\Z}E: 0 \longrightarrow E\otimes \Omega^1_{X/\Z}
\stackrel{i}{\longrightarrow} P^1_{X/\Z}(E) \stackrel{p}{\longrightarrow}
E \longrightarrow 0.
\end{equation}

Besides, according to a classical result of Chern and 
Nakano (\cite{chern46, nakano55}), the 
holomorphic vector bundle $E_{\C}^\an$ over the complex manifold $X(\C)$,
seen as $\cC^\infty$-vector bundle, admits a unique connection $\nabla_{\ol{E}}$
that is  unitary with respect to the hermitian metric $\|.\|_{E}$, and
moreover is
compatible with its holomorphic structure in the sense that its
component $\nabla_{\ol{E}}^{0,1}$ of type $(0,1)$ coincides with the
$\ol{\partial}$-operator $\ol{\partial}_{E_{\C}}$ with coefficients in 
the holomorphic vector bundle
$E_{\C}^\an.$ The component $\nabla_{\ol{E}}^{1,0}$ of type $(1,0)$ of 
$\nabla_{\ol{E}}$
defines a $\cC^\infty$-splitting $s_{\ol{E}}$ of the Atiyah extension of 
the holomorphic vector bundle $E_{\C}^\an$:
\[
\cat_{X(\C)}E_{\C}: 0 \longrightarrow \Omega^1_{X(\C)}\otimes
E_{\C}
\stackrel{i_{\C}}{\longrightarrow} P^1_{X(\C)}(E_{\C})
\stackrel{p_{\C}}{\longrightarrow}
E_{\C} \longrightarrow 0.\]
Namely, for any point $x$ in $\cX(\C)$ and any $e$ in $E_{x}$, $s_{\ol{E}}(e)$ 
is the section of $E$ over $x_1$ that takes the value $e$ at $x$ and
is killed by $\nabla_{\ol{E}}^{1,0}$. 

Since the above analytic Atiyah extension $\cat_{X(\C)}E_{\C}$ is
precisely the extension deduced from $\cat_{X/\Z}E$ by the base change 
from $\Z$ to $\C$
and analytification, the pair $(\cat_{X/\Z}E,s_{\ol{E}})$ defines an
arithmetic extension, the \emph{arithmetic Atiyah extension} 
$\widehat{\cat}_{X/\Z} \ol{E}$ of the
hermitian vector bundle $\ol{E}$. Its class $\widehat{\rm at}_{X/\Z} \ol{E}$ 
in $\widehat{\rm Ext}^1_X(E,E \otimes \Omega^1_{X/\Z})$
--- the \emph{arithmetic Atiyah class} of $\ol{E}$ --- is mapped by the 
forgetful morphism $\nu$ to the 
``algebraic" Atiyah class ${\rm at}_{X/\Z} E$ of $E$ in 
${\rm Ext}^1_{\cO_X}(E,E \otimes \Omega^1_{X/\Z})$ (defined by
the extension $\cat_{X/\Z}E$) and by the ``second fundamental form"
morphism $\Psi$ to the curvature 
form of the Chern-Nakano connection $\nabla_{\ol{E}}$ (up to a sign).

\vspace{\baselineskip}

\noindent{\bf 0.4.}
In the first section of this article, we begin by reviewing the
constructions of the Atiyah extension in the classical $\C$-analytic and
algebraic frameworks. For the sake of simplicity, we deal with locally
free coherent sheaves only, and follow a naive approach --- we work
with relative differentials, and not with their ``correct" derived
version defined by the cotangent complex. This naive approach is 
sufficient when one considers --- as we shall in the sequel --- relative 
situations $f: X \rightarrow S$ where $X$ is integral, and $f$ is l.c.i. 
and generically smooth,  in which case $\mathbb L^\cdot_{X/S}$ is 
quasi-isomorphic to $\Omega^1_{X/S}.$

Then, in Section \ref{atiyah1}, we construct the arithmetic Atiyah class in
the following relative situation, which extends the one considered
in the previous paragraphs. Consider 
arithmetic schemes $X$ and
$S$, flat  over an arithmetic ring 
$(R,\Sigma, F_{\infty})$ (in the sense of 
\cite[3.1.1]{gilletsoule90}; see also \cite[1.1]{bostkuennemann1}),
and a morphism of $R$-schemes $\pi:X \rightarrow S$, smooth over the
fraction field $K$ of $R$. Then, to any hermitian vector bundle
$\ol{E}$ over $X$, we attach a class $\widehat{\rm at}_{X/S} \ol{E}$ in 
$\widehat{\rm Ext}^1_X(E,E \otimes \Omega^1_{X/S})$.  
Applying a trace morphism to
this class, we define the \emph{first Chern class} $\hat c_{1}^H(\ol{E})$
of $\ol{E}$ \emph{in arithmetic Hodge cohomology}, that lies  in the
group 
\[
\widehat{H}^{1,1}(X/S):=\widehat{\rm
Ext}^1_X(\cO_{X}, \Omega^1_{X/S}).
\]
The class $\widehat{\rm at}_{X/S} \ol{E}$ and its trace $\hat c_{1}^H(\ol{E})$ 
satisfy compatibility properties with pull-back and tensor operations 
on hermitian vector bundles that extend well-known properties of the classical
Atiyah  and first Chern  classes. In particular we construct a
functorial homomorphism
\[
\hat c_{1}^H:\widehat{\rm Pic} (X) \longrightarrow
\widehat{H}^{1,1}(X/S)
\]
from the group of isomorphism classes of hermitian line bundles over
$X$ to the arithmetic Hodge cohomology group.

In the last sections of this article, we investigate the kernel
of this morphism. It trivially vanishes on the image of 
\[
\pi^\ast : \widehat{\rm Pic} (S) \longrightarrow \widehat{\rm Pic}(X),
\] 
and we may wonder ``how large" this image $\pi^\ast(\widehat{\rm Pic} (S))$ is 
in $\ker \hat c^H_{1}.$ 

This 
question becomes a concrete problem of Diophantine geometry when the 
base arithmetic ring is a number field $K$ equipped with a non-empty set $\Sigma$
of embeddings $\sigma:K\hookrightarrow \C$ stable under complex conjugation, and 
when $S$ is $\Spec K$ and $X$ is projective over $K$. Indeed, in 
this case, the class of a hermitian line bundle $\ol{L}= (L, 
\|.\|_{L})$ over $X$ lies in the kernel of $\hat c^H_{1}$ precisely when $L$ 
admits an algebraic connection
$\nabla: L \rightarrow L \otimes \Omega^1_{X/K}$, defined over $K$,
such that the induced holomorphic connection 
$\nabla_{\C}: L_{\C} \rightarrow L_{\C}\otimes \Omega^1_{X_\Sigma(\C)}$
on the holomorphic line bundle $L_{\C}$ over 
\[
X_\Sigma(\C):= \coprod_{\sigma \in \Sigma} X_{\sigma}(\C)
\]
is unitary with respect to the hermitian metric $\|.\|_{L}.$

One easily checks that, if $L$ has a torsion class in ${\rm Pic}(X)$ and if the 
metric $\|.\|_{L}$ has vanishing curvature on $X_\Sigma(\C)$, then their exists such a 
connection. Moreover the converse implication, namely

\noindent ${\bf I1}_{X,\Sigma}$: \emph{if a hermitian line bundle $\ol{L}= (L, 
\|.\|_{L})$ over $X$ admits an algebraic connection $\nabla$ defined over $K$ such 
that the connection $\nabla_{\C}$ on $L_\C$ over $X_\Sigma(\C)$ is unitary with 
respect to $\|.\|_{L},$ then $L$ 
has a torsion class in ${\rm Pic}(X)$ and the metric $\|.\|_{L}$ has vanishing 
curvature,}

\noindent turns out to be equivalent with the following condition, where $\pi$ 
denotes the structural morphism from $X$ to $\Spec K$:

\noindent ${\bf I2}_{X,\Sigma}$:  \emph{the image of $\pi^\ast: \widehat{\rm Pic} 
(\Spec K) \rightarrow \widehat{\rm Pic}
(X)$ 
has finite index in the kernel
of}
\[
\hat c_1^H:\widehat{\rm Pic} (X) \longrightarrow
\widehat{H}^{1,1}(X/K).
\]

The equivalent assertions ${\bf I1}_{X,\Sigma}$ and ${\bf I2}_{X,\Sigma}$ may be translated in 
terms of $K$-rational points of the universal vector extension of the 
Picard variety of $X$. In this formulation, their validity 
 has been established by Bertrand \cite{bertrand95, bertrand98} when $\Sigma$ 
has a unique element (necessarily a real embedding of $K$) and when this Picard 
variety admits ``real 
multiplication"\footnote{namely, if this Picard variety $A$ has dimension $g$, 
the $\Q$-algebra ${\rm End}(A/K)\otimes_\Z \Q$ is assumed to be a totally real 
field of degree $g$ over $\Q$. Actually, Bertrand establishes a more precise 
result, concerning $g$ independent extensions of $A$ by the additive group 
$\mathbb{G}_a$; see   \cite[Theorem 3, pages 13-14]{bertrand98}.} as a 
consequence of 
the analytic subgroup theorem of W\"ustholz 
(\cite{wuestholz89}). 

Inspired by \cite{bertrand95, bertrand98} --- which we tried to understand in 
more geometric terms, avoiding the explicit use of differential forms and their 
periods, but working with  algebraic groups and their exponential maps--- we 
establish in Section  \ref{sect3}  the validity of ${\bf I1}_{X,\Sigma}$ and 
${\bf I2}_{X,\Sigma}$ when $\Sigma$ is 
arbitrary without any assumption on the Picard variety of $X$.
The proof proceeds by reducing to the case where $X$ is an abelian variety, and $\Sigma$ has a unique or two conjugate elements. To handle this case, 
we use a classical transcendence theorem of Schneider-Lang 
characterizing Lie algebras of algebraic subgroups of  commutative algebraic 
groups over number fields.  Our argument is presented in the first part of Section 3, and may be read independently of the rest of the article.


The validity of  ${\bf I1}_{X,\Sigma}$ and ${\bf I2}_{X,\Sigma}$ demonstrates 
that the first Chern class $\hat c^H_{1}(\ol{L})$ in the group 
$\widehat{H}^{1,1}(X/K)$ encodes  quite non-trivial Diophantine 
informations. In a later part of this work, we plan to study  
characteristic classes of higher degree, with values in the arithmetic Hodge 
cohomology groups   
\[
\widehat{H}^{p,p}(X/S):=\widehat{\rm
Ext}^p_X(\cO_{X}, \Omega^p_{X/S})
\]
defined as special instances of the higher arithmetic extension 
groups introduced in \cite[0.1]{bostkuennemann1}, that are deduced 
from the powers of the arithmetic Atiyah class $\widehat{\rm 
at}_{X/S}\ol{E}$ using suitably defined products on the higher arithmetic 
extension groups.

Let us also indicate that, starting from the results in Section 3, 
one may derive  finiteness results on  
$\ker \hat c^H_{1}/\pi^\ast(\widehat{\rm Pic} (S))$  for more general smooth 
projective morphisms $\pi: X \rightarrow S$ of arithmetic schemes over 
arithmetic rings, by considering the restriction of $\pi$ over points of 
$S$ rational over some number field. We leave this to the interested reader. 

In the final section of the article, we establish a geometric 
analogue of condition 
${\bf I1}_{X,\Sigma}$. 
We consider a smooth, projective, 
geometrically connected curve $C$ over some field $k$ of 
characteristic zero, its function field $K:=k(C),$ and a smooth 
projective variety $X$ over $k$ equipped with a dominant $k$ morphism 
$f:X\rightarrow C$, with geometrically connected fibers. To any line 
bundle $L$ over $X$ is attached its relative Atiyah class ${\rm 
at}_{X/C}L$ in $H^1(X,\Omega^1_{X/C}).$ We show that, when the fixed 
part of the abelian variety ${\rm Pic}^0_{X_{K}/K}$ is trivial, then   
${\rm 
at}_{X/C}L$ vanishes iff some positive power of $L$ is isomorphic to a line bundle of the 
form $f^\ast M$, where $M$ is a line bundle over $C.$ The proof relies on the Hodge index theorem expressed in the Hodge 
cohomology groups of $X$.

Considering the classical analogy between number fields and function fields, it 
may be interesting to observe that, when investigating the kernel of the 
relative Atiyah class of line bundles, a transcendence result --- in the 
guise of a criterion for a subspace of the Lie algebra of a commutative 
algebraic group to define an algebraic subgroup --- plays a key role in 
the ``number field case", while our main tool in the ``function field 
case" is intersection theory in Hodge cohomology.

In Appendix A, we describe arithmetic extension groups in terms of 
\v{C}ech cocycles.
Based on this description,  in the main part of the paper we calculate
explicit \v{C}ech cocycles which represent the arithmetic Atiyah class 
and the first Chern class in arithmetic Hodge cohomology.
Finally Appendix B summarizes basic facts concerning universal vector 
extensions of Picard varieties that are used in Sections 3 and 4.

It is a pleasure to thank 
A. Chambert-Loir and D. Bertrand for helpful discussions and
S. Kudla, M. Rapoport and J. Schwermer for invitations to the ESI 
in Vienna where part of the work on this paper was done.
We are grateful to the TMR network `Arithmetic geometry'
and the DFG-Forschergruppe `Algebraische Zykel und $L$-Funktionen'
for their support and to the universities of Regensburg and
Paris-Sud (Orsay) for their hospitality.
Finally we wish to thank the referee for his careful reading and his helpful suggestions for improving the exposition.

\section{Atiyah extensions in algebraic and analytic geometry}\label{sect1}

\subsection{Definition and basic properties}\label{atgen}

We consider simultaneously the algebraic and the analytic situation where
$\pi:X\to S$ is a morphism of locally ringed spaces which is either
\begin{enumerate}
\item[a)] 
a separated morphism of finite presentation between schemes, or
\item[b)]
an analytic  morphism between complex analytic spaces.
\end{enumerate}

We denote in both cases by $\cO_X$ the structure sheaf of regular resp. 
holomorphic functions on $X$.
Let $I$ denote the ideal sheaf, and
\[
\Delta^{(1)}:X^{(1)}\longrightarrow X\times_S X
\] 
the first infinitesimal neighborhood of the diagonal 
$\Delta\colon\,X\to X \times_S X$.
For $i=1,2,$ let $q_i:X^{(1)}\rightarrow X$
denote the composition of
$\Delta^{(1)}$ with the $i$-th projection.
We identify $(\Omega_{X/S}^1,d)$ with the $\cO_X$-module $I/I^2$ and the
universal derivation
\begin{equation}\label{kaehlerid}
d:\cO_X\longrightarrow I/I^2\,,\,\,d(\lambda)=q_2^*(\lambda)-q_1^*(\lambda).
\end{equation}
The $\cO_X$-modules $q_{1*}\cO_{X^{(1)}}$ and 
$q_{2*}\cO_{X^{(1)}}$ are 
canonically isomorphic as sheaves of $\cO_S$-modules.
We denote this $\cO_S$-module by $P^1_{X/S}$ and observe that 
$P^1_{X/S}$ carries two natural $\cO_X$-module structures via 
the left and right projection $q_1$ and $q_2$.
The canonical extension
\[
0\longrightarrow I/I^2 \longrightarrow \mathcal{O}_{X\times_S
  X}/I^2 \longrightarrow \mathcal{O}_{X\times_S X} / I \longrightarrow 0
\]
yields an exact sequence of $\cO_X$-modules
\begin{equation}\label{tae}
 0 \longrightarrow \Omega^1_{X/S} \longrightarrow P^1_{X/S} \longrightarrow \cO_X \longrightarrow 0
\end{equation}
for both $\cO_X$-module structures.
The left and right $\cO_X$-module structures yield canonical 
but different $\cO_X$-linear
splittings of (\ref{tae}) which map $1 \,{\rm mod}\, I$ to $1 \,{\rm mod} \,I^2$.

\subsubsection{}
Let $F$ denote a vector bundle (that is, a locally free coherent sheaf) on $X$.
We obtain from (\ref{tae}) an exact sequence of $\cO_X$-modules
\[
\cjet^1_{X/S}(F)\colon 
0 \longrightarrow F \otimes \Omega^1_{X/S} \stackrel{i_F}{\longrightarrow}
 P^1_{X/S}(F) \stackrel{p_F}{\longrightarrow} F \longrightarrow 0
\]
where 
\begin{equation}\label{definejet}
P^1_{X/S}(F)=q_{1*}q_2^*F.
\end{equation}
Indeed we have
\[
P^1_{X/S}(F)=P^1_{X/S} \otimes F 
\]
where the tensor product  is taken using the right
$\cO_X$-module structure on $P^1_{X/S}$, and then the sequence is viewed as sequence
of $\cO_X$-modules via the left $\cO_X$-module structure.
The canonical splitting of (\ref{tae}) for the right $\cO_X$-module structure
induces a canonical $\cO_S$-linear splitting of $\cjet^1_{X/S}(F)$.
We obtain a canonical direct sum decomposition
\begin{equation}\label{canosplit}
P^1_{X/S}(F)=F\oplus ( F\otimes\Omega_{X/S}^1 )
\end{equation}
of $\cO_S$-modules.
We use squared brackets $[\,,\,]$ when we refer to this decomposition.
A straightforward calculation shows that, in terms of this decomposition, the 
left $\cO_X$-module structure  of $P^1_{X/S}(F)$
 is given by
\begin{equation}\label{modstructure}
\lambda\cdot[f, \omega]=[\lambda\cdot f, \lambda\cdot \omega-f\otimes d\lambda]
\end{equation}
for local sections $\lambda$ of $\cO_X$, $f$ of $F$, and $\omega$ of 
$F\otimes \Omega_{X/S}^1$.
It follows that there is a one-to-one correspondence
\[
\genfrac{\{}{\}}{0pt}{}{\mbox{$\cO_X$-linear splittings 
}}{\mbox{$s\colon \, F\to P^1_{X/S}(F)$  of $\cjet^1_{X/S}(F)$}}
\longleftrightarrow 
\genfrac{\{}{\}}{0pt}{}{\mbox{algebraic resp. holomorphic}}{
\mbox{ connections $\nabla\colon\, F\to F\otimes\Omega^1_{X/S} $}}.
\]
Under this correspondence, a connection $\nabla$ corresponds to the
splitting $s_\nabla$ of $\cjet^1_{X/S}(F)$ given by the formula
\begin{equation}\label{sectgbconn}
s_\nabla:F\longrightarrow P^1_{X/S}(F)=F\oplus (F \otimes
\Omega_{X/S}^1)\,\,,\,\,
f\longmapsto \bigl[f,-\nabla(f)\bigr].
\end{equation}
For instance, the ``trivial" connection $\nabla:=d$ on $E=\cO_X$  is associated to the canonical left $\cO_X$-linear splitting of (\ref{tae}).

\subsubsection{}
The extension $\cjet^1_{X/S}(F)$ is called the \emph{extension given 
by the $1$-jets} or \emph{principal parts of first order associated with $F$}.
We denote the class of $\cjet^1_{X/S}(F)$ in 
${\rm Ext}_{\cO_X}^1(F,F\otimes \Omega_{X/S}^1)$
by $\jet^1_{X/S}(F)$ and abbreviate $\jet(F)=\jet^1_{X/S}(F)$ if $X/S$ is 
clear from the context.
We have followed in (\ref{kaehlerid}), (\ref{definejet}), and (\ref{sectgbconn})
the conventions
fixed in \cite[16.7]{EGAIV4}, \cite[III. (1.2.6.2)]{illusie71}, and 
\cite[(2.3.4)]{deligne70}.

\subsubsection{}
We recall from \cite[Propositions 6, 7 and 8]{atiyah57} that 
the assignment
\begin{eqnarray*}
\{\mbox{vector bundles on $X$}\}
& \longrightarrow & \{\mbox{short exact sequences of ${\cO_X}$-modules}\}\\
F&\longmapsto & {\cjet}_{X/S}^1(F)
\end{eqnarray*}
defines an additive, exact functor.
Furthermore ${\cjet}^1_{X/S}(F)$ is a short exact sequence of vector bundles 
if $\pi$ is smooth.

The following Lemma is a slight refinement of \cite[Proposition 10]{atiyah57}.

\begin{lemma}\label{attenprodcom}
Let $E$ and $F$ denote vector bundles on $X$. 

i) Let 
\[
B=\frac{\text{\rm Ker}\bigl( p_E \otimes \id_F- \id_E \otimes p_F : 
P^1_{X/S} (E) \otimes F \oplus E\otimes P^1_{X/S} (F)\to 
E\otimes F\bigr)}{\text{\rm Im}\bigl( (i_E\otimes \id_F, - \id_E\otimes i_F): 
E \otimes F \otimes \Omega^1_{X/S}\to P^1_{X/S}(E) \otimes F \oplus 
E\otimes P^1_{X/S} (F)\bigr)}.
\] 
denote the Baer sum 
of the extensions $\cjet_{X/S}^1 (E)\otimes  F$ and $E\otimes  \cjet_{X/S}^1(F)$. 
There exists a canonical isomorphism
\begin{equation}\label{compis}
\varphi:P_{X/S}^1(E\otimes  F)\longrightarrow B
\end{equation}
which fits into a commutative diagram
\[
\begin{array}{ccccccccc}
 0&\longrightarrow &  E\otimes  F \otimes\Omega_{X/S}^1 &\longrightarrow 
&P_{X/S}^1(E\otimes  F) & \longrightarrow &E\otimes  F & \longrightarrow &0\\
 & &|| & & \,\,\,\,\,\downarrow {\scriptstyle \varphi}& & ||& & \\
 0 &\longrightarrow &E\otimes F \otimes  \Omega_{X/S}^1&\longrightarrow 
&B & \longrightarrow &E\otimes  F & \longrightarrow &0.
\end{array}
\]
Consequently we have
\[
\jet_{X/S}^1(E\otimes  F)
=\jet_{X/S}^1(E)\otimes  F+E\otimes  \jet_{X/S}^1(F)
\]
in ${\rm Ext}^1_{\cO_X}(E\otimes F,E\otimes F\otimes  \Omega_{X/S}^1)$.

ii)
Let $\nabla_E$ and $\nabla_F$ denote connections on $E$ and $F$.
We equip the tensor product $E\otimes F$ with the product connection
\begin{equation}\label{prodconn}
\nabla_{E\otimes  F}
=\nabla_E\otimes  {\rm id}_F+{\rm id}_E\otimes  \nabla_F.
\end{equation}
The connections $\nabla_E$, $\nabla_F$, and $\nabla_{E\otimes  F}$ induce sections
$s_E$,$s_F$, and $s_{E\otimes  F}$ of $\cjet_{X/S}^1(E)$, $\cjet_{X/S}^1(F)$, 
and $\cjet_{X/S}^1(E\otimes F)$  respectively. We have
\[
\varphi\circ s_{E\otimes F}=(s_E\otimes {\rm id}_F,{\rm id}_E\otimes s_F)
\]
where the notation on the right hand side refers to the description
of the Baer sum given above.
\end{lemma}

\proof 
i) 
Let $IM={\rm Im}(i_E\otimes \id_F, - \id_E\otimes i_F)$.
Recall that
\[
P^1_{X/S}(E\otimes F)
=(E\otimes F)\oplus (E\otimes F\otimes \Omega_{X/S}^1).
\]
There exists a unique $\cO_S$-linear map (\ref{compis}) which satisfies
\begin{eqnarray*}
\varphi\bigl([e_0\otimes f_0,e_1\otimes f_1\otimes \alpha]\bigr)
&=&\bigl([e_0, 0]\otimes f_0+[0,e_1\otimes \alpha]\otimes f_1\bigr)
\oplus\bigl(e_0\otimes [f_0, 0]\bigr)\, {\rm mod }\, IM\\
&=&\bigl([e_0, 0]\otimes f_0\bigr)\oplus \bigl(e_0\otimes [f_0, 0]
+e_1\otimes [0, f_1\otimes \alpha]\bigr)\,{\rm mod}\, IM
\end{eqnarray*}
for local sections $e_0,e_1$ of $E$, $f_0, f_1$ of $F$ and $\alpha$ of 
$\Omega_{X/S}^1$.
It is straightforward to check that $\varphi$ is well defined  
and makes our diagram commutative.
It remains to show that $\varphi$ is also $\cO_X$-linear.
This follows from
\begin{eqnarray*}
\varphi\bigl(\lambda\cdot[e_0\otimes f_0, 0]\bigr)
&=& \varphi\bigl([\lambda \cdot e_0\otimes f_0, -e_0\otimes f_0\otimes d\lambda]\bigr)\\
&=& \bigl([\lambda \cdot e_0, 0]\otimes f_0 - [0, e_0\otimes d\lambda]\otimes f_0\bigr)
\oplus \bigl( \lambda \cdot e_0\otimes [f_0, 0]\bigr)\,{\rm mod}\,IM\\
&=&\lambda \cdot \varphi\bigl([e_0\otimes f_0, 0]\bigr)
\end{eqnarray*}
as $\varphi$ induces the identity on $\Omega_{X/S}^1\otimes E\otimes F$.

ii) 
For local sections $e$ of $E$ and $f$ of $F$, we get
\begin{eqnarray*}
\varphi \circ s_{E\otimes F}(e\otimes f)
&=& \bigl( [e, -\nabla e]\otimes f\bigr)\oplus \bigl(e\otimes [f,- \nabla f]\bigr)\mod\,IM\\
&=& (s_E\otimes {\rm id}_F,{\rm id}_E\otimes s_F)(e\otimes f)
\end{eqnarray*}
which proves ii).
\qed

\begin{corollary}\label{atiyahdual}
Let $E$ be a vector bundle on $X$ and denote 
$$j_E:\cO_X\rightarrow E\otimes E^\vee\simeq\End(E)$$
the canonical morphism
of vector bundles which maps $1$ to $\id_E$. 
The Baer sum of $\cjet_{X/S}^1(E)\otimes E^\vee$ and $E\otimes \cjet_{X/S}^1(E^\vee)$ 
is canonically isomorphic
to $\cjet_{X/S}^1(E\otimes E^\vee)$. The pullback $\cjet_{X/S}^1(E\otimes E^\vee)\circ j_E$ of $\cjet^1_{X/S}(E\otimes E^\vee)$ along $j_E$ --- defined as the upper 
extension in the commutative diagram
\begin{equation}\label{pulldiag1}
\begin{array}{ccccccccc}
0& \rightarrow &E\otimes E^\vee \otimes \Omega_{X/S}^1 & \rightarrow & Q & \rightarrow & \cO_X & \rightarrow & 0\\& 
&\| & & \downarrow & & \,\,\,\downarrow {\scriptstyle j_E}\\
0& \rightarrow &E\otimes E^\vee \otimes \Omega_{X/S}^1 & \rightarrow & P^1_{X/S}(E\otimes E^\vee) & \rightarrow &E\otimes E^\vee & \rightarrow & 0
\end{array}
\end{equation}
whose righthand square is cartesian \emph{(compare \cite[App. A.4.2]{bostkuennemann1})} ---
 admits a canonical
splitting.
\end{corollary}

\proof
The first statement follows from Lemma \ref{attenprodcom}.
The map $j_E$ induces by functoriality a morphism from 
$\cjet_{X/S}^1(\cO_X)$ to $\cjet_{X/S}^1(E\otimes E^\vee)$.
Since the righthand side in (\ref{pulldiag1}) is cartesian, we obtain a
commutative diagram
\begin{equation}\label{pulldiag2}
\begin{array}{ccccccccc}
0& \longrightarrow &\Omega_{X/S}^1 & \longrightarrow & P_{X/S}^1(\cO_X) &
\longrightarrow & \cO_X & \longrightarrow & 0\\
& &\,\,\,\,\,\,\,\,\,\,\,\,\,\,\,\,\,
\downarrow {\scriptstyle j_E\otimes {\rm id}_{\Omega_{X/S}^1}} & & \,\,\,\,\, \downarrow {\scriptstyle \varphi}
& & \| \,\,\,\\
0& \longrightarrow &E\otimes E^\vee \otimes \Omega_{X/S}^1 & \longrightarrow &
Q & \longrightarrow &\cO_X & \longrightarrow & 0.
\end{array}
\end{equation}
The canonical splitting $s_d$ of 
of $\cjet_{X/S}^1(\cO_X)$ (that correspond to the connection $d$ on $\cO_X$) induces via (\ref{pulldiag2}) the requested
canonical splitting $\varphi \circ s_d$ of $\cjet_{X/S}^1(E\otimes E^\vee)\circ j_E$.
\qed

\begin{lemma}\label{pullbackconnection}
Consider a commutative diagram
\[
\begin{array}{ccc}
\tilde X &\stackrel{f}{\longrightarrow} & X\\
\,\,\,\,\,\downarrow {\scriptstyle \tilde \pi} & &\,\,\,\,\,\,\downarrow {\scriptstyle \pi}\\
\tilde S &\stackrel{g}{\longrightarrow} & S
\end{array}
\]
in the category of locally ringed spaces where $\tilde \pi$ and  $\pi$ 
are morphisms as in situation \ref{atgen}, a) or b).  
Let $E$ be a vector bundle on $X$ and denote  by $f^*$
the canonical map
$f^*\Omega_{X/S}^1\to \Omega_{ \tilde X/\tilde S}^1$.

i) There exists a canonical $\cO_{\tilde X}$-linear map
\[
\phi:f^*P^1_{X/S}(E)\longrightarrow P^1_{\tilde X/\tilde S}(f^*E)
\]
which makes the diagram
\begin{equation}\label{commbcd}
\begin{array}{ccccccccc}
 0&\longrightarrow &f^*E\otimes_{\cO_{\tilde X}} f^*\Omega_{X/S}^1 &\longrightarrow 
&f^*P_{X/S}^1(E) & \longrightarrow &f^*E & \longrightarrow &0\\
 & &\downarrow {\scriptstyle {\rm id}_{f^*E}\otimes f^*} & 
& \,\,\,\,\,\downarrow {\scriptstyle \phi}& & ||& & \\
 0 &\longrightarrow & f^*E\otimes_{\cO_{\tilde X}}\Omega_{\tilde X/\tilde S}^1 &\longrightarrow 
& P_{\tilde X/\tilde S}^1(f^*E)& \longrightarrow &f^*E & \longrightarrow &0.
\end{array}
\end{equation}
commutative.
Consequently we have
\[
({\rm id}_{f^*E}\otimes f^*)\circ {\jet}_{X/S}^1(E)={\jet}_{\tilde X/\tilde S}^1(f^*E)
\]
in ${\rm Ext}_{\cO_{\tilde X}}^1(f^*E,f^*E\otimes_{\cO_{\tilde X}} \Omega_{\tilde X/\tilde S}^1)$.

ii) A connection $\nabla_E$ on $E$ induces a splitting $s_E$ of 
${\cjet}_{X/S}^1(E)$. The splitting 
\[
s_{f^*E}:=\phi\circ f^*(s_E)
\]
induces a connection $f^*\nabla_E$ on $f^*E$
which is uniquely determined by
\begin{equation}\label{pullbdef}
(f^*\nabla_E)(f^*s)=f^*(\nabla_E\,s):= ({\rm id}_{f^\ast E} \otimes f^\ast)(f^{-1}(\nabla_E s))
\end{equation}
for local sections $s$ of $E$.
\end{lemma}

Notice that the case where $\tilde \pi$ is as in situation  \ref{atgen}, b)
 and $\pi$ as in situation  \ref{atgen}, a) is allowed.
\proof
i) Observe that the upper sequence in (\ref{commbcd}) is exact as $E$ is
locally free.
Recall that 
\begin{equation}\label{decjet1}
f^*P^1_{X/S}(E)=\bigl[f^{-1}E\oplus f^{-1}(E\otimes_{\cO_X} \Omega_{X/S}^1)\bigr]
\otimes_{f^{-1}\cO_X}\cO_{\tilde X}
\end{equation}
and
\begin{equation}\label{decjet2}
P^1_{\tilde X/\tilde S}(f^*E)=f^*E\oplus f^*E\otimes_{\cO_{\tilde X}}\Omega_{\tilde X/\tilde S}^1.
\end{equation}
By the very definitions of $f^\ast E$ and $f^\ast (E \otimes_{\cO_X} \Omega^1_{X/S}),$ we have $f^{-1}\cO_X$-linear canonical maps
\[
f^{-1}E\longrightarrow f^*E
\]
and
\[
f^{-1}(E\otimes_{\cO_X} \Omega_{X/S}^1)
\rightarrow f^*(E\otimes_{\cO_X} \Omega_{X/S}^1)
\stackrel{\sim}{\rightarrow} f^*E\otimes_{\cO_{\tilde X}}f^*\Omega_{X/S}^1
\stackrel{{\rm id}_{f^*E}\otimes f^*}{\longrightarrow} 
f^*E\otimes_{\cO_{\tilde X}}\Omega_{\tilde X/\tilde S}^1.
\]
The direct sum of these maps induces 
a $g^{-1}\cO_S$-linear morphism 
\[
\bigl[f^{-1}E\oplus f^{-1}(E\otimes_{\cO_X} \Omega_{X/S}^1)\bigr]
\longrightarrow f^*E\oplus f^*E\otimes_{\cO_{\tilde X}}\Omega_{{\tilde X}/\tilde S}^1.
\]
It is straightforward to check that this morphism is $f^{-1}\cO_X$-linear
for the module structure given by formula (\ref{modstructure}).
Via (\ref{decjet1}) and (\ref{decjet2}), we obtain the desired morphism $\phi$
which fits by construction in the diagram (\ref{commbcd}).

ii) is a straightforward consequence of the construction of $\phi$ in the proof
of i).
\qed

\subsection{Cotangent complex and Atiyah class}
In situation \ref{atgen}, a) resp. b), the cotangent complex $\mathbb{L}_{X/S}^\cdot$ is
constructed in \cite[II.1.2]{illusie71} resp. \cite[2.38]{buchweitzflenner03} as 
an object in the 
derived category $D(\cO_X{\rm -mod})$ of $\cO_X$-modules. 
Consider $\Omega_{X/S}^1$ as a complex concentrated in degree zero.
The cotangent complex 
$\mathbb{L}_{X/S}^\cdot$ comes with a natural morphism
\begin{equation}\label{aucot}
\mathbb{L}_{X/S}^\cdot\longrightarrow \Omega_{X/S}^1
\end{equation}
in $D(\cO_X{\rm -mod})$ which is a quasi-isomorphism if $X$ is smooth over $S$.
Given a vector bundle $E$ over $X$, the {\it Atiyah class} of $E$ is defined 
in \cite[IV.2.3]{illusie71} resp.  \cite[\S 3]{buchweitzflenner03} 
as an element 
\[
\at_{X/S}(E)\in {\rm Ext}_{\cO_X}^1(E,E\otimes^{\mathbb{L}} \mathbb{L}_{X/S}^\cdot)
={\rm Hom}_{D(\cO_X{\rm -mod})}(E,E\otimes^{\mathbb{L}} \mathbb{L}_{X/S}^\cdot[1]).
\]

If $X\stackrel{\pi}{\rightarrow} S$ is a morphism of schemes, the Atiyah
class of Illusie maps  under the morphism 
induced by (\ref{aucot}) to the class (compare \cite[Cor. IV.2.3.7.4]{illusie71})
\[
\jet^1_{X/S}(E)\in {\rm Ext}_{\cO_X}^1(E,E\otimes \Omega_{X/S}^1).
\]
Furthermore, according to \cite[ Prop. II.1.2.4.2]{illusie71},  (\ref{aucot}) induces an isomorphism
\begin{equation}\label{encoreuniso}
H_0(\mathbb{L}^\cdot_{X/S}) \stackrel{\sim}{\longrightarrow} \Omega^1_{X/S}.
\end{equation}

If $X\stackrel{\pi}{\rightarrow} S$ is a smooth morphism of complex 
analytic spaces, the Atiyah class 
of Buchweitz and Flenner maps  under the morphism induced by (\ref{aucot}) to 
the {\it opposite} class of $\jet^1_{X/S}(E)$ (\cite[3.27]{buchweitzflenner03}).

If the canonical morphism (\ref{aucot}) is a 
quasi-isomorphism, we call $\cjet_{X/S}^1(E) $ the {\it Atiyah extension
associated with $E$}
and denote it by $\cat_{X/S}(E)$.

The associated extension class ${\rm at}_{X/S}(F)$
equals the opposite of the Atiyah classes ${\rm At}(F)$ in 
\cite{buchweitzflenner03} and $b(F)$ in 
\cite[Section 4]{atiyah57}.
It coincides with the Atiyah class defined in \cite{angenioletal89}.
Compare also \cite[2.4 and Rem. 3.17]{buchweitzflenner03} for a discussion 
of signs related to the Atiyah class.

The following Lemma implies in particular that (\ref{aucot}) 
is a quasi-isomorphism in the situations considered in the next sections.

\begin{lemma}\label{augqis}
Let $\pi:X\to S$ be a locally complete intersection  (l.c.i.) 
morphism of schemes such that $X$ is integral and $\pi$ is generically smooth, in the sense that the
smooth locus of $\pi$ is dense in $X$.
Then the morphism (\ref{aucot})
is a quasi-isomorphism. 
\end{lemma}

\proof
It is sufficient to show our claim locally on $X$ as the formation 
of (\ref{aucot}) is compatible with restrictions to open subsets.
Hence we may assume that $\pi$ admits a factorization 
\[
X\stackrel{j}{\longrightarrow}Q\stackrel{q}{\longrightarrow} S
\]
where $j$ is a regular immersion defined by some regular ideal sheaf ${ J}$ 
and $q$ is smooth. We obtain an exact sequence 
\begin{equation}\label{exactfact}
0\longrightarrow
{J}/{J}^2\stackrel{\phi}{\longrightarrow}
j^*\Omega_{Q/S}^1\stackrel{\psi}{\longrightarrow}
\Omega_{X/S}^1\longrightarrow 0.
\end{equation}
This is well known up to the injectivity of $\phi$ which holds
as $\phi$ is a morphism of locally free sheaves which is injective over the
smooth locus of $\pi$. The complex
\[
{J}/{J}^2\stackrel{\phi}{\longrightarrow}
j^*\Omega_{Q/S}^1
\]
situated in degrees minus one and zero is a cotangent complex for $f$ 
by \cite[Cor. III.3.2.7]{illusie71}.
Therefore it follows from the exactness of (\ref{exactfact}) on the left and the isomorphism (\ref{encoreuniso}) that  (\ref{aucot}) is in fact a 
quasi-isomorphism.
\qed

\subsection{$\cC^\infty$-connections compatible with the holomorphic structure}\label{cocost}
Let $E$ denote a holomorphic vector bundle on a complex manifold $X$. 
Recall that a $\mathcal{C}^\infty$-connection
\[
\nabla\colon\, A^0(X,E)\longrightarrow A^1(X,E)
\] 
on $E$
is called {\it compatible with the holomorphic structure} if its $(0,1)$-part 
coincides with the Dolbeault
operator, i.e.
$\nabla^{0,1} = \overline\partial_E$.
Consider the Atiyah extension 
associated with $E$
\[
\cat_X(E)\colon 0\longrightarrow E\otimes \Omega^1_X
\stackrel{i_E}{\longrightarrow} P_{X/\C}^1(E) 
\stackrel{p_E}{\longrightarrow} E\longrightarrow 0.
\] 
In the same way as before, we obtain a one-to-one correspondence 
\[
\nabla\longleftrightarrow s_{\nabla^{1,0}}
\]
between $\mathcal{C}^\infty$-connections on the vector 
bundle $E$ which are compatible with the holomorphic structure and
$\mathcal{C}^\infty$-splittings
\begin{equation}\label{sign10}
s_{\nabla^{1,0}}:E\longrightarrow P_{X/\C}^1(E)\,,\,\,
f\longmapsto [f,-\nabla^{1,0}(f)]
\end{equation}
of the extension $\cat_X(E)$. 

It is straightforward to check that this correspondence satisfies compatibility properties with tensor operations and pull back similar to the ones established in Lemma \ref{attenprodcom}, Corollary \ref{atiyahdual}, and Lemma \ref{pullbackconnection} above.

The one-to-one correspondence described above extends in a straightforward
way to the relative situation where $X/S$ is a holomorphic family of
complex manifolds. We leave the details of this construction to the 
interested reader.

Assume that $E$ carries a hermitian metric $h$. 
A $\mathcal{C}^\infty$-connection $\nabla$ on $\overline E=(E,h)$ is
called {\it unitary} if and only if it satisfies
\[
dh(s,t) = h(\bigtriangledown s,t) + h(s,\bigtriangledown
t)\,\,\,\,\mbox{ for all }s,t\in A^0(X,E).
\]
Recall that a hermitian holomorphic vector bundle $\overline E=(E,h)$ carries a unique
unitary connection $\nabla_{\ol{E}}$ which is compatible with the holomorphic structure
(\cite{chern46}, \cite{nakano55}; see also \cite[Ch. 0.5]{griffithsharris78} or \cite[Sect. II.2]{wells80}; this connection is sometimes called the \emph{Chern connection} of $(E,h)$). Moreover the assignement $\overline{E} \mapsto \nabla_{\overline{E}}$ is compatible with direct sums, tensor products, duals and pull-backs.

\begin{lemma}\label{curvaturesff}
Let $\overline E=(E,h)$ be a hermitian holomorphic vector bundle on $X$.
Let $\nabla=\nabla_{\ol{E}}$ denote the unitary 
$\mathcal{C}^\infty$-connection on $E$ which is compatible
with the complex structure. 
The curvature form
\[
\nabla^2\in A^{1,1}\bigl(X,{\iend}(E)\bigr)
\]
and the second fundamental form 
\[
\alpha\in A^{0,1}\bigl(X,{\iend}(E)\otimes \Omega_X^1\bigr)
\]
associated with ${\cat}_{X}(E)$ and its 
$\mathcal{C}^\infty$-splitting $s_{\nabla^{1,0}}$ as in 
\cite[A.5.2]{bostkuennemann1} satisfy
\begin{equation}\label{coseff}
\alpha=-\nabla^2
\end{equation}
where we read the canonical isomorphism 
\[
A^{1,1}\bigl(X,{\iend}(E)\bigr)\stackrel{\sim}{\rightarrow}
A^{0,1}\bigl(X,{\iend}(E)\otimes \Omega^1_X\bigr)\,,\,\,
f\otimes (\alpha\wedge\beta)\mapsto (f\otimes \alpha)\wedge\beta.
\]
(compare \cite[1.1.5]{bostkuennemann1}) as an identification.
\end{lemma}

\proof
Recall from \cite[A.5.2]{bostkuennemann1} that $\alpha$ is determined by
\[
\overline \partial_{P_{X/\C}^1(E)\otimes E^\vee}
\bigl(s_{\nabla^{1,0}}\bigr)=(i_E\otimes {\rm id}_{A^{0,1}_X})(\alpha).
\]
It is sufficient to verify (\ref{coseff}) locally on $X$.
Hence we may assume that $E$ admits a holomorphic frame. 
We describe $\nabla$ and $\nabla^2$ with respect to this frame
by the connection matrix $\theta$ and the curvature matrix $\Theta$.
Following the conventions in \cite[Ch. III]{wells80}, we have
\[
\Theta_{ik}=d\theta_{ik}+\sum_{j}\theta_{ij}\wedge\theta_{jk}.
\]
The connection matrix $\theta$ has type $(1,0)$ 
and the curvature matrix $\Theta$ has type $(1,1)$ by {\it loc. cit.}
Hence the equality above becomes
\begin{equation}\label{connpartdolb}
\Theta=\overline{\partial}\theta.
\end{equation}
Let $\tilde \nabla$ denote the connection on $E$ whose connection matrix is zero.
The associated splitting $s_{\tilde \nabla^{1,0}}$ of $\cat_X(E)$ is holomorphic.
Hence (\ref{sectgbconn}) and (\ref{connpartdolb}) give
\[
\overline\partial_{P_{X/\C}^1(E)\otimes E^\vee }
\bigl(s_{\nabla^{1,0}}\bigr)=\overline\partial_{ P_{X/\C}^1(E)\otimes E^\vee}
\bigl( s_{\nabla^{1,0}}-s_{\tilde \nabla^{1,0}}\bigr)
=-\overline{\partial}(\theta)=-\Theta=-\nabla^2.
\]
\qed

\section{The arithmetic Atiyah class of a vector bundle with connection}\label{atiyah1}
In this section we fix an arithmetic ring $R=(R,\Sigma,F_\infty)$ in the sense of \cite[3.1.1]{gilletsoule90}. We denote $K$ the fraction field of $R$, 
and we let $S:=\Spec R$.


\subsection{Definition and basic properties}
Let $X$ be an integral arithmetic scheme over $R$ (in the sense of \cite{gilletsoule90}, or \cite[1.1]{bostkuennemann1}) with
a flat, l.c.i. structural morphism $\pi:X\to S$. Recall that the generic fiber $X_K$ of $X$ is smooth (by the very definition of an arithmetic scheme in \emph{loc. cit.}), and observe that $\pi$ satisfies the assumptions in Lemma \ref{augqis}.

Let $E$ be a vector bundle on $X$. 
We consider the commutative square
\[
\begin{array}{ccc}
(X_\Sigma(\C),\cO_{X_\Sigma}^{\rm hol}) &\stackrel{j}{\longrightarrow} & (X,\cO_X)\\
\downarrow {\scriptstyle \pi_\C} & &\downarrow {\scriptstyle \pi}\\
(S_\Sigma(\C),\cO_{S_\Sigma}^{\rm hol}) &\stackrel{j_0}{\longrightarrow} & (S,\cO_S).
\end{array}
\]
Lemma \ref{pullbackconnection} implies that the formation of the 
Atiyah extension of $E$ is
compatible with base change with respect to this diagram.
More precisely, we have a canonical  analytification isomorphism
\[
P^1_{X/S}(E)^{\rm hol}_\C
\stackrel{\sim}{\longrightarrow} P^1_{X_\Sigma(\C)/S_\Sigma(\C)}(E^{\rm hol}_\C)
\]
where we put $F_\C^{\rm hol}=j^*F$ for every $\cO_X$-module $F$.

\subsubsection{}
We have seen in \ref{cocost} that there is a one-to-one correspondence between
$\mathcal{C}^\infty$-connections
\[
\nabla\colon\, A^0\bigl(X_\Sigma(\mathbb{C}),E_\mathbb{C}\bigr)\to 
A^1\bigl(X_\Sigma(\mathbb{C}),E_\mathbb{C}\bigr)
\]
which are compatible with the holomorphic structure and commute with 
the action of $F_\infty$, and sections
\[
s_\nabla\colon\, E_\mathbb{C}\to P^1_{X/S}(E)_\mathbb{C}
\]
such that $(\cat_{X/{S}}E,s_\nabla)$ is an arithmetic extension. 
This correspondence allows us to associate 
its {\it arithmetic Atiyah extension} $(\cat_{X/S}E, s_\nabla)$ and
its {\it arithmetic Atiyah class}
\[
\widehat{\rm at}_{X/S}(E,\nabla)\in \widehat{\rm
  Ext}^1_X(E,E\otimes \Omega_{X/S}^1)
\]
to any  vector bundle $E$ on
$X$ equipped with an $F_\infty$-invariant $\cC^\infty$-connection
$\nabla$ on $E_\C$ that is compatible with the holomorphic structure.

If $\overline{E}$ is a hermitian vector bundle over $X$, we obtain
the {\it arithmetic Atiyah extension $(\cat_{X/S}E, s_{\nabla_{\overline{E}}})$ of $\overline{E}$} and
its {\it arithmetic Atiyah class}
\[
\widehat{\rm at}_{X/S}(\overline{E}):=
\widehat{\rm at}_{X/S}(E,\nabla_{\overline{E}})\in \widehat{\rm
  Ext}^1_X(E,E\otimes \Omega_{X/S}^1),
\] 
where $\nabla_{\overline{E}}$ denotes the 
unitary connection on $E_\C^{\rm hol}$ over $X_{\Sigma}(\C)$ that is compatible with the 
complex structure.
As a direct consequence of this definition and Lemma \ref{curvaturesff}, we get
a formula for the ``second fundamental form" 
(compare the introduction and \cite[2.3.1]{bostkuennemann1})
\[
\Psi \bigl(\widehat{\text{at}}_{X/S}(\overline E)\bigr)\in
A^{0,1} 
\bigl(X_\mathbb{R},\,{\iend}(E)\otimes \Omega^1_{X/S}\bigr).
\]
Namely 
\begin{equation}
\Psi \bigl(\widehat{\at}_{X/S}(\overline E)\bigr)= - R_{\overline{E}},
\end{equation}
under the canonical identification
\[
A^{1,1} \bigl(X_\mathbb{R},\,{\iend}(E)\bigr)= A^{0,1} 
\bigl(X_\mathbb{R},\, {\iend}(E)\otimes\Omega^1_{X/S} \bigr),
\]
where $R_{\overline{E}}:=\nabla_{\overline{E}}^2$ denotes 
the curvature of $\overline E$.

In particular, when $\overline{E}$ is a hermitian line bundle over $X,$ 
\begin{equation}\label{formfirstcc}
\frac{1}{2\pi i}\,\Psi \bigl(\widehat{\at}_{X/S}(\overline E)\bigr)
= - \frac{1}{2\pi i}\, R_{\overline{E}} =: c_1(\overline E)
\end{equation}
is the first Chern form of $\overline{E}$.

We collect basic properties of the arithmetic Atiyah class.

\begin{proposition}\label{tpla}

i)
Let $(E,\nabla_E)$ and $(F,\nabla_F)$ be vector bundles on
$X$ equipped with $F_\infty$-invariant $\mathcal{C}^\infty$-connections 
compatible with the holomorphic structure.
We equip the tensor product $E\otimes F$ with the product connection.
Then the equality
\[
\widehat{\rm at}_{X/S}(E\otimes F,\nabla_{E\otimes F})=\widehat{\rm
  at}_{X/S}(E,\nabla_{E})\otimes F + E\otimes \widehat{\rm at}_{X/S}( F,\nabla_{ F})
\]
holds in $\widehat{\rm Ext}_X^1(E\otimes F,
E\otimes F\otimes\Omega_{X/S}^1)$.

ii)
Let $\ol{E}$ and $\ol{F}$ be hermitian vector bundles on $X$, 
and $\ol{E}\otimes\ol{F}$ their tensor product equipped with 
the product hermitian metric.
Then the equality
\[
\widehat{\rm at}_{X/S}(\ol{E}\otimes \ol{F})=\widehat{\rm
  at}_{X/S}(\ol{E})\otimes F + E\otimes \widehat{\rm at}_{X/S}(\ol{ F})
\]
holds in $\widehat{\rm Ext}_X^1(E\otimes F,
E\otimes F\otimes\Omega_{X/S}^1)$.

iii)
Let $\ol{E}$ be a hermitian vector bundles on $X$, and 
$\ol{E}^\vee$ the dual hermitian vector bundle. Then the equality 
\begin{equation}\label{atardual}
\widehat{\rm at}_{X/S}(\ol{E})=
-\widehat{\rm at}_{X/S}(\ol{E}^\vee)
\end{equation}
holds in
\begin{multline}\label{isos}
\widehat{\rm Ext}_X^1(E,E\otimes \Omega_{X/S}^1) \simeq 
\widehat{\rm Ext}_X^1(\cO_X, E^\vee \otimes E\otimes \Omega_{X/S}^1) \\
\simeq \widehat{\rm Ext}_X^1(\cO_X,(E^\vee)^\vee \otimes E^\vee \otimes \Omega_{X/S}^1)
\simeq \widehat{\rm Ext}_X^1(E^\vee , E^\vee \otimes \Omega_{X/S}^1),
\end{multline}
where the first and last isomorphisms  in (\ref{isos}) are the canonical isomorphisms 
in \cite[2.4.6]{bostkuennemann1}, and the second one is deduced from the 
isomorphism $E^\vee \otimes E \simeq E\otimes  E^\vee$ exchanging the two 
factors and the canonical biduality isomorphism $E \simeq (E^\vee)^\vee.$

iv)
Let $f:X\rightarrow Y$ be a morphism of integral arithmetic schemes which are
generically smooth l.c.i. over $S$.
Let $(E,\nabla_E)$ be a vector bundle on $Y$ with 
$F_\infty$-invariant $\mathcal{C}^\infty$-connection which is
compatible with the holomorphic structure.
The canonical map 
$f^*:f^*\Omega_{Y/S}^1\rightarrow \Omega_{X/S}^1$ 
induces  a homomorphism
\[
\widehat{\rm Ext}_X^1(f^*E,f^*E\otimes f^*\Omega_{Y/S}^1)\longrightarrow 
\widehat{\rm Ext}_X^1(f^*E,f^*E\otimes \Omega_{X/S}^1)
\]
by pushout along ${\rm id}_{f^*E}\otimes f^*$.
We still denote the image of $f^*\widehat{\rm at}_{Y/S}(E,\nabla_E)$ under this map
by $f^*\widehat{\rm at}_{Y/S}(E,\nabla_E)$ and equip
$f^*E_\C^{\rm hol}$ with the connection $f^*\nabla_{E}$ described in (\ref{pullbdef}).
Then we have the equality
\[
f^*\widehat{\rm at}_{Y/S}(E,\nabla_E)=\widehat{\rm at}_{X/S}(f^*E,f^*\nabla_{E}).
\]
in $\widehat{\rm Ext}_X^1(f^*E,f^*E\otimes \Omega_{X/S}^1)$.

v) 
Let $f:X\rightarrow Y$ be a morphism of integral arithmetic schemes which are
generically smooth l.c.i. over $S$.
Let $\overline{E}$ denote a hermitian vector bundle on $Y$, and $f^*\ol{E}$ its pull-back on $X.$
Then the inverse image $f^*\widehat{\rm at}_{Y/S}(\overline{E})$ may be defined in $\widehat{\rm Ext}_X^1(f^*E,f^*E\otimes \Omega_{X/S}^1)$ as in iv) and satisfies \[
f^*\widehat{\rm at}_{Y/S}(\overline{E})=\widehat{\rm at}_{X/S}(f^*\overline{E}).
\]

\end{proposition}

\proof
Assertion i) follows from Lemma \ref{attenprodcom} and its variant for $\cC^\infty$-connections compatible with the holomorphic structure, and assertion
ii) is a direct consequence of i) and of the fact that the Chern connection of a tensor product of hermitian vector bundles coincides with the tensor product of their Chern connections. To establish iii), observe that 
Corollary \ref{atiyahdual} and  the compatibility of the canonical splitting 
given there 
with holomorphic and hermitian structures leads to the equality
\[
(E^\vee \otimes \widehat{\rm at}_{X/S}(\ol{E}))\circ j_E=
-(\widehat{\rm at}_{X/S}(\ol{E}^\vee)\otimes E)\circ j_E
\]
in $\widehat{\rm Ext}_X^1(\cO_X,{\iend}(E)\otimes \Omega_{X/S}^1)$
where $.\circ j_E$ denotes the pushout along $j_E$. 
Equality (\ref{atardual}) then follows from the very definitions of the isomorphisms in (\ref{isos}) in \cite[Prop. 2.4.6]{bostkuennemann1}.
iv) and v) follow from \ref{pullbackconnection}. 
\qed

Let $\overline{E}$ be a hermitian line bundle on $X$. 
We give a cocycle description of $\widehat{\rm at}(\overline{E})$ 
based on the description of arithmetic extension
groups by \v{C}ech cocycles given in Appendix A.

\begin{proposition}\label{cocycdescrat}
Let $\overline{E}=(E,h)$ be a hermitian vector bundle of rank $n$
on $X$. 
Choose an affine, open cover $\mathcal{U}=(U_i)_{i\in I}$
of $X$ such that $E$ admits a frame
\[
f_i:\cO_{U_i}^n\stackrel{\sim}{\longrightarrow}E|_{U_i}
\]
over each $U_i$.
For $i\in I$, we define
$$h_i:=h(f_{i,\C},f_{i,\C})
=\bigl(h(f_{i,\C}(e_l),f_{i,\C}(e_k))\bigr)_{1\leq k,l\leq n} 
\in 
 {\rm Mat}_n
\bigl(\mathcal{C}^\infty(U_{i,\Sigma}(\C),\C)^{F_\infty}\bigr),$$ 
where $e_l:=(\delta_{\alpha l})_{1\leq \alpha \leq n},$ and
$$
\partial \log \,h_i:=f_i\circ h_i^{-1}\circ (\partial h_i)\circ f_i^{-1}
\in
A^0\bigl(U_{i,\R},{\it End}(E)\otimes \Omega_{X/S}^1\bigr).
$$
For $i,j\in I$, we define
\begin{eqnarray*}
f_{ij}:=f_j^{-1}\circ f_i&\in& {\rm Mat}_n\bigl( \cO_X(U_{ij})\bigr)\\
{\rm dlog}\,f_{ij}:=f_j\circ (d f_{ij})\circ f_i^{-1}
&\in & \Gamma \bigl(U_{ij},{\it End}(E)\otimes \Omega_{X/S}^1\bigr).
\end{eqnarray*}
Then the isomorphism 
$$\hat{\rho}_{\mathcal{U},E,E\otimes \Omega^1_{X/S}}: 
\widehat{\rm Ext}_X^1 (E,E\otimes \Omega^1_{X/S})\to 
\check{H}^0\bigl(\cU,C({\rm ad}_{End(E)\otimes \Omega^1_{X/S}})\bigr)$$ 
constructed in Lemma \ref{bigdiaext} maps
$\widehat{\rm at}_{X/S}(\overline{E})$ to the class of 
\[
\bigl((-{\rm dlog}\,f_{ij})_{i,j\in I},(-\partial {\rm log}\,h_i)_{i\in I}\bigr).
\]
\end{proposition}

\proof 
Let $\nabla$ denote the unitary connection on $E_{\C}$ which is compatible with
the holomorphic structure.
We compute a cocycle $\bigl((\alpha_{ij})_{i,j}, (\beta_i)_i\bigr)$ which 
represents the image of the arithmetic extension $({\cat}(E),s_\nabla)$ 
under  $\hat{\rho}_{\mathcal{U},E,E\otimes \Omega^1_{X/S}}$.
We follow the construction of $\hat{\rho}_{\mathcal{U},E,E\otimes \Omega^1_{X/S}}$ given in Appendix A. Consider the diagram
\[
\begin{array}{ccccccccc}
0&\longrightarrow&{\mathcal End} (E)\otimes \Omega_{X/S}^1&
\longrightarrow&W &\longrightarrow&\cO_X&\longrightarrow&0\\
&&||&&\downarrow &&\downarrow {\scriptstyle j_E}&\\
0&\longrightarrow &E\otimes\Omega_{X/S}^1\otimes E^\vee&\longrightarrow 
&  P^1_{X/S}(E)\otimes E^{\vee}&\longrightarrow &E\otimes E^{\vee}&\longrightarrow &0.
\end{array}
\]
where the lower exact sequence is the extension ${\cat}(E)\otimes E^\vee$ and the
upper exact sequence is the pullback $({\cat}(E)\otimes E^\vee)\circ j_E$
of the lower exact sequence by $j_E$.
There is a unique connection $\nabla_i\colon E|_{U_i}\to E|_{U_i}\otimes \Omega_{U_i/S}^1$ 
such that $\nabla_i(f_i)=0$.
It satisfies
\[
\nabla_j(f_i)=\nabla_j(f_j\cdot f_{ij})=f_j\cdot d f_{ij}, 
\]
where the frames $f_i$ and $f_j$ are seen as ``line vectors" with entries sections of $E$.
The connection $\nabla_i$ determines an $\cO_{U_i}$-linear splitting $s_{\nabla_i}$ of 
${\cat}(E)$ over $U_i$ as in (\ref{sectgbconn}).
We write $j_E(1_X)=f_i\otimes f_i^\vee$, where $f_i^\vee$ denotes the
dual frame of $E^\vee$ --- which we may see as a ``column vector" with entries sections of $E^\vee$ --- and get
\begin{eqnarray*}
\alpha_{ij}&=&(s_{\nabla_j}\otimes {\rm id}_{E^\vee}-s_{\nabla_i}\otimes 
{\rm id}_{E^\vee})\circ j_E(1_X)\\
&=&(-\nabla_j+\nabla_i)f_i\otimes f_i^\vee \\
&=&\bigl (-f_j\cdot (d f_{ij})\bigr)\otimes f_i^\vee\\
&=&-{\rm dlog}\,f_{ij}.
\end{eqnarray*}
We observe that we have
\[
\nabla^{1,0}(f_i)=f_i\cdot h_i^{-1}\cdot (\partial h_i)
\]
by \cite[III.2, eq. (2.1)]{wells80}. Hence
\begin{eqnarray*}
\beta_i
&=&(s_{\nabla^{1,0}}\otimes {\rm id}_{E^\vee}-s_{\nabla_i}\otimes 
{\rm id}_{E^\vee})\circ j_E(1_X)\\
&=&-f_i\circ h_i^{-1}\circ (\partial h_i)\circ f_i^{-1}\\
&=&-\partial \log \,h_i.
\end{eqnarray*}
Our claim follows.
\qed

The properties of the arithmetic Atiyah class in Proposition \ref{tpla} may be 
recovered by straightforward cocycle computations using 
Proposition \ref{cocycdescrat}.

\subsubsection{}
Let us indicate that there is a straightforward
generalization of the construction 
of the arithmetic extension class $\at_{X/S}(E,\nabla)$
in $\widehat{\rm Ext}^1_X(E, E\otimes \Omega_{X/S}^1)$
given above when $S$ is a flat
arithmetic scheme over $\Spec R$, $X$ an integral arithmetic scheme
equipped with a l.c.i. morphism $\pi:X\rightarrow S$, smooth over $K$,
and $\nabla$ is a relative $\mathcal{C}^\infty$-connection for
$X_\Sigma(\C)/S_\Sigma(\C)$.

If the relative connection $\nabla$ is induced by an absolute connection
$\nabla_X$ via the canonical map
\begin{equation}\label{canmap}
\Omega_{X/{\Spec R}}^1\rightarrow \Omega_{X/S}^1,
\end{equation}
the relative and the absolute Atiyah class are related as follows.
The commutative square
\[
\begin{array}{ccc}
X &\stackrel{{\rm id}_X}{\longrightarrow} & X\\
\,\,\,\,\,\,\downarrow {\scriptstyle \pi} & &\downarrow\\
S &{\longrightarrow} & \Spec R
\end{array}
\]
induces by Lemma \ref{pullbackconnection} a commutative diagram
\begin{equation}
\begin{array}{ccccccccc}
 0&\longrightarrow &E\otimes \Omega_{X/{\Spec R}}^1 &\longrightarrow 
&P_{X/{\Spec R}}^1(E) & \longrightarrow &E & \longrightarrow &0\\
 & &\downarrow & 
& \,\,\,\downarrow {\scriptstyle \phi}& & ||& & \\
 0 &\longrightarrow & E\otimes\Omega_{ X/ S}^1 &\longrightarrow 
& P_{X/S}^1(E)& \longrightarrow &E & \longrightarrow &0.
\end{array}
\end{equation}
which identifies $\cat_{X/S}(E)$ with the pushout of $\cjet^1_{X/{\Spec R}}(E)$
along the canonical map (\ref{canmap}).
We have $s_\nabla=\phi_\C\circ s_{\nabla_X}$.
Hence the pushout of the arithmetic extension 
$(\cjet^1_{X/{\Spec R}}E,s_\nabla)$ along the canonical map (\ref{canmap})
is by its very definition 
in \cite[2.4.1]{bostkuennemann1}
the arithmetic extension $(\cat_{X/S}(E),s_\nabla)$.




%

\subsection{The first Chern class in arithmetic Hodge cohomology}\label{secfcc}
\subsubsection{}
For a hermitian vector bundle $\ol{E}$ on an arithmetic scheme $X$, flat and l.c.i. over $S=\Spec R,$ 
we put
\[
\hat{c}_1^H(\ol{E}):=\hat{c}_1^H(X/S,\ol{E}):={\rm tr}_E\circ 
(\widehat{\rm at}_{X/S}(\ol{E})\otimes E^\vee )\circ j_E
\in
\widehat{\rm Ext}^1(\cO_X,\Omega_{X/S}^1)
\]
where ${\rm tr}_E:E \otimes E^\vee {\to} \cO_X$ and 
$j_E:\cO_X{\to} End(E) \simeq E \otimes E^\vee$
denote the canonical morphisms.
We call $\hat{c}_1^H(\ol{E})$ the {\it first Chern class of $\overline E$ in
arithmetic Hodge cohomology}.

When $\ol{E}$ is a hermitian line bundle, ${\rm tr}_E$ and $j_E$ are 
the ``obvious" isomorphisms, and $\hat{c}_1^H(\ol{E})$ is nothing else than 
$\widehat{\rm at}_{X/S}(\ol{E})$ in $$\widehat{\rm Ext}_X^1(E,E \otimes \Omega_{X/S}^1) \simeq
\widehat{\rm Ext}_X^1(\cO_X,E^\vee \otimes E \otimes \Omega_{X/S}^1)
\simeq \widehat{\rm Ext}_X^1(\cO_X,\Omega_{X/S}^1).$$

Using the the description of the arithmetic Atiyah class
in terms of \v{C}ech cocycles in Proposition \ref{cocycdescrat}, and the 
expression of the differential of the determinant in terms of the trace, we 
obtain, after a straightforward computation:
\begin{equation}
\hat{c}_1^H(\ol{E}) = \hat{c}_1^H(\det \ol{E}).
\end{equation}
 
 Proposition \ref{cocycdescrat} also leads immediately to the following 
description of the first Chern class  in
arithmetic Hodge cohomology for hermitian line bundles:

\begin{lemma}
Let $\overline{L}$ be a hermitian line bundle on an arithmetic scheme
$X$. Choose an affine, open cover $\mathcal{U}=(U_i)_{i\in I}$
of $X$ such that $L$ admits a trivialization $l_i\in \Gamma(U_i,L)$ over $U_i$.
Put
\[
f_{ij}:= l_j^{-1}\cdot l_i\in \Gamma(U_{ij},\cO^*).
\]
Then 
\[
\hat{\rho}_{\mathcal{U},\Omega_{X/S}^1}\bigl(\hat{c}_1^H(\overline{L})\bigr)
=\bigl[(-{\rm dlog}\,f_{ij})_{i,j\in I},(-\partial {\rm log}\,\|l_i\|^2)_{i\in I}\bigr].
\]
\end{lemma}

\subsubsection{}
Let $\widehat{\rm Pic}(X)$ denote the group of isometry classes of hermitian line bundles
on $X$.
It follows immediately from Proposition \ref{tpla} that the map
\[
\hat{c}_1^H:\widehat{\rm Pic}(X)\rightarrow \widehat{\rm Ext}_X^1(\cO_X,\Omega_{X/S}^1)
\]
is a group homomorphism which satisfies
\[
\hat{c}_1^H(X/S,\,.\,) \circ f^*=f^*\circ \hat{c}_1^H(Y/S,\,.\,)
\]
for every morphism $f\colon X\to Y$ of integral, flat, l.c.i, arithmetic $S$-schemes.

\subsubsection{}
We consider the diagrams
\begin{equation}\label{arfcc1}
\begin{array}{cccccccc}
\cO(X)^* &\stackrel{\log\vert.\vert^2}\rightarrow & A^{0,0}(X_\R)&
\stackrel{a}{\rightarrow} & \widehat{\rm Pic}(X) & {\rightarrow} &
  {\rm Pic}(X)\\
\,\,\,\,\,\,\,\,\,\downarrow {\scriptstyle \rm - dlog} & &\, \downarrow
{\scriptstyle -\partial} & &\,\,\, \downarrow {\scriptstyle \hat{c}_1^H}&
&\downarrow {\scriptstyle c^H_1}\\
\Gamma(X,\Omega_{X/S}^1) &\rightarrow & 
{A^{0}(X_\R,\Omega_{X/S}^1)} &\stackrel{b}{\rightarrow} &
\widehat{\rm Ext}_X^1(\mathcal{O}_X,\Omega_{X/S}^1) &\stackrel{\nu}{\rightarrow} & 
{\rm Ext}_{\cO_X}^1(\mathcal{O}_X,\Omega_{X/S}^1).
\end{array}
\end{equation}
and
\begin{equation}\label{arfcc2}
\begin{array}{ccc}
\widehat{\rm Pic}(X) &\stackrel{c_1}{ \longrightarrow} &  A^{1,1}(X_\R) \\
\,\,\,\,\downarrow {\scriptstyle \hat{c}^H_1} & & \downarrow {\scriptstyle \iota} \\
\widehat{\rm Ext}_X^1(\mathcal{O}_X,\Omega_{X/S}^1)&\stackrel{\Psi}{\longrightarrow} 
& A^{0,1}(X_\R,\Omega_{X/S}^1).
\end{array}
\end{equation}
Here $A^{p,p}(X_\R)$ is by definition the space of {\it real} $(p,p)$-forms $\alpha$ on
the complex manifold $X_\Sigma(\C)$ which satisfy 
$F_\infty(\alpha)=(-1)^p\alpha$ (compare \cite[3.2.1]{gilletsoule90}).
The monomorphism $\iota$ is defined by
\[
A^{p,p}(X_\R)\hookrightarrow  A^{0,p}(X_\R,\Omega_{X/S}^p)\,,\,\,\,
\alpha\mapsto (2\pi i)^p \alpha 
\]
(compare \cite[1.1.5]{bostkuennemann1}).
Furthermore we have used the following morphisms
\begin{eqnarray*}
\log\vert.\vert^2 : \cO(X)^*&\longrightarrow & A^{0,0}(X_\R)\,,\,\,\,
f\longmapsto \log |f|^2,\\
{\rm dlog}\colon\cO(X)^*&\longrightarrow & \Gamma(X,\Omega_{X/S}^1)\,,\,\,\,
f\longmapsto f^{-1}df, \\
\Gamma(X,\Omega_{X/S}^1) &\longrightarrow & {A^0(X_\R,\Omega_{X/S}^1)}\,,\,\,\,
\alpha\longmapsto \alpha_\C, \\
\partial\colon \ker\,\partial\overline{\partial}|_{A^{0,0}(X_\R)}&\longrightarrow &
{A^0(X_\R,\Omega_{X/S}^1)}\,,\,\,\,f\longmapsto \partial f,\\
a\colon\ker\,\partial\overline{\partial}|_{A^{0,0}(X_\R)}&
{\longrightarrow} & \widehat{\rm Pic}(X)\,,\,\, f\longmapsto [(\cO_X,\|.\|_f)] 
\mbox{ with } \|1_X\|^2_f=\exp\,f, \\
b\colon {A^0(X_\R,\Omega_{X/S}^1)} &{\longrightarrow} &
\widehat{\rm Ext}_X^1(\mathcal{O}_X,\Omega_{X/S}^1),\,
T\longmapsto 
\Bigl[0 \to \Omega_{X/S}^1\stackrel{\binom{{\rm id}}{0}}{\to} \Omega_{X/S}^1\oplus \cO_X
\end{eqnarray*}
\hspace{9cm}
$\stackrel{(0, {\rm id})}{\to} \cO_X \to 0,
s:=\genfrac{(}{)}{0pt}{}{T}{{\rm id}}\Bigr]$

(compare the introduction and \cite[2.2]{bostkuennemann1}),
\begin{eqnarray*}
\widehat{\rm Pic}(X) & \longrightarrow & {\rm Pic}(X)\,,\,\,\,
[(L, \|\,.\,\|)]\longmapsto [L],\\
\nu\colon \widehat{\rm Ext}_X^1(\mathcal{O}_X,\Omega_{X/S}^1) &\longrightarrow & 
{\rm Ext}^1_{\mathcal{O}_X}(\mathcal{O}_X,\Omega_{X/S}^1)\,,\,\,\,
[(\mathcal{E},s)]\longmapsto [\mathcal{E}],\\
c_1^H\colon{\rm Pic}(X)&\longrightarrow & {\rm Ext}^1_{\mathcal{O}_X}
(\mathcal{O}_X,\Omega_{X/S}^1)
\,,\,\,[L]\longmapsto [{\rm tr}_L\circ {\rm at}_{X/S}(L)\circ i_L],\\
c_1\colon\widehat{\rm Pic}(X) &{ \longrightarrow} &  A^{1,1}(X_\R)\,,\,\,
[\overline{L}=(L,\|\,.\,\|)]\longmapsto -(2\pi i)^{-1}\nabla^2_{\overline{L}}, \\
\Psi\colon \widehat{\rm Ext}_X^1(\mathcal{O}_X,\Omega_{X/S}^1)&{\longrightarrow} 
& A^{0,1}(X_\R,\Omega_{X/S}^1),\,\,\,\mbox{defined in \cite[2.3.1]{bostkuennemann1}}.
\end{eqnarray*}

The horizontal lines in (\ref{arfcc1}) are exact by \cite[(2.5.2)]{gilletsoule89}
and \cite[2.2.1]{bostkuennemann1}.
Observe the analogy between (\ref{arfcc1}) and \cite[(2.5.2)]{gilletsoule89}.

\begin{proposition}
The diagrams (\ref{arfcc1}) and (\ref{arfcc2}) are commutative.
\end{proposition}

\proof
For $f$ in $\cO(X)^*$, we have
\begin{equation}
\partial \log\,|f|^2=\frac{\partial
  (f\ol{f})}{f\ol{f}}=\frac{\partial f}{f}=\frac{d f}{f}={\rm dlog}\,f
\end{equation}
which shows the commutativity of the left square in (\ref{arfcc1}).
The unitary connection $\nabla_f$ on  $(\cO_X,\|.\|_f)$ 
that is compatible with the holomorphic structure is given 
according to \cite[III.2 formula (2.1)]{wells80} by the formula
\[
\nabla_f^{1,0}(1)=\partial f\in A^0(X_\R,\Omega_{X/S}^1).
\]
Taking into account the correspondence between connections and splittings in 1.3 above (and notably the sign in (\ref{sign10})), it follows that the middle square commutes.
The commutativity of the right square holds by definition.
The square (\ref{arfcc2}) is commutative by formula (\ref{formfirstcc}).
\qed

\section{Hermitian line bundles with vanishing arithmetic Atiyah class}\label{sect3}

This section is devoted to the proof of assertions ${\bf I1}_{X,\Sigma}$ and ${\bf I2}_{X,\Sigma}$ in the Introduction (see 0.4 \emph{supra}).

In the first part of the section, we establish the special case of ${\bf I1}_{X,\Sigma}$ where $X$ is an abelian variety and $\Sigma$ has a unique or two conjugate elements. As mentioned in the Introduction, the validity of ${\bf I1}_{X,\Sigma}$ in this case has been established by Bertrand (\cite{bertrand95, bertrand98}) under suitable hypotheses of ``real multiplication".

In the second part of the section, we use some classical properties of Picard varieties to extend ${\bf I1}_{X,\Sigma}$ to arbitrary smooth projective varieties $X$ over number fields. Finally we establish ${\bf I2}_{X,\Sigma}$, which describes the kernel of the first class Chern in arithmetic Hodge cohomology $\hat{c}_1^H$ ``up to a finite group".

\subsection{Transcendence and line bundles with connections on abelian varieties}\label{trans}

The next paragraphs are devoted to the proof of the following theorem:

\begin{theorem}\label{thI}
Let $A$ be an abelian variety over a number field $K$, and $(L, \nabla)$ a line bundle over $A$ equipped with a connection (defined over $K$).

If there exists a field embedding $\sigma: K \hookrightarrow \C$ and a hermitian metric $\|.\|$ on the complex line bundle $L_\sigma$ on $A_\sigma(\C)$ such that the connection $\nabla_\sigma$ is unitary with respect to $\|.\|,$ then $L$ has a torsion class in $\Pic(A)$, and the metric $\|.\|$ has vanishing curvature. 

\end{theorem}

Actually this implies that the connection $\nabla$ is the unique one on $L$ such that $(L,\nabla)$ has a torsion class in the group of isomorphism classes of line bundles with connections over $A$ (see \ref{vanishingatar} \emph{infra}).

Let us indicate that this result admits an alternative formulation in terms of universal vector extensions of abelian varieties and their maximal compact subgroups, in the spirit of Bertrand's articles \cite{bertrand95, bertrand98}:

\begin{theorem}\label{thIalternative}
Let $B$ be an abelian variety over a number field $K$,  $B^\#$ the universal vector extension of $B$, and $P$ a point in $B^\#(K)$. 

If there exists a field embedding $\sigma: K \hookrightarrow \C$ such that the point $P_\sigma$ belongs to the maximal compact subgroup of $B^\#_\sigma (\C)$, then $P$ is a torsion point in $B^\#(K)$. 

\end{theorem}

Actually, for any given $K$ and $\sigma$, the implications in the statement of Theorems \ref{thI} and \ref{thIalternative} are equivalent when the abelian varieties $A$ and $B$ are dual to each other. This follows from the description of the universal vector extension $B^\#$ and of the maximal compact subgroup of $B^\#_\sigma (\C)$ recalled in Appendix B (see notably B.6 applied to $k=K$ and $X=A$, in which case $E_{X/k}= B^\#$, and B.7 applied to $X=A_\sigma$, in which case $E_{X/\C}(\C)=B^\#_\sigma (\C)$).

The formulation in Theorem \ref{thI} turns out to be more convenient for the proof, which will proceed along the following lines.

Firstly, the data $(L,\nabla)$ in Theorem \ref{thI} may be ``translated" in terms of algebraic groups:  the total space of the $\mathbb{G}_m$-torsor associated to $L$ defines a commutative algebraic group $L^\times$ , and the connection $\nabla$ an hyperplane in its Lie algebra $\Lie L^\times.$ Then an application of the theorem of Schneider-Lang to this situation will show that, \emph{if there exists a family $(\gamma_1,\ldots,\gamma_g)$ of points in the lattice of periods $\Gamma_{A_\sigma}$ of $A_\sigma$ which constitutes a $\C$-basis of $\Lie A_\sigma$ such that the monodromy of the complex line bundle with connection $(L_\sigma,\nabla_\sigma)$ along each $\gamma_i$  lies in $\overline{\Q}^\ast$, then $L$ is torsion.}

This criterion easily leads to a derivation of Theorem \ref{thI} when the image of the embedding $\sigma$ lies in $\R$. Indeed a simple ``reality" argument then shows that the monodromy of  $(L_\sigma,\nabla_\sigma)$ along the ``real periods" of $A_\sigma$ lies in $\{1,-1\}$.

When the image of $\sigma$ does not lie in $\R$, we may assume that $K$ is Galois over $\Q$, and consider the involution $\tau$ of $K$ such that $\sigma \circ \tau = \overline{\sigma}$.  It will turn out that we may apply the above criterion to the line bundle with connection on $A\times_K A_\tau$ defined as the external tensor product of $(L,\nabla)$ and $(L_\tau,\nabla_\tau)$ to establish that $L \boxtimes L_\tau$, hence $L$, is torsion. 


\subsubsection{Line bundles with connections on abelian varieties}\label{cav}

Let $A$ be an abelian variety over a field $k$ of characteristic 
zero, and $L$ a line bundle over $A$. We may choose a rigidification 
of $L$, namely a trivialization $\phi: k\simeq L_{e}$ of its fiber at the 
zero element $e$ of $A(k)$, or equivalently the vector 
$\ell:=\phi(1)$ in 
$L_{e}\setminus\{0\}.$

In the sequel, we shall assume that the following 
equivalent\footnote{Indeed (ii) and (iii) are equivalent
by the very definition of $\at_{A/k}L$,
(ii) is equivalent
to the rational vanishing of the first Chern class of $L$
(hence (i) implies (ii)), and if the first Chern class of $L$ 
vanishes rationally, one gets (i) from \cite[II.2 Cor. 1 to Th. 2]{kleiman66}, as
 $\Pic_{A/k}^0=\Pic_{A/k}^\tau$ by \cite[Cor. 6.8]{mumford82}.}
conditions are satisfied:

(i) \emph{the line bundle $L$ is algebraically equivalent to the trivial 
line bundle;}

(ii) \emph{the Atiyah class $\at_{A/k}L (=\jet^1_{A/k} L)$ of $L$ vanishes;}

(iii) \emph{the line bundle $L$ may be equipped with an algebraic 
connection $\nabla$.}

Observe that the connection $\nabla$ is 
necessarily flat\footnote{To establish this, we may reduce to the case $k=\C$ and use transcendental arguments. We may also  assume that $k$ is algebraically closed, and observe that the curvature $\nabla^2$ of an algebraic connection on $L$ depends only on the isomorphism class of $L$ and defines a morphism of algebraic groups over $k$ from the dual abelian variety $A^\vee$ to the additive group $\Gamma(A,\Omega^2_{A/k}) (\simeq \wedge^2 (\Lie A)^\vee).$ Since $A$ is proper and connected, any such morphism is zero.} and that the set of connections on $A$ is a torsor under the $k$-vector space $\Gamma(A, \Omega^1_{A/k}) \simeq (\Lie A)^\vee$ of regular 1-forms on $A$, which acts additively on this set.

Beside, the $\G_m$-torsor $L^\times$ defined by 
deleting the zero section from the total space\footnote{namely, the spectrum of the quasi-coherent $\cO_A$-algebra $\bigoplus_{n \in \N} L^{\vee\otimes n}.$}
$\mathbb{V}(L^\vee)$ of 
$L$ admits a unique structure of commutative algebraic group over 
$k$ such that  
the diagram
\begin{equation}
    \label{groupLcross}
    0 \longrightarrow \G_{m,k} \stackrel{\phi}{\longrightarrow} L^\times \stackrel{\pi}{\longrightarrow} A  
    \longrightarrow 0,
\end{equation}
--- where $\phi$ denotes the composite morphism $\G_{m,k} \stackrel{\phi}{\simeq} L_{e}^{\times} 
    \hookrightarrow L^\times$ and  $\pi$ the restriction of the ``structural morphism" 
from $\mathbb{V}(L^\vee)$ to $A$ --- becomes a short exact sequence of 
commutative algebraic groups.
Its zero element is the $k$-point $\epsilon \in L^\times (k)$ defined by $\ell$. (See for instance \cite{serre59}, VII.3.16.)

From (\ref{groupLcross}), we derive a short exact sequence of 
$k$-vector spaces:
\begin{equation}
    \label{LiegroupLcross}
    0 \longrightarrow \Lie \G_{m,k}  
   \stackrel{\Lie \phi}{\longrightarrow} \Lie L^\times \stackrel{\Lie \pi}{\longrightarrow} 
    \Lie A  
    \longrightarrow 0.
\end{equation}

Recall that a connection over a vector bundle on a smooth algebraic variety may be described \emph{\`a la} Ehresmann as an equivariant splitting of the differential of the structural morphism of its frame bundle (see for instance  \cite{kobayashinomizu63}, Chapter II, or \cite{spivak70}, Chapter 8; the constructions of \emph{loc. cit.} in a differentiable setting can be immediately transposed in the algebraic framework of smooth algebraic varieties). 
In the present situation, 
a connection $\nabla$ on $L$ may thus be seen as a $\G_{m,k}$-equivariant 
splitting of the surjective morphism of vector bundles over 
$L^\times$ defined by the differential of $\pi$:
$$D\pi: T_{L^\times} \longrightarrow \pi^\ast T_{A}.$$
In particular,  its value at the unit element $\epsilon$ of $L^\times$ defines a 
$k$-linear splitting 
$$\Sigma: \Lie A \longrightarrow \Lie L^\times$$
of (\ref{LiegroupLcross}).

Conversely, from any 
$k$-linear right inverse $\Sigma$ of $\Lie \pi$, we deduce 
a $\G_m$-equivariant splitting of $D\pi$ by constructing its 
$L^\times$-equivariant extension to $L^\times.$

Through these constructions, connections on $L$ and $k$-linear splittings of 
(\ref{LiegroupLcross}) 
correspond bijectively. 
Indeed, by means of the identification  
\[
\begin{array}{rcl}
 \Lie \G_{m,k}
  & \stackrel{\sim}{\longrightarrow}  & k   \\
 \lambda\cdot X \frac{\partial\,\,}{\partial X} & \longmapsto  & \lambda, 
\end{array}
\]
the set of $k$-linear splittings of  (\ref{LiegroupLcross}) becomes naturally a torsor under $(\Lie A)^\vee$, and the above constructions are compatible with the (additive) actions of $(\Lie A)^\vee$ on the set of these splittings and on the set of connections on $L$.

This correspondence may also be described as follows. 
A linear splitting 
$\Sigma$ as above may also be seen as a morphism $\tilde{\ell}: 
A_{e,1}\rightarrow L^{\times}_{\epsilon,1}$ from the first infinitesimal 
neighbourhood $A_{e,1}$ of $e$ in $A$ to the first infinitesimal 
neighbourhood $L^{\times}_{\epsilon,1}$ of $\epsilon$ in $L^\times$ 
which is a right inverse of the map $\pi_{\epsilon, 1}: 
L^{\times}_{\epsilon,1} \rightarrow A_{e,1}$ deduced from $\pi.$ In 
other words, $\tilde{\ell}$ is a section of $L$ over $A_{{e,1}}$ such 
that $\tilde{\ell}(e)=l.$ The connection $\nabla$ associated to 
$\Sigma$  is the unique one such that $\nabla \tilde{\ell}(e)=0$. 


\subsubsection{The complex case}\label{complexcase}

 If $G$ is a commutative algebraic group over $\C$, its exponential map 
will be denoted $\exp_{G}$. It is the unique morphism of 
$\C$-analytic Lie groups
$$\exp_{G}: \Lie G \longrightarrow G(\C)$$
whose differential at $0\in \Lie G$ is $\Id_{\Lie G}$. Its kernel 
$$\Gamma_{G}:= \ker \exp_{G}$$ is a discrete additive subgroup of 
$\Lie G$. When $G$ is 
connected, $\exp_{G}$ is a universal covering of $G(\C)$, and $\Gamma_{G}$ may be identified with the fundamental group 
$\pi_{1}(G(\C),0_{G})$, or with the homology group $H_{1}(G(\C),\Z)$. 

Let us go back to the situation considered in paragraph \ref{cav}, in the case where the base field  $k$ is $\C$, and fix the algebraic connection $\nabla$ on $L.$

Then the diagram
\[
\begin{array}{ccc}
\Lie L^\times &\stackrel{D\pi}{\longrightarrow} & \Lie A\\
{ \;\;\;}\downarrow {\scriptstyle \exp_{L^\times}} & &{\;\;\;}\downarrow {\scriptstyle 
\exp_{A}}\\
L^\times(\C) &\stackrel{\pi}{\longrightarrow} & A(\C).
\end{array}
\]
is commutative. Consequently the morphism of groups 
$$\exp_{L^\times}\circ\, \Sigma : \Gamma_{A} \longrightarrow 
L^\times(\C)$$
takes its value in $\ker \pi \simeq \C^\ast.$ It coincides with the 
monodromy representation 
$$\rho: \Gamma_{A}=H_{1}(A(\C), \Z) \longrightarrow \C^\ast$$
of the line bundle with flat connection $(L,\nabla)$ --- or more properly of the corresponding objects in the analytic category ---  over $A(\C)$. Indeed, the horizontal $\G_{m,\C}$-equivariant foliation on $L^\times(\C)$ defined by $\nabla$ is translation invariant, and its leaves are precisely the translates in $L^\times(\C)$ of the image of $\exp_{L^\times}\circ \Sigma.$

\subsubsection{An application of the Theorem of Schneider-Lang}

To establish Theorem \ref{thI}, we shall use the following classical transcendence result on 
commutative algebraic groups:

\begin{theorem}\label{SL}
    Let $K$ be a number field and $\sigma: K \hookrightarrow \C$ a 
    field embedding, and let $G$ be a commutative algebraic group 
    over $K$, and $V$ a $K$-vector subspace of $\Lie G$.
    
    If there exists a basis $(\gamma_{1},\ldots,\gamma_{v})$ of the 
    complex vector space $V_{\sigma}$ such that, for every $i \in 
    \{1,\ldots,v\},$ $\exp_{G_{\sigma}}(\gamma_{i})$ belongs to 
    $G(\overline{\Q}),$ then $V$ is the Lie algebra of some algebraic 
    subgroup $H$ of $G$.
\end{theorem}

We have denoted $\overline{\Q}$ the algebraic closure of $\Q$ in 
$\C$. By means of the embedding $\sigma$, it may be seen as an algebraic 
closure of $K$, and the group $G(\overline{\Q})$ of
$\overline{\Q}$-rational points of $G$ becomes a subgroup of the 
group $G_{\sigma}(\C)$ of its complex points.

Observe also that the subgroup $H$ whose existence is asserted in 
Theorem \ref{SL} may clearly be chosen connected, and then $H$ is 
clearly unique, defined over $K$, and the group $H_{\sigma}(\C)$ of 
its complex points coincides with $\exp_{G_{\sigma}}(V_{\sigma}).$

Theorem \ref{SL} has been established by Lang (\cite{lang66}, IV.4, Theorem 2), who elaborated on 
some earlier work of Schneider on abelian functions and the 
transcendence of their values \cite{schneider41}. We refer the reader 
to 
\cite{waldschmidt87} (where it appears as Th\'eor\`eme 
5.2.1) for more details on Theorem \ref{SL} and its classical applications.

Let us point out that  Theorem \ref{SL} is  now subsumed by  various renowned more recent 
results --- namely, the transcendence criterion of Bombieri and the analytic subgroup
theorem of W\"ustholz. The reader may find  a recent survey of these 
and related transcendence results on commutative algebraic groups in the monograph \cite{bakerwustholz07}.

We now return to the situation  considered in paragraph \ref{cav}, where we assume that the base field $k$ is a number field $K$. 

Taking into account the relation in the complex case between the monodromy of connections on $L$ 
and the exponential map of the algebraic group $L^\times$  described in \ref{complexcase}, we 
derive from the theorem of Schneider-Lang (Theorem \ref{SL} above) 
applied to the algebraic group $G=L^\times$:

%
%

\begin{corollary}\label{corSL}
    Let $A$ be an abelian variety of dimension $g$ over a number 
    field $K$, and $(L,\nabla)$ a line bundle over $L$ equipped with a 
    flat connection (defined over $K$).
    
    Let $\sigma: K \hookrightarrow \C$ be a field embedding, and let 
    $\rho_{\sigma}:\Gamma_{A_{\sigma}} \longrightarrow \C^\ast$ 
    denote the monodromy representation attached to the flat complex 
    line bundle $(L_{\sigma},\nabla_{\sigma})$ over $A_{\sigma}(\C).$
    
    If there exists $\gamma_{1},\ldots,\gamma_{g}$ in $\Gamma_{A_\sigma}$
     such that $(\gamma_{1},\ldots,\gamma_{g})$ is a 
    basis of the $\C$-vector space $\Lie A_{\sigma}$ and such that, for 
    every $i \in \{1,\ldots,g\}$, $\rho_\sigma(\gamma_i)$ belongs 
    to $\overline{\Q}^\ast$, then $L$ has a 
    torsion class in $\Pic(A)$.
\end{corollary}

Observe that conversely,  if $n$ is a positive integer such that $L^{\otimes n} \simeq \cO_A$, the unique connection $\nabla_L^{\rm tor}$ on $L$ such the $n$-th tensor power of the line bundle with connection $(L,\nabla_L^{\rm tor})$ is isomorphic to $(\cO_A,d)$ is such that, for any $\sigma: K \hookrightarrow \C$, the image of the monodromy $\rho_\sigma$ of $(L_\sigma,\nabla_{L,\sigma}^{\rm tor})$ lies in the $n$-th roots of unity, hence in $\overline{\Q}^\ast.$ By elaborating slightly on the proof below, one may show that, with the notation of Corollary \ref{corSL}, the connection $\nabla$ necessarily coincides with the connection $\nabla_L^{\rm tor}$ so defined. We leave this to the interested reader.

\proof We consider the $K$-linear map $\Sigma: \Lie A \longrightarrow 
\Lie L^\times$ associated to the connection $\nabla$ as in \ref{cav}, and 
its image $V:= \Sigma (\Lie A).$ 
The vectors 
$\tilde{\gamma}_{i}:= \Sigma_{\sigma}(\gamma_{i}),$ $1\leq i 
\leq g,$
constitute a basis of the $\C$-vector space $V_{\sigma}$. Moreover 
the image $\exp_{L^\times_{\sigma}}(\tilde{\gamma}_{i})$ of 
$\tilde{\gamma}_{i}$ by the exponential map of 
$L^\times_{\sigma}$ is the point of $L^\times_{\sigma,e}\simeq \C^\ast$ 
defined by the monodromy $\rho_{\sigma}(\gamma_{i})$ of $\gamma_{i}$. 
According to our assumption, these images belong to 
$L^\times(\overline{\Q})$. 

The theorem of Schneider-Lang now shows that 
$V$ is the Lie algebra of a connected algebraic subgroup $H$, defined 
over $K$. Since $\Lie \pi_{\mid H}: \Lie H = V \rightarrow \Lie A$ is an 
isomorphism of $K$-vector spaces, the morphism of algebraic groups 
$\pi_{\mid H}: H \rightarrow A$ is \'etale, and consequently $H$ is an 
abelian variety over $K$ and $\pi_{\mid H}$ an isogeny.

By the very construction of $H$ as a subscheme of $L^\times$, the 
inverse image $\pi_{\mid H}^\ast L$ of $L$ on $H$ is 
trivial. If $N$ denotes the degree of $\pi_{\mid H}$, it follows that 
$L^{\otimes N}$ --- which is isomorphic to the norm, relative to 
$\pi_{\mid H}$, of $\pi_{\mid H}^\ast L$ --- is a trivial line bundle.
\qed

\subsubsection{Reality I}  Let us keep the framework of paragraph \ref{cav}, 
and suppose now that the base field $k$ is $\R$.

The line bundle with connection $(L,\nabla)$   defines 
a real analytic line bundle with flat connection $(L^\R, 
\nabla^\R)$ over the compact real analytic Lie group 
$A(\R)$. Its monodromy defines a representation $\rho_{\R}$ of 
the fundamental group $\pi_{1}(A(\R),0_{A})$, or equivalently 
of the homology group $H_{1}(A(\R)^\circ, \Z)$ of the 
connected component of $0_{A},$ with values in $\R^\ast.$

Actually the inclusion $\iota: A(\R)^\circ \hookrightarrow A(\C)$ 
defines an injective map of free abelian groups, of respective ranks 
$g$ and $2g$,
$$\iota_\ast: H_{1}(A(\R)^\circ,\Z) \longrightarrow H_{1}(A(\C),\Z),$$
and the monodromy representation $\rho_{\R}$ coincides with the restriction 
$\rho_\C \circ \iota_\ast$ of the monodromy representation
$$\rho_\C: H_{1}(A(\C),\Z) \longrightarrow \C^\ast$$
defined by the $\C$-analytic line bundle with flat connection 
$(L_\C,\nabla_\C)$ over the compact $\C$-analytic Lie 
group $A(\C).$

\begin{lemma}\label{real}
    The following conditions are equivalent:
    
    {\rm (i)} There exists a hermitian metric $\Vert.\Vert$ on the complex 
    line bundle $L_\C$ on $A(\C)$ such that the 
    connection $\nabla_\C$ is unitary with respect to 
    $\Vert.\Vert$\footnote{or, equivalently, such that 
    $\nabla_\C$ is the Chern connection associated to 
    $\Vert.\Vert$.}.

 {\rm (ii)}  The monodromy representation $\rho_{\R}$ takes its values 
    in $\{1,-1\}.$
\end{lemma}

Clearly Condition (i) is equivalent to: 

 {\rm (i')} \emph{ The monodromy representation $\rho_{\C}$ takes its values in 
    $U(1):=\{z \in \C \mid \vert z\vert =1\}.$}

In the sequel, we shall only use the implications $ {\rm (i)} \Rightarrow 
 {\rm (i')}  \Rightarrow  {\rm (ii)} $, which are straightforward.
 To show $ {\rm (ii)} \Rightarrow  {\rm (i')} $, let $\Gamma^+:= \iota_\ast(H_{1}(A(\R)^\circ,\Z)),$ and
observe that the elements of $\Gamma_{A_\C}$ which are ``purely 
imaginary" in $\Lie A_\C \simeq (\Lie A){\otimes}_{\R}\C$ 
constitute a subgroup $\Gamma^-$ of rank $g$ such that 
$\Gamma^+ \cap \Gamma^-=\{0\}$, that 
$\Gamma/\Gamma^+ \oplus \Gamma^-$ is a 2-torsion group, and that the image $\rho_\C(\Gamma^-)$ of  $\Gamma^-$ by the monodromy 
representation lies in $U(1).$ We leave the details to the reader.

\subsubsection{Reality II} In this paragraph, we still
keep the framework of the paragraph \ref{cav}, and we now assume that the base 
field $k$ is $\C$. We may apply the considerations of the last paragraph  to the abelian variety 
over $\R$ deduced from $A$ by Weil restriction of scalar from $\C$ to 
$\R$. This leads to the following results, that we formulate without explicit reference to Weil restriction. 

Let  $A_{-},$ $L_{-},$ $\nabla_{-}$ be respectively the complex abelian 
variety, the line bundle over $A_{-}$, and the connection over 
$L_{-}$ deduced from $A,$ $L,$ and $\nabla$ by the base change $\Spec 
\C \rightarrow \Spec \C$ defined by complex conjugation.

Let us consider the complex abelian variety
$$B:= A \times A_{-},$$
the two projections
$$\pr: B \longrightarrow A \;\;\;\;\mbox{ and } \pr_{-}: B 
\longrightarrow A_{-},$$
and $(\tilde{L},\tilde{\nabla})$ the line bundle with connection over 
$B$ defined as the tensor product of $\pr^\ast(L,\nabla)$ and 
$\pr_{-}^\ast(L_{-},\nabla_{-})$.

Let $j: \Lie A \rightarrow \Lie A_{-}$ denote the canonical 
$\C$-antilinear isomorphism. It maps bijectively $\Gamma_{A}$ onto 
$\Gamma_{A_{-}},$ and we may introduce  the diagonal embedding
$$
\begin{array}{cccl}
    \Delta: & \Gamma_{A} & \longrightarrow &
    \Gamma_{A} \oplus \Gamma_{A_{-}}  \simeq \Gamma_{B} \\
     & \gamma & \longmapsto & (\gamma, j(\gamma)).
\end{array}
$$
Observe that any $\Z$-basis $(\gamma_{1},\ldots,\gamma_{2g})$ of $\Gamma_A$ is a $\R$-basis of 
$\Lie A$, and consequently its image 
$(\Delta(\gamma_{1}),\ldots,\Delta(\gamma_{2g}))$ by $\Delta$ is a 
$\C$-basis of $\Lie B$.

Let $\rho$ (resp. $\rho_{-},$ $\tilde{\rho}$) be the monodromy 
representation of $\Gamma_{A}$ (resp. $\Gamma_{A_{-}},$ $\Gamma_{B}$) 
defined by the line bundle with connection $(L,\nabla)$ (resp. 
$(L_{-},\nabla_{-}),$ $(\tilde{L},\tilde{\nabla})$).

It is straightforward that, for any $\gamma$ in $\Gamma_{A},$ the 
following relations hold:
$$\rho_{-}(j(\gamma))=\overline{\rho(\gamma)},$$
and
$$\tilde{\rho}(\Delta(\gamma))=\rho(\gamma).\rho_{-}(j(\gamma))=\vert 
\rho(\gamma) \vert^2.$$

These observations establish:

\begin{lemma}\label{real2}
    If there exists a hermitian metric $\Vert.\Vert$ on the complex 
    line bundle $L$ on $A(\C)$ such that the 
    connection $\nabla$ is unitary with respect to 
    $\Vert.\Vert$, then the image $\Delta(\Gamma)$ of the diagonal 
    embedding $\Delta$ contains a $\C$-basis of $\Lie B$, and is 
    included in the kernel of the monodromy representation 
    $\tilde{\rho}$ of $(\tilde{L},\tilde{\nabla})$.
    \end{lemma}

\subsubsection{Conclusion of the proof of Theorem \ref{thI}}

The following statement is a straightforward consequence of 
Corollary \ref{corSL} to the Theorem of Schneider-Lang, combined with  Lemma \ref{real} above:

\begin{corollary}\label{theoreal}
    Let $A$ be an abelian variety  over a number 
    field $K$, and $(L,\nabla)$ a line bundle over $A$ equipped with a 
    flat connection defined over $K$, and let $\sigma: K \hookrightarrow \C$ be a field embedding that is \emph{real}, namely such that its image $\sigma(K)$ lies in $\R$.

    If there exists a hermitian metric $\Vert.\Vert$ on the complex 
    line bundle $L_\sigma$ on $A_\sigma(\C)$ such that the 
    connection $\nabla_\sigma$ is unitary with respect to 
    $\Vert.\Vert$, then $L$ has a 
    torsion class in $\Pic(A)$.
\end{corollary}

If we use Lemma \ref{real2} instead of Lemma \ref{real}, we may prove:

\begin{corollary}\label{theotau}
    Let $A$ be an abelian variety  over a number 
    field $K$, and $(L,\nabla)$ a line bundle over $A$ equipped with a 
    flat connection defined over $K$.
    
  Let $\sigma: K \hookrightarrow \C$ be a field embedding, and let 
  $\tau$ be a (necessarily involutive) automorphism of the field $K$ such that $\sigma \circ 
  \tau = \overline{\sigma}.$

    If there exists a hermitian metric $\Vert.\Vert$ on the complex 
    line bundle $L_\sigma$ on $A_\sigma(\C)$ such that the 
    connection $\nabla_\sigma$ is unitary with respect to 
    $\Vert.\Vert$, then $L$ has a 
    torsion class in $\Pic(A)$.
\end{corollary}

Observe that when $\tau = {\rm Id}_K$ Corollary \ref{theotau} reduces to Corollary \ref{theoreal} above. We have however chosen to present explicitly the statement of Corollary  \ref{theoreal} and its proof above, since the basic idea behind  the proofs of Corollaries  \ref{theoreal} and \ref{theotau} appears more clearly in the first one, which indeed has been inspired by Bertrand's proof in \cite{bertrand95} and \cite{bertrand98}.

 \noindent \emph{Proof of Corollary \ref{theotau}.} As usual we denote $A_{\tau},$ $L_{\tau},$ and $\nabla_{\tau}$  respectively the  abelian 
variety over $K$, the line bundle over $A_{\tau}$, and the connection over 
$L_{\tau}$ deduced from $A,$ $L,$ and $\nabla$ by the base change $\Spec 
K \rightarrow \Spec K$ defined by $\tau$.
We may also introduce the  abelian variety over $K$
$$B:= A \times A_{\tau},$$
the two projections
$$\pr: B \longrightarrow A \;\;\;\;\mbox{ and } \pr_{\tau}: B 
\longrightarrow A_{\tau},$$
and $(\tilde{L},\tilde{\nabla})$ the line bundle with connection over 
$B$ defined as the tensor product of $\pr^\ast(L,\nabla)$ and 
$\pr_{\tau}^\ast(L_{\tau},\nabla_{\tau})$.

Lemma \ref{real2} applied to $(A_{\sigma}, L_{\sigma},\nabla_{\sigma})$ 
shows that the hypotheses of Corollary \ref{corSL} are satisfied by 
the abelian variety $B$ over $K$, and the line bundle with connection 
$(\tilde{L},\tilde{\nabla})$ over $B$. Consequently $\tilde{L}$ has a 
torsion class in $\Pic (B)$, and so $L$ itself --- which is isomorphic 
to the restriction of $\tilde{L}$ to $A\times \{e\} \simeq A$ --- has 
a torsion class in $\Pic (A).$ 
\qed

Finally consider  $K,$ $A,$ $(L,\nabla),$ $\sigma$ and  $\Vert.\Vert$ as in the statement of Theorem \ref{thI}. 

Let us first show that  $L$ has a torsion class in $K$. To achieve this, let us  choose a finite field extension $K'$ of $K$ admitting an automorphism $\tau$ and an embedding $\sigma'$ in $\C$ that extends $\sigma$ and satisfies $\sigma'\circ \tau= \overline{\sigma'}$ --- for instance the subfield $K'$ of $\C$ generated by $\sigma(K)$ and its image by complex conjugation. We may apply Corollary \ref{theotau} to the number field $K'$ equipped with the complex embedding $\sigma'$, and to the abelian variety $A_{K'}$ and the line bundle with connection $(L_{K'}, \nabla_{K'})$ deduced from $A$ and $(L,\nabla)$ by the base change $\Spec K' \rightarrow \Spec K.$ Therefore $L_{K'}$ has a torsion class in $\Pic(A_{K'}).$
Since the base change morphism $$\Pic(A) \longrightarrow \Pic(A_{K'})$$ is injective, this indeed implies that $L$ has a torsion class in $\Pic(A)$. 

To complete the proof of Theorem \ref{thI}, it is sufficient to observe that the curvature of $\Vert.\Vert$ --- or equivalently, of the $\mathcal{C}^\infty$-connection $\nabla_{\mathcal{C}^\infty} = \nabla_\sigma + \overline{\partial}_{L_\sigma}$ on $L_{\sigma}$ ---  vanishes for reason of type\footnote{One could also argue that this curvature coincides with the one of the holomorphic connection $\nabla_\sigma$, which vanishes, as recalled above.}: it is a $2$-form on $A_\sigma(\C)$ of type $(2,0)$, since $\nabla_\sigma$ is holomorphic, and purely imaginary, since $\nabla_{\mathcal{C}^\infty}$ is unitary.

\subsection{Hermitian line bundles with vanishing arithmetic Atiyah class on smooth projective varieties over number fields}\label{vanishingatar}

Let $K$ be a number field, and $\Sigma$ a non-empty set of field embeddings 
of $K$ in $\C$, stable under complex conjugation. 

To these data is naturally attached the arithmetic ring in the sense of 
Gillet-Soul\'e (\cite{gilletsoule90}, 3.1.1)  defined as the triple 
$(K,\Sigma, F_\infty)$ where $F_\infty$ denotes the conjugate linear 
involution of $\C^\Sigma$ defined by 
$F_\infty(a_\sigma)_{\sigma \in \Sigma}
:= (\overline{a_{\overline\sigma}})_{\sigma \in \Sigma}.$

\subsubsection{}\label{remcomplex}

Recall that, for any line bundle $M$ over a smooth projective connected variety $V$ over $\C$, the following conditions are equivalent, as a consequence of the GAGA principle and Hodge theory:
\begin{enumerate}
\item[(a1)] \emph{the Atiyah class $\at_{V/\C}M$ of $M$ vanishes in $H^{1,1}(V/\C):= \Ext^1_{\cO_V}(\cO_V, \Omega^1_{V/\C});$}

\item[(a2)] \emph{the first Chern class $c_1(M^\an)$ of the holomorphic line bundle $M^\an$ over $V(\C)$ deduced from $M$ vanishes rationally \emph{(that is, in $H^2(V(\C), \Q)$, or equivalently in $H^2(V(\C), \C)$)};}

\item[(a3)] \emph{there exists a $\mathcal{C}^\infty$-hermitian metric $\Vert . \Vert$ with vanishing curvature on $M^\an.$}
\end{enumerate}

Moreover, when they are satisfied, the metric $\Vert . \Vert$ is unique up to a constant factor in $\R^\ast_+$, and the $(1,0)$-part $\nabla^{1,0}$ of the $\mathcal{C}^\infty$-connection $\nabla$ on $M^\an$ that is unitary (for $\Vert . \Vert$) and compatible with the holomorphic structure is the unique integrable holomorphic connection $\nabla_M^{\rm u}$ whose monodromy lies in 
$U(1):= \{z\in \C \mid \vert z \vert =1 \}.$  
Observe also that $\nabla_M^{\rm u}$ algebraizes, and may be seen as as an ``algebraic" connection on the line bundle $M$ on the algebraic variety $V$ over $\C$.

\subsubsection{}\label{prerem}

Let $X$ be a smooth, projective, geometrically connected scheme over $K$,  
and $E_{X/K}$ the universal vector extension of $\Pic^0_{X/K}$ (see 
Appendix B for basic facts on Picard varieties and their universal vector extensions).

In the sequel, we shall consider $X$ and $\Spec K$ as  arithmetic schemes 
over the arithmetic ring $(K,\Sigma,F_\infty)$.

In particular, a hermitian line bundle $\ol{L}$ over $X$ is the data of a 
line bundle $L$ over $X$ and of a $\mathcal{C}^\infty$-hermitian metric 
$\Vert . \Vert_{\ol{L}},$ invariant under complex conjugation, on the 
holomorphic line bundle $L^\an_\C$ over $X_\Sigma(\C)= \coprod_{\sigma \in \Sigma}  X_\sigma(\C).$

According to the observations in \ref{remcomplex},  for any line bundle $L$ over $X$, the following conditions are equivalent:
\begin{enumerate}
\item[(b1)] \emph{the Atiyah class $\at_{X/K} L$ of $L$ in $H^{1,1}(X/K):= \Ext^1_{\cO_X}(\cO_X, \Omega^1_{X/K})$ vanishes;}
\item[(b2)] \emph{there exists a 
 $\mathcal{C}^\infty$-hermitian metric $\Vert . \Vert$ with 
vanishing curvature, invariant under 
complex conjugation, on the holomorphic line bundle $L^\an_\C$ over $X_\Sigma(\C)$.}
  \end{enumerate}
 
 When (b1) and (b2)  are realized, the metric $\Vert . \Vert$ is 
unique, up to some multiplicative constant, on every 
component $X_\sigma(\C)$ of $X_\Sigma(\C).$

Observe also that these conditions hold precisely when some positive power of  
the line bundle  $L$ is algebraically equivalent to zero\footnote{By definition 
a line bundle on $X$ is
algebraically equivalent to zero if and only if 
its restriction to the geometric fiber $X_{\overline{K}}$
is algebraically equivalent to zero.} (see for instance \cite[II.2 Cor. 1 to Th. 2]{kleiman66}).

\subsubsection{}\label{preremsuite}

Consider now  a  line bundle $L$   on $X$ satisfying Conditions (b1) and (b2) above, and let us choose  
a 
 $\mathcal{C}^\infty$ hermitian metric $\Vert . \Vert$ on $L_\C$, as in Condition (b1) above.
 
We shall denote $\overline{L}$ the hermitian line bundle $(L, \Vert . \Vert)$ over $X$, and $\nabla_{\overline{L}}$  the unitary connection on $L_\C$ which
is compatible with the holomorphic structure. It does not depend on the actual choice of $\Vert . \Vert$. Indeed,
for any $\sigma$ in $\Sigma$,  the 
$(1,0)$-part $\nabla_{\overline{L}}^{1,0}$
of $\nabla_{\overline{L}}$ coincides with $\nabla^{\rm u}_{L_\sigma}$ over $X_\sigma(\C)$.

It is a straightforward consequence of our definitions that the following 
conditions are equivalent:
\begin{enumerate}
\item[(1)]
\emph{the line bundle $L$ admits a connection $\nabla\colon L\to L\otimes \Omega_{X/K}^1$ (over $K$) 
such that the induced holomorphic connection $\nabla_{\C}$ on $L_\C$ over $X_\Sigma (\C)$ equals $\nabla_{\overline{L}}^{1,0}$, or equivalently such that for any $\sigma$ in $\Sigma$ 
the induced holomorphic connection $\nabla_{\sigma}$ on $L_\sigma$ over $X_\sigma$ equals $\nabla_{L_\sigma}^{\rm u};$}
\item[(2)] \emph{the class
$\hat{c}_1^H(\overline{L}):=\hat{c}_1^H(X/ \Spec K, \overline{L})$,
or in other words
the arithmetic Atiyah class
$\widehat{\rm at}_{X/K}(\overline{L})$, 
 vanishes in 
$\hat{H}^{1,1}(X/K):= \widehat{\Ext}^1_X(\cO_X, \Omega^1_{X/K})$;} 
\end{enumerate}

Observe also that, when $L$ is algebraically equivalent to zero,   the pair 
$(L_\C,\nabla_{\overline{L}}^{1,0})$ --- or equivalently the family $(L_\sigma, \nabla^{\rm u}_{L_\sigma})_{\sigma \in \Sigma}$ --- determines a
point $P=P_{\overline{L}}$ in the maximal compact subgroup of
\[
E_{X/K}(\R):=\bigl[\coprod\limits_{\sigma\in \Sigma}E_{X/K}(\C)\bigr]^{F_\infty}.
\]
(details of this construction may be found in the Appendix in 
B.7 and B.8), and Conditions (1) and (2) are also equivalent to:

\begin{enumerate}
\item[(3)] \emph{the point 
$P_{\overline{L}}$ in the maximal compact subgroup of $E_{X/K}(\R)$ is the image of a $K$-rational point of $E_{X/K}.$}
\end{enumerate}

We claim that, \emph{if a line bundle  
 $L$ over $X$ defines a torsion point in $\Pic(X),$ then Conditions \emph{(1)}  and \emph{(2)} are  
satisfied.}

Indeed, if $n$ is a positive integer and 
$\alpha: \cO_X \rightarrow L^{\otimes n}$ is an isomorphism of line bundles 
over $X$, we may introduce the connection $\nabla_L^{\rm tor}$ on $L$, defined over 
$K$, such that the connection $\nabla^{\rm tor}_{L^{\otimes n}}$ on $L^{\otimes n}$ 
deduced from $\nabla^{\rm tor}_L$ by taking its $n$-th tensor power makes $\alpha$ an 
isomorphism of line bundles with connections from $(\cO_X, d)$ to 
$(L^{\otimes n},\nabla_{L^{\otimes n}})$\footnote{More generally, for any two line bundles $L$ and $M$ over $X$, any connection $\nabla_M$ on $M$ and any isomorphism $\alpha: M \to L^{\otimes n}$, there exists a unique connection $\nabla_L$ on $L$ such that
the connection $\nabla_{L^{\otimes n}}$ on $L^{\otimes n}$ 
deduced from $\nabla_L$ by taking its $n$-th tensor power makes $\alpha$ an 
isomorphism of line bundles with connections from $(M, \nabla_M)$ to 
$(L^{\otimes n},\nabla_{L^{\otimes n}}).$ It may be defined by the following identity, valid for any local regular section $l$ of $L$: $n. l^{\otimes n-1}\otimes \nabla_L l= (\alpha \otimes {\rm Id_{\Omega^1_{X/K}}})\nabla_M( \alpha^{-1}(l^{\otimes n})).$ }. 
For any $\sigma$ in $\Sigma,$ the two connections $\nabla^{\rm tor}_{L,\sigma}$ and $\nabla^{\rm u}_{L_\sigma}$ on $L_\sigma$ coincide, since the monodromy of $\nabla^{\rm tor}_{L,\sigma}$ lies in the $n$-th roots of unity. Consequently \emph{Condition (1) is satisfied by $\nabla:= \nabla_L^{\rm tor}.$}

\subsubsection{}\label{enfinuntheoreme}

It turns out that, \emph{conversely, if Conditions (1) and (2) hold, 
then $L$ has a torsion class in $\Pic(X)$ and the connection $\nabla$, uniquely defined by (1), necessarily coincides with $\nabla_L^{\rm tor}$.} 
This is basically the content of Theorems \ref{thI} and \ref{thIalternative} 
when $X$ is an abelian variety and $\Sigma$ has one or two conjugate elements. 
It holds more generally for any $X$ as above:
 
\begin{theorem}\label{eqcond}
Let $X$ be a smooth, projective, geometrically connected variety over $K$,   
and let $\pi:X\to \Spec \,K$ its structural morphism, that we consider as a 
morphism of arithmetic schemes over the arithmetic 
ring $(K,\Sigma,F_\infty)$.

\emph{(i)} Let $\ol{L}= (L, \|.\|_L)$ be  a hermitian line bundle 
 over $X$. If $L$ admits an algebraic connection  $\nabla: L \rightarrow L \otimes \Omega^1_{X/K}$ such 
that $\nabla_{\C}$ is unitary with respect to $\|.\|_{L},$ then $L$ 
has a torsion class in ${\rm Pic}(X)$, the metric $\|.\|_L$ has vanishing 
curvature, and $\nabla$ coincides with $\nabla^{\rm tor}_L$.

\emph{(ii)}
For any hermitian line bundle $\ol{L}$ on $X$, if the first Chern class 
$\hat{c}_1^H(\ol{L})$ in $\hat{H}^{1,1}(X/K):= \widehat{\Ext}^1_X(\cO_X, \Omega^1_{X/K})$
vanishes, then there exists a positive integer $n$
such that $\ol{L}^{\otimes n}$ is isometric to the trivial bundle
$\cO_X$ equipped with a metric constant on every component $X_\sigma(\C)$ 
of $X_\Sigma(\C)$ --- or equivalently, such that the class 
of $\ol{L}^{\otimes n}$ in $\widehat{\rm Pic}
(X)$ belongs to the image of $\pi^\ast: \widehat{\rm Pic} 
(\Spec K) \rightarrow \widehat{\rm Pic}
(X).$

\emph{(iii)}
Let $P\in E_{X/K}(K)$ be a 
$K$-rational point of the universal vector extension $E_{X/K}$ that
belongs to the maximal compact subgroup of 
$E_{X/K}(\mathbb{R})$. 
Then $P$ is a torsion point in $E_{X/K}(K)$. 

\end{theorem}

\proof 


We prove below that the assertions (i)--(iii) are equivalent for any 
given variety $X$ as above.
The isomorphism (\ref{bidualisom}) will then show that it is sufficient to show
(iii), hence any of the assertions (i)--(iii), for abelian varieties.
In order to prove i), we may choose
 $\sigma$ in $\Sigma$ and replace  the set of embeddings
$\Sigma$ by $\{\sigma \}$  (resp.
$\{\sigma,\overline{\sigma}\}$) if $\sigma$ is a real (resp.
 complex) embedding. 
In this situation,
(i) has been proved for abelian varieties as Theorem \ref{thI} in Section \ref{trans} \emph{supra}.

For any given hermitian line bundle $\overline{L}$, the equivalence of the implications in
(i) and (ii) is a straightforward consequence of the observations  in \ref{preremsuite}
and of the implication $$\hat c_1^H (\ol{L})= 0 \Rightarrow c_1(\ol{L})=0,$$
which follows from
 the commutativity of (\ref{arfcc2}).


To establish the implication (ii) $\Rightarrow $ (iii),  
consider  $P$ in $E_{X/K}(K)$ a 
$K$-rational point of the universal vector extension that
belongs to the maximal compact subgroup of 
$E_{X/K}(\mathbb{R})$.
Replacing $K$ by a finite extension, we may assume that $P$ is represented
by a line bundle $L$ algebraically equivalent to zero
with an integrable connection $\nabla$.
If $P$ belongs to the maximal compact subgroup of $E_{X/K}(\mathbb{R})$,
we have $\nabla_\C=\nabla_{\overline{L}}^{1,0}$ where $\overline{L}$ carries
a hermitian metric with curvature zero.
As observed  in \ref{preremsuite} above,  this implies that $\hat{c}_1^H(\overline{L})=0$. 
According to (ii), there exists some integer $m>0$ such that $\overline{L}^{\otimes m}$ is isometric
to the trivial bundle $\cO_X$ with a constant metric.
It follows that $(L,\nabla)^{\otimes m}$ is isomorphic to the trivial 
bundle $\cO_X$ with the trivial connection, and consequently that $P$ belongs to the $m$-torsion of $E_{X/K}(K)$.

Finally, we show the implication (iii) $\Rightarrow $ (ii).
Let $\ol{L}= (L, \|.\|_L)$ be a hermitian line bundle over $X$ such that the class
$\hat{c}_1^H(\overline{L}):=\widehat{\rm at}_{X/K}(\overline{L})$ vanishes. Then ${\rm at}_{X/K}(L)$ vanishes too, and there exists a positive integer $m$ such that $L^{\otimes m}$ is algebraically equivalent to zero. By replacing $\overline{L}$  by $\overline{L}^{\otimes m},$ we may therefore assume that $L$ is algebraically equivalent to zero. As observed in \ref{preremsuite}, the point $P_{\overline{L}}$ associated to $(L_\C, \|.\|_L)$ lies in the maximal compact group of $E_{X/K}(\R)$, and is the image of a $K$-rational point of $E_{X/K}.$ According to (iii), it is a torsion point. This implies that $L$ has a torsion class in $\Pic (X)$, and that $\nabla_{\overline{L}}$ coincides with the connection $\nabla_{L,\C}^{\rm tor}.$ This establishes that $\overline{L}$ satisfies the conclusion of (i), and consequently, as observed above, of (ii).
\qed

\subsection{Finiteness results on the kernel of $\hat{c}^H_1$}
We may use Theorem \ref{eqcond} to investigate the kernel of the first Chern class in arithmetic Hodge cohomology.  Indeed  this Theorem 
easily leads to a derivation of the assertion $\mathbf{I2}_{X,\Sigma}$ in the Introduction (which conversely contains Part (ii) of Theorem \ref{eqcond}):

\begin{corollary}\label{findex}
The image of
$$\pi^\ast: \widehat{\rm Pic} 
(\Spec K) \longrightarrow \widehat{\rm Pic}
(X)$$ 
has finite index in the kernel
of
\[
\hat c_1^H:\widehat{\rm Pic} (X) \longrightarrow
\widehat{H}^{1,1}(X/K).
\]
\end{corollary} 

\proof
A hermitian metric with curvature zero on the trivial line bundle on $X$ is constant  on every component $X_\sigma(\C)$ of $X_\Sigma(\C).$
Therefore, if we introduce the canonical map 
\[
w:\widehat{\rm Pic} (X)\rightarrow {\rm Pic} (X)\hookrightarrow {\rm Pic}_{X/K}(K), 
\]
then we have:
\[
\Ker(\hat{c}_1^H)\cap \Ker(w)= {\Im} \bigl(\pi^*\colon \widehat{\rm Pic} (S){\rightarrow}
\widehat{\rm Pic} (X)\bigr).
\]
Hence the map $w$ induces an injection of 
\begin{equation}\label{arhomfin}
\frac{
\Ker\bigl({\hat c^H_1}\colon \widehat{\rm Pic} (X){\longrightarrow} 
\widehat{\rm Ext}_X^1(\mathcal{O}_X,\Omega^1_{X/K})\bigr)}
{\Im \bigl(\pi^*\colon \widehat{\rm Pic} (\Spec \, K){\longrightarrow}
\widehat{\rm Pic} (X)\bigr)}
\end{equation}
into ${\rm Pic}_{X/K}(K)$.
Theorem \ref{eqcond} (iii) implies that the image of (\ref{arhomfin}) is contained in the torsion subgroup
of ${\rm Pic}_{X/K}(K)$\footnote{Actually this morphism factorizes through the torsion subgroup ${\rm Pic}(X)_{\rm tor}$ of ${\rm Pic}(X)$, and one may easily show that the so defined injection $\Ker (\hat c^H_1)/\Im (\pi^\ast) \rightarrow {\rm Pic}(X)_{\rm tor}$ is an isomorphism.}. This is a finite group as the N\'eron-Severi group 
\[
NS_{X/K}(\overline{K})
={\rm Pic}_{X/K}(\overline K)/{\rm Pic}^0_{X/K}(\overline K)
\] 
and ${\rm Pic}^0_{X/K}(K)$ are finitely 
generated abelian groups by \cite[Th. 5.1]{kleiman71} and the theorem of 
Mordell-Weil.
\qed

We may also establish a similar finiteness result where the base scheme $\Spec K$ is replaced by an ``arithmetic curve":
 
\begin{corollary} 
Let $\cO_K$ denote the ring of integers in a number field $K$, and let us work 
 over the arithmetic ring  
$(\cO_K,\Sigma,F_\infty).
$
Let $S$ denote a non-empty open subset  of ${\rm Spec}\,\cO_K$, and let
$X$ be a smooth projective $S$-scheme with geometrically connected fibers.
Then
\begin{equation}\label{arhomfin2}
\frac{
\Ker\bigl({\hat c^H_1}\colon \widehat{\rm Pic} (X){\longrightarrow} 
\widehat{\rm Ext}^1_X(\mathcal{O}_X,\Omega^1_{X/S})\bigr)}
{\Im \bigl(\pi^*\colon \widehat{\rm Pic} (S){\longrightarrow}
\widehat{\rm Pic} (X)\bigr)}
\end{equation}
is a finite group.
\end{corollary}

\proof
Let $X_K$ denote the fiber of $X$ over ${\rm Spec}\,K$.
We consider $X_K$ as an arithmetic scheme over the 
arithmetic field $K=(K,\Sigma,F_\infty)$.
There is a canonical restriction map
\[
\nu:\widehat{\rm Pic} (X)\longrightarrow \widehat{\rm Pic} (X_K).
\]
Any element in $\Ker\,\nu\cap \Ker \,\hat{c}_1^H(X/S,\, .\,)$
is generically trivial and carries a constant metric. 
The sequence
\[
{\rm Pic}(S)\longrightarrow {\rm Pic}(X)\longrightarrow {\rm Pic}(X_K)
\]
is exact as the fibers of $X/S$ are integral.
Hence
\[
\Ker\,(\nu)\cap \Ker \,\hat{c}_1^H(X/S,\, .\,)\subseteq
{\Im\, \bigl(\pi^*\colon \widehat{\rm Pic} (S){\longrightarrow}
\widehat{\rm Pic} (X)\bigr)}.
\]
Moreover $\nu$ maps $\Im \bigl(\pi^*\colon \widehat{\rm Pic} (S){\longrightarrow}
\widehat{\rm Pic} (X)\bigr)$ onto $\Im \bigl(\pi^*\colon \widehat{\rm Pic} (\Spec K){\longrightarrow}
\widehat{\rm Pic} (X_K)\bigr)$. Consequently it induces an embedding of (\ref{arhomfin2}) into (\ref{arhomfin}).
The latter group is finite by Theorem \ref{eqcond}. 
Our claim follows.
\qed








\section{A geometric analogue}\label{sect4}

\subsection{Line bundles with vanishing relative Atiyah class on 
fibered projective varieties}

\subsubsection{Notation}

 In this section, we consider a smooth projective geometrically connected curve $C$ over a field $k$ of characteristic $0$, and a smooth projective 
 variety $V$ over $k$ equipped with a dominant $k$-morphism $\pi:V \rightarrow C$, with geometrically connected fibers.
 
 Observe that the morphism $\pi$ is flat, and smooth over 
an open dense subscheme of $C$, namely over the complement of the finite set $\Delta$  of closed points $P$ 
in $C$ such that the (scheme theoretic) fiber $\pi^\ast(P)$ is not smooth over $k$. 
 
 Let $K:=k(C)$ denote the function field of $C$. The generic fiber $V_K$ of $\pi$ is a smooth projective geometrically connected variety over $K.$ Conversely, according to Hironaka's resolution of singularities, any such variety over $K$ may be constructed from the data of a $k$-variety $V$ and of a $k$-morphism $\pi:V \rightarrow C$ as above.
 
 Recall also that a divisor $E$ in $V$ is called \emph{vertical} if it belongs to 
the group of divisors generated by components of closed fibers of $\pi$, or 
equivalently, if its restriction $E_K$ to the generic fiber $V_K$ of $V$ vanishes.

In the sequel, we  assume that the dimension $n$ of $V$ is at 
least $2$. Moreover  we  choose an ample line bundle  $\cO(1)$ over $V$,
we  denote $H$ its first Chern class in the Chow group $CH^1(X),$ and for any integral subscheme $D$ of positive dimension in $V$ and any line bundle $L$ over $V$,  we let:
$$\deg_{H,D} L:= \deg_k (c_1(L). H^{\dim D -1}.[D]).$$

Actually, we shall use this definition only when $D$ is a vertical 
divisor in $V$. Consequently, we could require $\cO(1)$ to be ample 
relatively to $\pi$ only, and when $n=2$ the choice of $\cO(1)$ is 
immaterial.

Observe that, if $\cO(1)$ is very ample and defines a projective embedding $\iota: V \hookrightarrow \P^N_k$, then, for any general enough $(\dim D -1)$-tuple $(H_1,\ldots,H_{\dim D -1})$ of projective hyperplanes in $\P^N_k$, the subscheme
$$C:= D \cap \iota^{-1}(H_1) \cap\ldots\cap\iota^{-1}(H_{\dim D -1})$$
in $\P^N_k$ is integral,  one-dimensional, and projective  over $k$, and its class $[C]$ in $CH_1(X)$ coincides with $H^{\dim D -1}.[D]$. Consequently  $\deg_{H,D} L$ is nothing but the degree $\deg_k c_1(L).[C]$ of the restriction of $L$ to the  ``general linear section" $C$ of $D$. 

Let us recall that, if $M$ is a smooth projective geometrically connected scheme over some field $k_0$ of characteristic zero, then the Picard functor $\Pic_{M/k_{0}}$ is representable by a separated group scheme over $k_0$, and that its identity component $\Pico_{M/k_{0}}$ is an abelian variety over $k_0$. A line bundle $L$ over $M$ is algebraically equivalent to zero\footnote{The reader should beware that, here as in the previous section, we use a ``geometric" definition of ``algebraically equivalent to zero", related as  follows to the one occuring in \cite{FultonIT}, 10.3: for any divisor $D$ in $M$ and any algebraic closure ${\overline{k}_0}$ of $k_0$, the line bundle $\cO(D)$ is algebraically equivalent to zero in our ``geometric"  sense  iff the divisor $D_{\overline{k}_0}$ on $M_{\overline{k}_0}$ is algebraically equivalent to zero in Fulton's sense. Also observe that (the first Chern class of) a line bundle on $M$  algebraically equivalent to zero in the above sense is numerically equivalent to zero in the sense of Fulton \cite{FultonIT}, 19.1.  In particular, with the notation of the previous paragraphs, for any line bundle $L$ algebraically equivalent to zero over $V$, $\deg_{H,D} L$ vanishes.} when the point in $\Pic_{M/k_{0}}(k_0)$ it defines belongs to  $\Pico_{M/k_{0}}(k_0)$, or equivalently, if its class in the N\'eron-Severi group of $M$ over $k_0$ --- defined as $\Pic_{M/k_{0}}(k_0)/\Pico_{M/k_{0}}(k_0)$ --- vanishes.

In particular, we may consider   the identity component  $\Pico_{V_K/K}$ of the Picard variety of the generic fiber $V_K$ of $\pi$; it is an abelian variety over $K$, and we shall denote $(B, \tau)$ its $K/k$-trace. By definition, $B$ is an abelian variety over $k$, and $\tau$ is a morphism of abelian varieties over $K$:
$$\tau: B_K \longrightarrow \Pico_{V_K/K}.$$
Since the base field $k$ is assumed to be of characteristic zero, this morphism is actually a closed immersion. We refer the reader to
Section \ref{Picard} \emph{infra} for a discussion and references concerning the definition of   $\Pico_{V_K/K}$ and $(B,\tau).$
 
\subsubsection{} The following theorem may be seen as a geometric counterpart, valid 
 over the function field $K:=k(C),$ of the characterization of hermitian
line bundles with {\bf v}anishing arithmetic {\bf A}tiyah class in Theorem \ref{eqcond} ii).

\begin{theorem}\label{MainGeom}
With the above notation, for any line bundle $L$ over $V$, the following 
three conditions are equivalent:

{\bf VA1} The relative Atiyah class $\at^1_{V/C}(L)$ vanishes in
\[
\Ext_{\mathcal{O}_X}^1(L,L\otimes \Omega_{V/C}^1)\simeq
H^1(V,\Omega_{V/C}^1).
\]

{\bf VA2} There exist a positive integer $N$ and a line bundle $M$ over $C$ 
such that the line bundle $L^{\otimes{N}} \otimes \pi^\ast M$ 
is algebraically equivalent to zero.

{\bf VA3} There exists a positive integer $N$ such that the line bundle $L_K^{\otimes N}$ on $V_K$ 
is algebraically equivalent to zero, and the attached $K$-rational point of the Picard variety 
$\Pico_{V_K/K}$
is defined by a $k$-rational point of the $K/k$-trace of $\Pico_{V_K/K}$. Moreover, for any component $D$ of a closed fiber of $\pi,$ the degree
$\deg_{H,D}L$ vanishes. 



\end{theorem}

Observe that, for any closed point $P$ of $C\setminus \Delta,$ its fiber $D:= \pi^\ast (P)$ is a divisor in $V$, smooth and geometrically connected over $k(P)$, and that, according to the projection formula, 
\begin{align*} 
\deg_{H,D}L & = \deg_k (c_1(L).H^{n-2}.[\pi^\ast(P)]) \\
& = \deg_k (\pi_\ast(c_1(L).H^{n-2}).[P])\\
& = [k(P):k]. \deg_K (c_1(L_K).c_1(\cO(1)_K)^{\dim V_K -1}.[V_K]).
\end{align*}
In particular, if some positive power $L_K^{\otimes N}$ of $L_K$ is algebraically equivalent to zero, then $\deg_{H,D}L$ vanishes.
Consequently, in condition {\bf VA3}, we may 
require the vanishing of $\deg_{H,D}L$ only for  components $D$ of 
the supports of the singular fibers $\pi^\ast(P),$ where  $P$ varies in $\Delta.$

The proof of the equivalence of conditions {\bf VA1} and {\bf VA2}, which uses the Hodge index theorem and basic 
properties of Hodge cohomology groups, will be presented in 
Sections \ref{HIT} and \ref{VA12} below. Then in Section 
\ref{Picard} and \ref{VA23} we shall recall some classical facts 
concerning the Picard variety $\Pico_{V_K/K}$ and its $K/k$-trace, 
and establish the equivalence of conditions {\bf VA2} and {\bf VA3}.

\subsection{Variants and complements}
Before we enter into the proof of Theorem \ref{MainGeom}, we discuss some
variants and related statements. Observe that the variants in \ref{varnfp}
make Theorem \ref{MainGeom} more similar to its "arithmetic counterpart" in
Theorem \ref{eqcond} ii), whereas Proposition \ref{affineva} would rather make less convincing
the analogy between the arithmetic framework in Section \ref{sect3}
and the geometric framework of the present section.

\subsubsection{}\label{varnfp} Recall that the following conditions are equivalent 
--- when they  hold, the Picard variety 
$\Pico_{V_K/K}$ will be said to have \emph{no fixed part}:

{\bf NFP1}\emph{ The $K/k$-trace of $\Pico_{V_K/K}$ vanishes,} or in other terms, \emph{for any abelian variety $A$ over $k$, there is no non-zero morphism of abelian varieties over $K$ from $A_K$ to $\Pico_{V_K/K}$.}

{\bf NFP2}\emph{ The morphism of $k$-abelian varieties naturally deduced 
from} $\pi: V \longrightarrow C$
$$\pi^\ast:\Pico_{C/k} \longrightarrow \Pico_{V/k}$$
 --- which has a finite kernel --- \emph{is an isogeny.}

{\bf NFP3}\emph{ The injective morphism of $k$-vector spaces
$$\pi^\ast: H^1(C, \cO_C) \longrightarrow H^1(V,\cO_V)$$
is an isomorphism.}

{\bf NFP4}\emph{ The injective morphism of $k$-vector spaces
$$\pi^\ast:H^0(C, \Omega^1_{C/k}) \longrightarrow H^0(V,\Omega^1_{V/k})$$
is an isomorphism.}

A few comments on these conditions may be appropriate.

The injectivity of $\pi^\ast$ in {\bf NFP2} may be derived by considering a smooth projective geometrically connected curve $C'$ in $V$ such that the morphism $\pi_{\mid C'}: C' \rightarrow C$ is finite. Let $i:C' \hookrightarrow V$ denote the inclusion morphism. The norm with respect to $\pi_{\mid C'}$ defines a morphism 
$\pi_{\mid C'\ast}: \Pico_{C'/k} \rightarrow \Pico_{C/k}$ of abelian varieties over $k$,
and the morphisms of abelian varieties $\pi^\ast$, $\pi_{\mid C'\ast}$, 
$\pi_{\mid C'}^{\ast}: \Pico_{C/k} \rightarrow \Pico_{C'/k},$  and $i^\ast: \Pico_{V/k} \rightarrow \Pico_{C/k}$ satisfy the relations
$$\pi_{\mid C'}^{\ast} =i^\ast \circ \pi^\ast$$
and $$\pi_{\mid C'\ast} \circ \pi_{\mid C'}^{\ast} = [\delta],$$
where $[\delta]$ denotes the morphism of multiplication by the degree $\delta$ of $\pi_{\mid C'}$ in $\Pico_{C/k}.$ This immediately implies that the kernel of $\pi^\ast$ is a subgroup of the $\delta$-torsion in $\Pico_{C/k}.$ 

The injectivity of $\pi^\ast$ in {\bf NFP4} is a consequence of the generic smoothness of the dominant morphism $\pi$ (recall that the base field $k$ is assumed to have characteristic zero). The injectivity of $\pi^\ast$ in {\bf NFP3} and the
equivalence of {\bf NFP3} and 
{\bf NFP4} follows from Hodge theory when $k=\C$, and therefore, by a 
standard base change argument, for any base field $k$ of characteristic  
zero.

The equivalence of {\bf NFP1} and {\bf NFP2} follows from the 
description of the $K/k$-trace of $\Pico_{V_K/K}$ recalled in
Proposition \ref{Kktrace} below. Finally, the equivalence of {\bf NFP2} and 
{\bf NFP3} follows from the identification of $H^1(C, \cO_C)$ (resp. 
$H^1(V,\cO_V)$) with $\Lie\,\Pico_{C/k}$ (resp. $\Lie\,\Pico_{V/k}$).

As demonstrated by the theorem of Mordell-Weil-Lang-N\'eron, it is natural to require a no fixed part condition when searching for statements valid over function fields that are as close as possible to their  arithmetic counterparts. This is indeed the case with Theorem \ref{MainGeom}. Namely, when $\Pico_{V_K/K}$ has no fixed part,  Conditions {\bf VA1-3}  are also equivalent to the following ones, which look more closely like the conditions appearing in i) and ii) of the ``arithmetic" Theorem \ref{eqcond}:

{\bf VA2'} \emph{
There exists a positive integer $N$ and a line bundle $M$ over $C$ 
such that the  line bundle $L^{\otimes N}$ is isomorphic to $\pi^\ast M.$}

{\bf VA3'} \emph{ The  class of $L_K$ in the  abelian group 
$\Pic_{V_K/K}(K)$  is torsion. Moreover, for any component $D$ of a closed fiber of $\pi,$ the degree
$\deg_{H,D}L$ vanishes. 
}

Indeed, the equivalence of {\bf VA3} and {\bf VA3'} when {\bf NFP1} holds is straightforward, and the equivalence of {\bf VA2} and {\bf VA2'} easily follows from {\bf NFP2}.

\subsubsection{} Generalizations of Theorem \ref{MainGeom} concerning a smooth projective variety $V$ over $k$ fibered over a projective variety $C$ of dimension $>1$ may be deduced from its original version with $C$ a curve by means of standard techniques, as in the proof of the Mordell-Weil-Lang-N\'eron theorem (\cf \cite{LangNeron59}). We leave this  to the interested reader.

\subsubsection{}
Finally observe that when the base $C$ is assumed to be affine instead of
projective, the determination of line bundles with vanishing relative
Atiyah class becomes a rather straightforward issue. For instance, we have:

\begin{proposition}\label{affineva}
Let $C$ be an affine integral scheme of finite type over a field $k$ of 
characteristic zero, and let $K:=k(C)$ denote its function field.
Let $\pi:V\rightarrow C$ be a smooth projective morphism, $L$ a line bundle 
over $V$, and $L_K$ the restriction of $L$ to the generic fibre $V_K$ of $\pi$.
The following conditions are equivalent:
\begin{enumerate}
\item[(i)]
the relative Atiyah class $\at_{V/C}^1(L)$ vanishes in
\[
\Ext^1_{\mathcal{O}_V}(L,L\otimes \Omega_{V/C}^1)\simeq H^1(V,\Omega_{V/C}^1);
\]
\item[(ii)]
the Atiyah class $\at_{V_K/K}(L_K)$ vanishes in $H^1(V_K,\Omega_{V_K/K})$;
\item[(iii)]
some positive power of $L_K$ is algebraically equivalent to zero over $V_K$.
\end{enumerate}
\end{proposition}

\proof
The equivalence (i)$\Leftrightarrow$ (ii) follows from the identification
\[
H^1(V,\Omega_{V/C}^1)\simeq H^0(C,R^1\pi_*\Omega_{V/C}^1)
\]
and from the fact that, since the base field has characteristic zero,
by Hodge theory the coherent sheaf $R^1\pi_*\Omega_{V/C}^1$ is a locally
free sheaf over $C$, the formation of which is actually compatible
with any base change.

The equivalence (ii)$\Leftrightarrow$ (iii) holds since the base field $K$
has characteristic zero (see for instance \ref{fcihc} below). 
\qed

\subsection{Hodge cohomology and first Chern class}\label{Hodge}

In this section, we review some basic properties of the Hodge cohomology of smooth projective varieties over fields of characteristic zero.  These properties are consequence of the duality theory for coherent sheaves on projective varieties, as explained in \cite{FGA}, expos\'e 149.

\subsubsection{Hodge cohomology groups}

Let $k$ be a field of characteristic  zero, and $\Sm_k$ the full subcategory 
of the category of $k$-schemes whose objects are smooth projective schemes 
$V$ over $k$.

To any object $V$ in $\Sm_k$ are attached his \emph{Hodge cohomology groups}:
\[ H^{p,q}(V/k):=H^q(V, \Omega^p_{V/k}).\]
These are finite dimensional $k$-vector spaces, and vanish if 
$\max (p,q) >d:=\dim V.$ Moreover,  the cup products
\[
\begin{array}{rcl}
 H^{p,q}(V/k) \times H^{p',q'}(V/k) & \longrightarrow   & H^{p+p',q+q'}(V/k)   \\
(\alpha, \alpha')  & \longmapsto   & \alpha.\alpha',  
\end{array}
\]
--- defined as the compositions of the products
\[H^q(V, \Omega^p_{V/k}) \times  H^{q'}(V, \Omega^{p'}_{V/k}) \longrightarrow
H^{q+q'}(V, \Omega^{p}_{V/k}\otimes\Omega^{p'}_{V/k})\] and of the mappings
\[ H^{q+q'}(V, \Omega^{p}_{V/k}\otimes\Omega^{p'}_{V/k})
\longrightarrow H^{q+q'}(V, \Omega^{p+p'}_{V/k})\]
deduced from the exterior product
$\wedge:  \Omega^{p}_{V/k}\otimes\Omega^{p'}_{V/k}
\longrightarrow \Omega^{p+p'}_{V/k}$ ---
make the direct sum $H^{\ast,\ast}(V/k):=\bigoplus_{(p,q)\in \N^2} H^{p,q}(V/k)$ 
a bigraded commutative\footnote{Namely, for any $\alpha$ (resp. $\alpha'$) in 
$H^q(V, \Omega^p_{V/k})$ (resp.
in $H^{q'}(V, \Omega^{p'}_{V/k})$), we have 
$\alpha.\alpha'=(-1)^{pp'+qq'} \alpha'.\alpha.$}  $k$-algebra.


Moreover, the ``top-dimensional" Hodge cohomology group $H^{d,d}(V/k)$ is 
equipped with a canonical $k$-linear form:
\[\int_{V/k}. : H^{d,d}(V/k) \longrightarrow k,\]
and the attached $k$-bilinear map 
\[
\begin{array}{rrcl}
<.,.> :  & H^{\ast,\ast}(V/k) \times  H^{\ast,\ast}(V/k) & \longrightarrow & k   \\
  & (\alpha, \beta) & \longmapsto   & \int_{V/k} \alpha.\beta    
\end{array}
\]
is a perfect pairing.

In particular, when $V$ is  a geometrically connected $k$-scheme, or equivalently 
when the linear map
\[
\begin{array}{ccl}
 k & \longrightarrow   &  \Gamma(V, \cO_V) = H^{0,0}(V/k) \\
  \lambda & \longmapsto   & \lambda.1_V   
\end{array}
\]
 is an isomorphism, then the ``residue map" also is:
 \[ \int_{V/k}. :  H^{d,d}(V/k) \stackrel{\sim}{\longrightarrow} k.\]
 Then we denote $\mu_V$ the unique element in $H^{d,d}(V/k)$ such that
 \[\int_{V/k} \mu_V=1.\]

These constructions are compatible in an obvious sense with extensions of the base field $k$. Let us also indicate that, when $k=\C,$ the trace map 
\[\int_{V/\C}. : H^{d,d}(V/\C) \longrightarrow \C \]
satisfies the following compatibility relation with the Dolbeault isomorphism
$${\rm Dolb}_{\Omega^d_{V/\C}}:H^d(V,\Omega^d_{V/\C}) \longrightarrow H_{\rm{Dolb}}^d(V,\Omega^d_{V/\C})$$
(we follow the notation of \cite{bostkuennemann1}, A.5.1) and the integration of top degree forms:
$$\int_{V(\C)}. : A^{d,d}(V(\C)) \longrightarrow \C.$$
For any $\alpha$ in $A^{d,d}(V(\C))$, of class $[\alpha]$ in $H_{\rm{Dolb}}^d(V,\Omega^d_{V/\C})$, we have: 
$$\int_{V/\C} {\rm Dolb}_{\Omega^d_{V/\C}}^{-1}([\alpha]) = \varepsilon_d \frac{1}{(2\pi i)^d} \int_{V(\C)} \alpha,$$
where $\varepsilon _d$ denotes a sign, function of $d$ only, depending on the sign conventions followed  in duality theory (we refer the reader to \cite{deligne84}, Appendice, and \cite{sastrytong03} for discussions of this delicate issue).

\subsubsection{The first Chern class in Hodge cohomology}\label{fcihc}
 Any line bundle $L$ over some $V$ in $\Sm_k$ admits a first Chern class 
$c_1(L)$ in $H^{1,1}(V/k).$ It may be defined as the class 
$$\at_{X/k} L= \jet^1_{X/k} L$$
in
 \begin{align} \Ext^1_{{\cO_V}}(L, \Omega^1_{V/k}\otimes L) & \simeq
     \Ext^1_{{\cO_V}}(\cO_{V}, \Omega^1_{V/k} \label{tensL})\\
     & \simeq H^1(V, \Omega^1_{V/k}) \label{extiso}. 
     \end{align}
of the extension given by the principal parts of first order associated
with $L$
 \[\cjet^1_{X/k} L: 0 \longrightarrow \Omega^1_{X/k} \otimes L
 \longrightarrow P^1_{X/k}(L)
 \longrightarrow L \longrightarrow 0\]
  (see Section 1.2 above). The isomorphism (\ref{tensL})
 is the (inverse of the) one defined by applying the functor $.\otimes
 L$  to complexes of $\cO_{V}$-modules, without intervention of signs.
 The isomorphism (\ref{extiso})
 is the one discussed in \cite{bostkuennemann1}, A.2 and A.4.)
 
 The so-defined first Chern class defines a morphism of abelian groups:
 \[
\begin{array}{rcl}
 \Pic (V) & \longrightarrow   & H^1(V, \Omega^1_{V/k}) =: H^{1,1}(V/k)   
 \\
\left[L\right]  & \longmapsto   &  c_1(L).  
\end{array}
\]
Moreover, this morphism factorizes through the N\'eron-Severi group
\[
\NS_{V/k}(k)={\rm Pic}_{V/k}(k)\big/{\rm Pic}^0_{V/k}(k);
\] 
the induced morphism on $\NS_{V/k}(k)$ vanishes precisely on its torsion subgroup $\NS_{V/k}(k)_{\rm tor}$
(compare for example \cite[II.2 Cor. 1 to Th. 2]{kleiman66}),
and consequently 
defines an injective morphism of groups
 \[c_1:  \NS_{V/k}(k) /\NS_{V/k}(k)_{\rm tor} \longrightarrow  H^{1,1}(V/k).\]
 In other words, for any line bundle $L$ on $V$, the following two conditions 
are equivalent:
 
 (i) the first Chern class $c_1(L)$ in $H^{1,1}(V/k)$ vanishes;
 
 (ii) for some positive integer $N$, the line bundle $L^{\otimes N}$ over $V$ 
is algebraically equivalent to zero.
 
 Let us also recall that the construction of the first Chern class in Hodge 
cohomology is compatible with pull-back by $k$-morphisms. It is also compatible 
with intersection theory. In particular, we have:
 
 \begin{proposition}\label{intinter}
 For any $d$-tuple $D_1, \ldots, D_d$ of divisors in some $d$-dimensional variety $V$ in $\Sm_k$, the following formula 
holds:
\begin{equation}\label{HodgeChow}
 \int_{V/k} c_1(\cO(D_1)).\cdots.c_1(\cO(D_d))=\deg_k ([D_1].\cdots. [D_d]),
 \end{equation}
 where $[D_i]$ denotes the class of $D_i$ in the Chow group 
$CH^1(V),$ $[D_1].\cdots. [D_d]$ their product in $CH^d(V)=CH_0(V)$ and
 \[\deg_k: CH_0(V) \stackrel{\pi_\ast}{\longrightarrow} CH_0(\Spec k) \simeq \Z\] 
the degree map, attached to the structural morphism $\pi : V \rightarrow \Spec k$ of $V.$ 
 
 In particular, if $d=1$ and $V$ is geometrically irreducible, then 
 \[c_1(\cO(D))= \deg_k D. \mu_V.\]
\end{proposition} 
 
To establish the equality (\ref{HodgeChow}), one easily reduces to the case where $k$ is algebraically closed and $V$ is connected. Then it follows from   \cite{FGA}, expos\'e 149 (Th\'eor\`eme 1,  Th\'eor\`eme 2, and its proof) when moreover the divisors $D_1,\ldots,D_n$ and their successive intersections $D_1 \cap D_2,$ $D_1\cap D_2 \cap D_3,$$\ldots, D_1 \cap D_2 \cap\cdots\cap D_n$ are smooth. Together with the invariance of both sides of  (\ref{HodgeChow}) by linear equivalence of $D_1,\ldots,D_n$ and Bertini theorem, this shows that (\ref{HodgeChow}) holds when $D_1,\ldots,D_n$ are very ample. The general case of (\ref{HodgeChow}) follows by multilinearity.

\subsection{An application of  the Hodge Index Theorem}\label{HIT}

Our proof of Theorem \ref{MainGeom} will rely on an application of 
Hodge Index Theorem to projective varieties fibered over curves that we discuss 
in the present Section. 

\subsubsection{The Hodge Index Theorem in Hodge cohomology}

Let $V$ be a smooth, projective, geometrically connected scheme over $k$, and 
let $h$ be the first Chern class $c_1(\cO(1))$ in $H^{1,1}(V/k)$ of some ample 
line bundle $\cO(1)$ on $V.$ 

We shall use the following straightforward consequence  of the Hodge Index Theorem (as formulated in \cite{kleiman71}, Appendix 7) and of the compatibility of intersection theory and products in Hodge cohomology stated in  Proposition \ref{intinter}:

\begin{proposition}\label{PropHIT} When $d:= \dim V \geq 2,$ for any class $\alpha$ 
of $H^{1,1}(V/k)$ in the image of 
$c_1: \Pic (V) \rightarrow H^{1,1}(V/k),$ the following conditions are equivalent:

(i) $\alpha = 0;$

(ii) $\alpha^2.h^{d-2}= \alpha.h^{d-1}=0$ in  $H^{d,d}(V/k) \simeq k.$ 

\end{proposition}

\subsubsection{An application to projective varieties fibered over curves}\label{fibered}

We keep the notation of the previous paragraph, and assume that $d:= \dim V$ is 
at least $2$. Moreover, we consider a smooth geometrically connected 
projective curve $C$ over $k$, and a dominant $k$-morphism $\pi: V \rightarrow C.$ 
We shall denote $K$ the function field $k(C)$ of $C,$ $V_K:= V\times_C \Spec K$ the 
generic fiber of $\pi$, and $\cO(1)_K$ the pull-back of $\cO(1)$ to $V_K$.

Let us introduce the following class in
$H^{1,1}(V/k)$: 
\[F:= \pi^\ast \mu_C.\] 
Observe that $\mu_C^2=0$ for dimension reasons, and that consequently $F^2=0.$ Moreover Proposition \ref{intinter} and the naturality of $c_1$  show that, for any divisor $E$  on $C$, 
$$c_1(\cO(E))= \deg_k E\cdot \mu_C$$
and
\begin{equation}\label{EF}
 c_1(\cO(\pi^\ast(E)))= \deg_k E\cdot F.
\end{equation}

\begin{lemma}\label{integer} 1) For any divisor $D$ on $V$, 
${\int_{V/k} c_1(\cO(D)). h^{d-1}}$ coincides with the intersection number $\deg_k ([D]. [H]^{d-1})$, where $H$ denotes the divisor of some non-zero rational section of $\cO(1).$ In particular, it is an integer.

2) We have:
\[\int_{V/k} F.h^{d-1}= \deg_{\cO(1)_K} V_K.\]
In particular, the class $F$ is not zero, and the image of 
$\pi^\ast: H^{1,1}(C/k) \rightarrow H^{1,1}(V/k)$ is precisely the $k$-line $k.F.$
\end{lemma}

\begin{proof}

Assertion 1) is a special case of Proposition  \ref{intinter}. 

To establish 2), let us choose a divisor $E$ with positive degree on $C$. We have
\begin{equation}\label{unpeudinter} 
\deg_k ([\pi^\ast(E)].[H]^{d-1})= \deg_k ([E].\pi_\ast([H]^{d-1}))
= \deg_k E. \deg_{\cO(1)_K} V_K,
\end{equation}
by basic intersection theory. Besides, according to Proposition \ref{intinter} and (\ref{EF}),  the left-hand 
side of (\ref{unpeudinter}) is also equal to 
\[\int_{V/k}  c_1(\cO(\pi^\ast(E))). c_1(\cO(1))^{d-1}
= \deg_k E.\int_{V/k} F.h^{d-1}.\]
Together with  (\ref{unpeudinter}), this establishes the announced relation.
\end{proof}



\begin{proposition}\label{quelbeta} With the above notation, for any class $\beta$ of 
$H^{1,1}(V/k)$ in the image of $c_1,$ the following conditions are equivalent:

\emph{(i)} $\beta$ belongs to $\Q.F$;

\emph{(ii)} $\beta$ belongs to $k.F$;

\emph{(iii)} $\beta.\beta=\beta.F=0$ in $H^{2,2}(V/k);$

\emph{(iv)} $\beta^2. h^{d-2}=\beta.F.h^{d-2}=0$ in $H^{d,d}(V/k) \simeq k.$ 

\end{proposition}

\begin{proof} The implications (i)$\Rightarrow$(ii)$\Rightarrow$(iii)$\Rightarrow$(iv) 
are straightforward. To establish the converse implications, observe that
${\int_{V/k} \beta. h^{d-1}}/{\int_{V/k} F. h^{d-1}}$ is a well defined rational number by Lemma \ref{integer}, and consider the class 
\[\alpha:= \beta - \frac{\int_{V/k} \beta. h^{d-1}}{\int_{V/k} F. h^{d-1}}. F\]
in $H^{1,1}(V/k).$ It satisfies
$\alpha. h^{d-1}= 0$
by its very definition (recall that $\int_{V/k}$ maps isomorphically $H^{d,d}(V/k)$ onto $k$). 
Moreover (\ref{EF}) shows that some positive multiple of $\alpha$ lies in the image of $c_1$. Finally, when condition (iv) holds, then $\alpha$ also satisfies 
$\alpha^2. h^{d-2}=0.$
Then, according to Proposition \ref{HIT}, $\alpha$ vanishes, 
or equivalently:
\[\beta =\frac{\int_{V/k} \beta. h^{d-1}}{\int_{V/k} F. h^{d-1}}. F.\]
This establishes (i).
\end{proof}

\subsection{The equivalence of {\bf VA1} and {\bf VA2}}\label{VA12} 

We keep the notation 
of the previous paragraph \ref{fibered}. In other words, the same hypotheses 
as in Theorem  \ref{MainGeom} are supposed to hold, except the connectedness of 
the geometric fibers of $\pi$.

The following result contains the equivalence of Conditions  {\bf VA1} and {\bf VA2} in Theorem \ref{MainGeom}:

\begin{theorem}\label{MainBis}
For any line bundle $L$ over $V,$ the following conditions are equivalent:

\emph{(i)} The relative Atiyah class $\at_{V/C} L$ vanishes in 
$H^{1,1}(V, \Omega^1_{V/C})$.

\emph{(ii)'} $c_1(L)$ belongs to $\Q.F$.

\emph{(ii)''} There exists a positive integer $N$ and a line bundle $M$ over 
$C$ such that $c_1(L^{\otimes N} \otimes \pi^\ast M)$ vanishes.

\end{theorem}

\begin{proof} The equivalence (ii)' $\Leftrightarrow$ (ii)'' is straightforward. 

To establish the implication (ii)' $\Rightarrow$ (i),
consider the canonical exact sequence of sheaves of 
K\"ahler differentials on $V$,
\[
0 \longrightarrow \pi^\ast \Omega^1_{C/k}
\stackrel{i}{\longrightarrow}
\Omega^1_{V/k}
\stackrel{p}{\longrightarrow}
\Omega^1_{V/C}
\longrightarrow 0, 
\]
and the associated exact sequence of cohomology groups
\[H^1(V,\pi^\ast \Omega^1_{C/k})
\stackrel{H^1(i)}{\longrightarrow}
H^1(V,\Omega^1_{V/k})
\stackrel{H^1(p)}{\longrightarrow}
H^1(V,\Omega^1_{V/C}).\]
As a special case of Lemma \ref{pullbackconnection}, i), we have
\begin{equation}\label{atat}
\at_{V/C} L= H^1(p)(\at_{V/k} L).
\end{equation}
Since $F$ belongs to the image of $H^1(i)$, hence to the kernel of $H^1(p),$ this establishes the implication (ii)' $\Rightarrow$ (i).

The implication (i)$\Rightarrow$(ii)' will follow from 
 the implication (iii)$\Rightarrow$(i) in Proposition \ref{quelbeta} (applied to $\beta:=c_1(L)$) combined 
with the following:

\begin{lemma}\label{vanishing}
For any line bundle $L$ over $V$, if the relative Atiyah class $\at_{V/C} L$ 
vanishes in $H^1(V,\Omega^1_{V/C})$, then $c_1(L).F$ and $c_1(L)^2$ vanish in 
$H^2(V, \Omega^2_{V/k}).$
\end{lemma}

To establish this lemma, 
observe that the cup product
\begin{equation} \label{cupvanishes}
H^{1,1}(V/k) \otimes H^{1,1}(V/k) \longrightarrow H^{2,2}(V/k)
\end{equation}
vanishes on $\im H^1(i) \otimes \im H^1(i)$. Indeed the map of sheaves of 
$\cO_V$-modules defined as the composition
\[
\pi^\ast \Omega^1_{C/k} \otimes \pi^\ast \Omega^1_{C/k}
\stackrel{i\otimes i}{\longrightarrow}
\Omega^1_{V/k} \otimes \Omega^1_{V/k}
\stackrel{.\wedge.}{\longrightarrow}
\Omega^2_{V/k}
\]
vanishes by functoriality of the exterior product, since $\Omega^2_{C/k}=0$. 
This entails the vanishing of the cup product (\ref{cupvanishes}) on 
$\ker H^1(p) \otimes \ker H^1(p)$ and on $\ker H^1(p) \otimes \im \pi^\ast,$ 
where $\pi^\ast$ denotes the pull-back map in Hodge cohomology 
$\pi^\ast:H^{1,1}(C/k) \rightarrow H^{1,1}(V/k).$

According to (\ref{atat}), $\at_{V/C} L$ vanishes precisely 
when  $c_1(L)=\at_{V/k} L$
belongs to $\ker H^1(p)$, in which case $c_1(L)^2$ and  $c_1(L).F$  vanish in 
$H^2(V, \Omega^2_{V/k})$ by the observation above. This completes the proof of 
Lemma \ref{vanishing}, hence of Theorem \ref{MainBis}.
\end{proof}

\subsection{The Picard variety of a variety over a function 
field}\label{Picard}

In this paragraph, we recall some classical facts concerning the relations between the Picard varieties of $C$ and $V$, and the $K/k$-trace of the Picard variety of the generic fiber $V_K$ of $V$. (For modern presentations of Chow's classical theory of the $K/k$-trace of abelian varieties over $K$, we refer to \cite{conrad06} and Hindry's Appendix A in \cite{kahn06}.)

Let $(B,\tau)$ be the $K/k$-trace of $\Pico_{V_K/K}$. By construction, $B$ is an abelian variety over $k$, and $\tau$ is a morphism of abelian varieties over $K$
$$\tau: B_K \longrightarrow \Pico_{V_K/K}.$$ The pair $(B,\tau)$ is characterized by the following universal property: for any abelian variety $\tilde B$ over $k$ and any morphism of abelian varieties over $K$
$$\psi : \tilde B_K \longrightarrow \Pico_{V_K/K},$$
there exists a unique morphism 
$$\beta: \tilde B \longrightarrow  B$$ such that
$$\psi=\tau \circ \beta_K.$$
Actually, since our base field $k$ has characteristic zero, $\tau$ is an embedding.

The inclusion $V_K \hookrightarrow V$ induces 
a morphism of abelian varieties over $K$
$$\phi: \Pico_{V/k,K} \longrightarrow \Pico_{V_K/K}.$$
According to the universal property above, there exists a unique 
morphism of abelian varieties over $k$
$$\alpha: \Pico_{V/k} \longrightarrow B$$ 
such that 
$$\phi= \tau \circ \alpha_K.$$

Beside we may consider the morphism 
$$\pi^\ast : \Pico_{C/k} \longrightarrow \Pico_{V/k}$$
defined by functoriality from $\pi: V \rightarrow C.$

The following Proposition is established as Proposition 3.3 in \cite{hindryetal05}, where references are made to similar earlier results due to Tate, Shioda, and Raynaud.
 
\begin{proposition}\label{Kktrace}
The morphism $\alpha$ is surjective, and the morphism $\pi^\ast$ is an isogeny from $\Pico_{C/k}$ onto the abelian 
variety $(\ker \alpha)^\circ$ defined as  the identity component of the 
$k$-group scheme $\ker \alpha.$
    
\end{proposition}

In brief, the following  diagram of abelian varieties over $k$
$$0 \longrightarrow \Pico_{C/k} \stackrel{\pi^\ast}{\longrightarrow} 
\Pico_{V/k} \stackrel{\alpha}{\longrightarrow} B \longrightarrow 0$$
is ``exact up to some finite group schemes". Together with Poincar\'e's reducibility theorem, this implies that the diagram of abelian groups
\begin{equation}\label{almex}
0 \longrightarrow \Pico_{C/k}(k) \stackrel{\pi^\ast}{\longrightarrow} 
\Pico_{V/k}(k) \stackrel{\alpha}{\longrightarrow} B(k) \longrightarrow 0
\end{equation}
is ``exact up to some finite groups."

\begin{corollary}\label{imtau}

For any line bundle $L$ over $V$, the following conditions are equivalent:

(i) There exists a positive integer $N$ such that the class of 
$L^{\otimes N}_K$ in $\Pic_{V_K/K}(K)$ belongs to $\tau(B(k)).$

(ii) There exist a positive integer $N$ and a line bundle $L'$ over $V$, 
algebraically equivalent to zero, such that, over $V_K$,
$$L^{\otimes N}_K \simeq L'_K.$$

(iii) There exist a positive integer $N$, a line bundle $L'$ over $V$, 
algebraically equivalent to zero, and a vertical divisor $E$ over $V$ such 
that, over $V$,
$$L^{\otimes N} \simeq L' \otimes \cO(E).$$

\end{corollary}

\proof The equivalence of (ii) and (iii) is straightforward. The one of (i) and (ii) follows from the ``almost exactness" of (\ref{almex}) and the fact that any element of the group $\Pico_{V/k}(k)$ has a positive multiple that may be represented by an actual line bundle\footnote{Indeed the functor $\Pico_{V/k}$ may be introduced via sheafification for the \'etale topology, hence given any $\alpha$ in $\Pico_{V/k}(k)$, we can find a finite (separable) extension $k'$ of $k$ and a line bundle $M'$ on $V':= V\otimes_{k}k'$ that represents the image of $\alpha$ in $\Pico_{V/k}(k')$. Then $[k':k].\alpha$ is represented by the line bundle $M:=N_{V'/V}(M')$ on $V$ defined as the norm of $M'$.
} over $V$, algebraically equivalent to zero.
\qed

\subsection{The equivalence of {\bf VA2} and {\bf VA3}}\label{VA23} 
In this section, we complete the proof of Theorem \ref{MainGeom} by 
establishing the equivalence of conditions {\bf VA2} and {\bf VA3}.

The implication {\bf VA2}$\Rightarrow${\bf VA3} follows from the implication 
(ii)$\Rightarrow$(i) in Corollary \ref{imtau} and from the invariance of 
$\deg_{H,D} L$ under algebraic equivalence of line bundles.

Conversely let us consider a line bundle $L$ over $V$ that satisfies {\bf VA3}. 

According to the implication (i)$\Rightarrow$(iii) in Corollary \ref{imtau}, we may 
find a positive integer $N$, a line bundle $L'$ over $V$, algebraically 
equivalent to zero, and a vertical divisor $E$ in $V$ such that 
$L^{\otimes N} \simeq L' \otimes \cO(E).$

Moreover, for every vertical integral divisor $D$ in $V$, we have
$$\deg_{H,D} L^{\otimes N} = N. \deg_{H,D} L = 0$$
 by {\bf VA3}, and 
 $$\deg_{H,D} L' =0$$
 since $L'$ is algebraically equivalent to zero. Therefore,
 $$\deg_{H,D} \cO(E)=0.$$
 Lemma \ref{vertdiv} below shows that, after possibly replacing $L$ and $L'$ 
by some positive power, the divisor $E$ is of the form $\pi^\ast(E')$ for some 
divisor $E'$ on $C$. Consequently,
 $$L^{\otimes N} \otimes \pi^\ast \cO(-E') \simeq L'$$
 is algebraically equivalent to zero, and $L$ satisfies {\bf VA2}.
 
 \begin{lemma}\label{vertdiv}
 For any vertical divisor $E$ on $V$, the following conditions are equivalent:
 
 (i) For every vertical divisor $D$ on $V,$ 
 $$\deg_{H,D} \cO(E)=0.$$
 
 (ii) There exist a divisor $E'$ on $C$ and a positive integer $N$ such that
 $$N\cdot E = \pi^\ast E'.$$

\end{lemma}

 This is well known, at least when $n=2$ and $k$ is algebraically closed, in 
which case it is  traditionally attributed  to  Zariski. 
We refer to \cite{deligne1972b} for a discussion of related results concerning 
intersection theory on surfaces, and to \cite{hindryetal05}, Lemme 2.1 for a similar result. We sketch a proof below for the sake of completeness. 
 
 \begin{proof} To establish the implication (ii)$\Rightarrow$(i), observe  that, for any integral vertical 
divisor $D$ on $V$,  the following equality holds in the Chow group $CH^0(C)$
\begin{equation}\label{encorezero}
\pi_\ast (H^{n-2}.D)=0.
\end{equation}
(Indeed the class in $CH_1 (V)$  of $H^{n-2}.D$ may be represented by a cycle in $Z_1(D)$, and consequently the left-hand side of (\ref{encorezero}) may be represented by a cycle in $Z_1(C)$ supported by $\pi(D)$. Since the latter is zero-dimensional, any such cycle vanishes.) Consequently, by the projection formula, for any divisor  $E'$ in $C$, we have
 \begin{align*}
 \deg_{H,D} \cO(\pi^\ast E') & = \deg_k (H^{n-2}.D.\pi^\ast E')\\
 &= \deg_k (\pi_\ast(H^{n-2}.D).E')\\
 &=0.
\end{align*}

To establish the implication (i)$\Rightarrow$(ii), we may assume that $E$ 
is supported by the fiber $\pi^\ast(P)$ of some closed point $P$ of $C$. 
Let $D_1,\ldots,D_r$ be the components of $\vert\pi^\ast(P)\vert$, and let 
$n_1,\ldots,n_r$ be the positive integers defined by the equality of divisors in $V$:
$$\pi^\ast P=\sum_{i=1}^r n_i. D_i.$$ We want to prove that if some divisor 
supported by $\pi^\ast(P),$ 
$E:=\sum_{i=1}^r m_i. D_i,$
satisfies 
$$\deg_{H,D_j} \cO(E)=0,$$
for every $j\in \{1,\ldots,r\},$ 
then $E$ is a rational multiple of $\pi^\ast(P)$, that is, there 
exists $m$ in $\Q$ such that
$$(m_1,\ldots,m_r)=m(n_1,\ldots,n_r).$$
In other words, we want to establish  that the kernel of the symmetric 
quadratic form attached to the matrix $(q_{ij})_{1\leq i,j\leq r}$ defined by
$$q_{ij}:= \deg_k (H^{n-2}.D_i.D_j)$$
is included in the line $\Q.(n_1,\ldots,n_r).$

To establish this inclusion, observe that the converse implication (ii)$\Rightarrow$(i), 
applied to $D=D_i$ and $E=\pi^\ast P$, shows that
$$\sum_{j=1}^r q_{ij} n_j= 0$$
for every $i\in \{1,\ldots,r\}.$
This yields the following expression for the quadratic form  defined by the $q_{ij}$'s:
$$\sum_{i,j=1}^r q_{ij} m_i m_j = - \sum_{1\leq i < j \leq r} q_{ij} n_i n_j 
\left(\frac{m_i}{n_i}-\frac{m_j}{n_j}\right)^2.$$
The required property now follows from the following two observations:

1) For any two distinct elements $i$ and $j$ in $\{1,\ldots,r\}$, the cycle 
theoretic intersection $D_i.D_j$ of the Cartier divisors $D_i$ and $D_j$ is 
the cycle attached to the intersection scheme $D_i \cap D_j$, which is either 
empty or purely $(n-2)$-dimensional, and consequently, by the ampleness of $H$, 
the degree 
$q_{ij}:= \deg_k (H^{n-2}.[D_i \cap D_j])$
\emph{is non-negative, and positive if $D_i \cap D_j$ is not empty.}

2) The scheme $\pi^\ast (P)$ is connected, and consequently \emph{there is 
no partition of $\{1,\ldots,r\}$ in two non-empty subsets $I$ and $J$ such 
that $(i,j) \in I \times J \Rightarrow q_{ij}=0.$}
\end{proof}

\begin{appendix}

\section{Arithmetic extensions and \v{C}ech cohomology}

Let $X$ be an arithmetic scheme over an arithmetic ring $R=(R,\Sigma, F_\infty),$ $E$ a quasi-coherent $\cO_X$-module on $X,$ 
and $\cU=(U_i)_{i\in I}$
an affine, open covering  of $X$. We fix 
a well ordering on $I$ and consider the (alternating) \v{C}ech complex 
$\bigl(\mathcal{C}^{\cdot}(\mathcal{U},E),\delta\bigr)$ where
\[
\mathcal{C}^{p}(\mathcal{U},E):=\prod\limits_{i_0<\ldots < i_p}
E(U_{i_0\ldots i_p}),
\]
with usual notation
\[
U_{i_0\ldots i_p}=U_{i_0}\cap \ldots \cap U_{i_p},
\]
and where the differential 
$\delta:\mathcal{C}^{p}(\mathcal{U},E)\rightarrow \mathcal{C}^{p+1}(\mathcal{U},E)$ 
is given by the formula
\[
(\delta\alpha)_{i_0,\ldots,i_{p+1}}:=\sum\limits_{k=0}^{p+1}(-1)^k
\alpha_{i_0,\ldots,\widehat i_{k},\ldots,i_{p+1}}\big|_{U_{i_0}\cap 
\ldots \cap U_{i_{p+1}}}.
\]
Recall from \cite[2.5]{bostkuennemann1} that we have a 
natural morphism of locally ringed spaces
\[
\rho:(X_\Sigma(\C),\mathcal{C}^\infty_{X_\Sigma})\longrightarrow 
(X_\Sigma(\C),\cO^\an_{X_\Sigma})\longrightarrow(X,\cO_X),
\]
and that, if
\[
E_\C:= \rho^\ast E
\]
denotes the  $\mathcal{C}^\infty$-module over $X_\Sigma(\C)$ deduced from $E$\footnote{Namely, when $E$ is coherent and locally free, the sheaf of $\mathcal{C}^\infty$-sections over $X_\Sigma(\C)$ of the holomorphic vector bundle $E_\C^{\rm hol}$ deduced from $E$.},  
there is a natural
morphism of $\cO_X$-modules, given by adjunction,
\[
{\rm ad}_E\colon E\,\longrightarrow\, (\rho_*E_\C)^{F_\infty}.
\]
 It induces a morphism of \v{C}ech complexes
\[
\mathcal{C}^{\cdot}(\mathcal{U},{\rm ad}_E)\colon 
\mathcal{C}^{\cdot}(\mathcal{U},E)\,
\longrightarrow\, \mathcal{C}^{\cdot}(\mathcal{U},(\rho_*E_\C)^{F_\infty}).
\]

Concerning cone constructions, in the sequel we use the sign conventions  
discussed in \cite[A.1]{bostkuennemann1}.

We consider the \v{C}ech hypercohomology  $\check{H}^0\bigl(\cU,C({\rm ad}_E))\bigr)$
of the cone $C({\rm ad}_E)$ of ${\rm ad}_E$ with respect to the covering $\cU$, namely 
the cohomology in degree zero of the cone 
$C\bigl(\mathcal{C}^{\cdot}(\mathcal{U},{\rm ad}_E) \bigr).$
This cone is a complex of $R$-modules which 
starts as
\[
0\longrightarrow 
\mathcal{C}^{0}(\mathcal{U},E)
\stackrel{\genfrac{(}{)}{0pt}{}{-\delta}{{\rm ad}_E}}
{\longrightarrow} 
\mathcal{C}^{1}(\mathcal{U},E)\oplus\mathcal{C}^{0}(\mathcal{U},(\rho_*E_\C)^{F_\infty})
\stackrel{\genfrac{(}{)}{0pt}{}{-\delta\,\,\,\,0}{{\rm ad}_E\,\,\delta}}{\longrightarrow} 
\mathcal{C}^{2}(\mathcal{U},E)\oplus\mathcal{C}^{1}(\mathcal{U},(\rho_*E_\C)^{F_\infty})
\]
where $\mathcal{C}^{0}(\mathcal{U},E)$ sits in degree $-1$.
Hence  $\check{H}^0\bigl(\cU,C({\rm ad}_E)\bigr)$
is the quotient
\begin{equation}\label{cocycleh1}
\frac{\Bigl\{(\alpha,\beta) \in \mathcal{C}^1(\cU,E) \oplus 
\mathcal{C}^0\bigl(\cU,(\rho_*E_\C)^{F_\infty}\bigr)\,\big\vert\, \delta\alpha = 0 
\wedge {\rm ad}_E(\alpha)=-\delta(\beta)\Bigr\}}
{\Bigl\{ \bigl(-\delta(\gamma),{\rm ad}_E(\gamma)\bigr)\,\big\vert\,\gamma \in 
\mathcal{C}^0(\mathcal{U},E)\Bigr\} }.
\end{equation}
According to the standard properties of  the cone construction (in the category of $R$-modules) and the very definition of  \v{C}ech cohomology as cohomology of the \v{C}ech complex, this group fits into a natural exact sequence:
\begin{multline}\label{A2}
\check{H}^0\bigl(\cU,E)\bigr)\longrightarrow 
\check{H}^0\bigl(\cU,(\rho_*E_\C)^{F_\infty}\bigr)\longrightarrow 
\check{H}^0\bigl(\cU,C({\rm ad}_E)\bigr) \\ \longrightarrow 
\check{H}^1\bigl(\cU,E)\bigr)\longrightarrow 
\check{H}^1\bigl(\cU,(\rho_*E_\C)^{F_\infty})\bigr).
\end{multline}

\begin{lemma}\label{bigdiaext}
 Let $E$ be quasi-coherent $\cO_X$-module.
There exists a canonical commutative diagram 
\[
\begin{array}{cccccccc}
\Gamma (X,E) & \rightarrow & 
A^0(X_\mathbb{R},E) & \rightarrow 
&\widehat{\rm Ext}^1_X (\mathcal{O}_X,E) & \rightarrow & 
{\rm Ext}^1 (\mathcal{O}_X,E) & \rightarrow 0\\
\downarrow {\scriptstyle } & &\downarrow {\scriptstyle } 
 & &\downarrow {\scriptstyle \hat\rho_{\mathcal{U},E}} & & 
\downarrow {\scriptstyle \rho_{\mathcal{U},E}} \\
\check{H}^0(\cU,E))&\rightarrow &
\check{H}^0(\cU,(\rho_*E_\C)^{F_\infty})&\rightarrow& 
\check{H}^0(\cU,C({\rm ad}_E))&\rightarrow &
\check{H}^1(\cU,E))&\rightarrow 0
\end{array}
\]
with exact horizontal lines where all vertical maps are isomorphisms.
\end{lemma}

\proof
The upper exact sequence is established in \cite[2.2]{bostkuennemann1}.

We have
\[
\check{H}^1\bigl(\cU,(\rho_*E_\C)^{F_\infty})\bigr)
=\check{H}^1\bigl(\rho^{-1}\cU,(E_\C)^{F_\infty})\bigr),
\]
and the latter group is zero as \v{C}ech cohomology of a fine
sheaf with respect to an
open covering vanishes (see for instance \cite[II.3.7  and II.5.2.3 (b)]{godement73}).
Consequently we obtain the lower exact sequence from (\ref{A2}).

The two left vertical maps are given by the natural isomorphisms
induced by the restriction maps of the sheaves $E$ and $(\rho_*E_\C)^{F_\infty}$.

We now define $\rho_{\mathcal{U},E}$.
Let
\[
\mathcal{E}: 0 \longrightarrow E \longrightarrow F 
\stackrel{\pi}{\longrightarrow} \mathcal{O}_X \longrightarrow 0
\]
be an extension of $\cO_X$-modules.  
The map $\pi$ admits a section $\varphi_i$ over each affine scheme $U_i$. 
The difference $\alpha_{ij} = \varphi_j\vert_{U_{ij}} - \varphi_i \vert_{U_{ij}} $ 
determines an element in $ \Gamma(U_{ij}, E)$. 
The family $(\alpha_{ij})_{ij}$ defines a $1$-cocycle in $\mathcal{C}^1(\cU,E)$ 
whose class in $\check{H}^1 (\cU,E)$ does not depend on the choices of the $\varphi_i$. 
One obtains a canonical isomorphism (compare for example \cite[Prop. 2]{atiyah57})
\[
\rho_{\mathcal{U},E}: {\rm Ext}^1_{\cO_X}(\mathcal{O}_X,E) \longrightarrow \check{H}^1  (\cU,E)\,,\,\, 
[\mathcal{E}] \longmapsto \big[(\alpha_{ij})_{ij}\big].
\]

Finally we define $\hat\rho_{\mathcal{U},E}$.
Let $(\mathcal{E},s)$ be an arithmetic extension with $\mathcal{E}$ as above. 
Choose the $\varphi_i$ as before and define
\[
\beta_i=s\vert_{U_i}-{\rm ad}_E(\varphi_i) \in A^{0,0} (U_{i,\mathbb{R}},E).
\]
We have ${\rm ad}_E(\alpha_{ij})= \beta_i\vert_{U_{ij}} - \beta_j\vert_{U_{ij}}$. Hence the pair
$\bigl((\alpha_{ij})_{ij}, (\beta_i)_i\bigr)$ 
determines an element $\hat\rho_{\mathcal{U},E}(\mathcal{E},s)$ in (\ref{cocycleh1}), i.e. 
in $\check{H}^0\bigl(\cU,C({\rm ad}_E)\bigr)$. 
This class does not depend on the choices of the $\varphi_i$.
Given different sections $\tilde\varphi_i$ which led to cocyles 
$\bigl((\tilde\alpha_{ij})_{ij}, (\tilde\beta_i)_i\bigr)$ as above, we consider
\[
\gamma\in\mathcal{C}^0(\mathcal{U},E)\,,\,\,\gamma_i=\varphi_i-\tilde\varphi_i
\]
and get
\[
\genfrac{(}{)}{0pt}{}{-\delta}{{\rm ad}_E}(\gamma)
=(\tilde \alpha,\tilde \beta)-(\alpha,\beta).
\] 
It is straightforward to check that
\[ 
\hat\rho_{\mathcal{U},E} : \widehat{\rm Ext}_X^1 (\cO_X,E)\longrightarrow 
\check{H}^0\bigl(\cU,C({\rm ad}_E)\bigr)\,,\,\,
\bigl[(\mathcal{E},s)\bigr] \longmapsto \bigl[(\alpha_{ij}), (\beta_i)\bigr]
\]
is a group homomorphism which fits into the above commutative diagram.
The five lemma implies that the map $\hat\rho_{\mathcal{U},E}$ is an isomorphism.
\qed

\begin{corollary}\label{bigdiaext2}
Let $F$, $G$ be quasi-coherent $\cO_X$-modules such that 
$F$ is a vector bundle on $X$.
There exists a canonical isomorphism 
\[ 
\hat\rho_{\mathcal{U},F,G} : \widehat{\rm Ext}_X^1 (F,G)\longrightarrow  
\check{H}^0\bigl(\cU,C({\rm ad}_{{\mathcal Hom}(F,G)})\bigr)
\]
which identifies $\widehat{\rm Ext}^1 (F,G)$ with the quotient
(\ref{cocycleh1}) for $E={\mathcal Hom}(F,G)$.
\end{corollary}

\proof
It is proved in \cite[2.4.6]{bostkuennemann1} that there is a canonical
isomorphism
\begin{equation}\label{iso1}
\widehat{\rm Ext}^1_X (F,G)\stackrel{\sim}{\rightarrow}
\widehat{\rm Ext}^1_X (\cO_X,{\mathcal Hom}(F,G))
\end{equation}
which maps the class of an arithmetic extension $(\mathcal{E},s)$ to the pushout
of $(\mathcal{E},s)\otimes F^\vee$ along the canonical map 
$j_F:\cO_X\rightarrow F\otimes F^\vee$.
Let $E={\mathcal Hom}(F,G)$.
We define  $\hat\rho_{\mathcal{U},F,G} $ as the composition of 
the isomorphisms (\ref{iso1}) and $\hat\rho_{\mathcal{U},E}$
in Lemma \ref{bigdiaext}.
\qed

\section{The universal vector extension of
a Picard variety}\label{univectext}
In this Appendix, we recall some basic facts concerning universal vector extensions 
of Picard varieties, which are essentially due to Messing and Mazur (\cite{messing73}, \cite{mazurmessing74}). We show in particular 
that the universal vector extension of the Picard variety ${\rm Pic}_{X/k}^0$
of a smooth projective variety $X$ over a field  $k$ of characteristic zero 
classifies line bundles with integrable connections (see (\ref{mainuve}) \emph{infra}; this is certainly 
well-known but, to our knowledge, only the case where $X$ is an 
abelian variety is treated in the literature).
We also describe the 
maximal compact subgroups of the Lie groups defined  by real and
complex points of universal vector extensions. 

\smallskip

\noindent{\bf B.1.} Let $S$ be a locally noetherian scheme.
In the sequel, we consider a morphism $f:X\rightarrow S$ of schemes which satisfies the following assumptions:
\begin{enumerate}
\item[ i) ]
The morphism $f$ is projective, smooth with geometrically connected
fibers.
\item[ii) ]
The Hodge to de Rham spectral sequence 
\[
E_1^{p,q}=R^qf_*\Omega_{X/S}^p\Rightarrow R^{p+q}f_*\Omega_{X/S}^\cdot
\]
degenerates at $E_1$ and the sheaves $R^qf_*\Omega_{X/S}^p$
are locally free.
\item[iii)]
The identity component ${\rm Pic}_{X/S}^0$ of the Picard scheme 
${\rm Pic}_{X/S}$ is an  abelian scheme.
\end{enumerate}

We observe that i) implies that ${\rm Pic}_{X/S}$ is representable 
by a $S$-group scheme \cite[n.232, Thm. 3.1]{FGA} and that $f_*\cO_X=\cO_S$
holds universally \cite[7.8.6]{EGAIII2}.
Furthermore i) implies ii) if $S$ is of characteristic zero \cite[Th. 5.5]{deligne68}
and i) implies iii) if $S$ is the spectrum of a field of characteristic zero
\cite[8.4]{boschetal90}.
It is shown in \cite[8.3]{katz70} that the formation of the coherent sheaves 
$R^qf_*\Omega_{X/S}^p$ and $R^{n}f_*\Omega_{X/S}^\cdot$
commutes with arbitrary base change if they are locally free for all $p,q \geq 0$ and all $n\geq 0.$

\smallskip

\noindent{\bf B.2.} 
We consider the complex
\[
\Omega_{X/S}^\times:0\longrightarrow \cO_X^* \stackrel{\rm dlog}{\longrightarrow}
\Omega_{X/S}^1\stackrel{\rm d}{\longrightarrow} \Omega_{X/S}^2
\stackrel{\rm d}{\longrightarrow} \ldots
\]
where $\cO_X^*$ sits in degree zero.
The group
\[
{\rm Pic}^\#(X/S):=H^1(X_{\rm fppf},\Omega_{X/S}^\times)
\]
classifies isomorphism classes of pairs $(L,\nabla)$ where $L$ is a line bundle
on $X$ and $\nabla$ is an integrable connection
\[
\nabla:L\longrightarrow L\otimes \Omega_{X/S}^1
\]
relative to $S$ \cite[(2.5.3)]{messing73}.
We denote by
\[
{\rm Pic}_{X/S}^\#:=R^1f_{{\rm fppf} \ast}\Omega_{X/S}^\times
\]
the fppf-sheaf on the category of $S$-schemes associated to the presheaf 
\[
T\mapsto {\rm Pic}^\#(X\times_ST/T)
\]
(see for instance \cite[8.1]{boschetal90}).
If $X_T=X\times_ST$ admits a section over $T$, we have \cite[(2.6.4)]{messing73}
\begin{equation}\label{piccoker}
{\rm Pic}_{X/S}^\#(T)={\rm Coker}\bigl(
{\rm Pic}(T) = {\rm Pic}^\#(T/T)\stackrel{f^*}{\longrightarrow}{\rm Pic}^\#(X\times_ST/T)\bigr).
\end{equation}

\smallskip

\noindent{\bf B.3.} 
If $T/S$ is a fpqc-morphism, we have
\begin{equation}\label{fpqcbc}
{\rm Pic}_{X/S}^\#\times_ST= {\rm Pic}_{X_T/T}^\#.
\end{equation}
Indeed, this is obvious if $T/S$ is fppf. Hence we may assume without loss
of generality that $X/S$ admits a section $\varepsilon.$
This allows us to describe elements in ${\rm Pic}_{X/S}^\#(T)$ as isomorphism classes of triples $(L,\nabla, r)$ where $L$ is  a line bundle on $X_T,$ $\nabla$ is an integrable connection relative to $T$, and $$r:\varepsilon^\ast L\stackrel{\sim}{\longrightarrow} \cO_T$$ is a rigidification. It follows from 
 fpqc-descent 
that ${\rm Pic}_{X/S}^\#$ is in fact an fpqc-sheaf on $S$, which implies (\ref{fpqcbc}). 

We will apply (\ref{fpqcbc}) in the situation where $S$ is the spectrum of an arithmetic ring and $T$ is the spectrum of $\R$ or $\C.$

\smallskip

\noindent{\bf B.4.} 
The exact sequence of complexes
\begin{equation}\label{kuexse}
0\longrightarrow \tau_{\geq 1}\Omega_{X/S}^\cdot\longrightarrow \Omega_{X/S}^\times
\longrightarrow \cO_X^* \rightarrow0 
\end{equation}
induces an exact sequence
\begin{equation}\label{H12}
 H^1(X_{\rm fppf},\tau_{\geq 1}\Omega_{X/S}^\cdot)\longrightarrow {\rm Pic}^\#(X/S)
\longrightarrow H^1(X_{\rm fppf},\cO^*_X)
\longrightarrow H^2(X_{\rm fppf},\tau_{\geq 1}\Omega_{X/S}^\cdot).
\end{equation}
Observe also that the first map in \ref{H12} is injective: this follows from the long exact sequence of $H^0$ 's and $H^1$'s associated with (\ref{kuexse}), from the vanishing of the map 
\[
{\rm dlog}\colon\Gamma(X,\cO_X^*)\longrightarrow \Gamma(X,\Omega_{X/S}^1)
\]
(implied by  Assumption B.1 i)), and the fppf-descent isomorphisms $\Gamma(X,\cO_X^*) \simeq \Gamma(X_{\rm fppf},\cO_X^*)$ and  $\Gamma(X,\Omega_{X/S}^1)\simeq \Gamma(X_{\rm fppf},\Omega_{X/S}^1).$

Using fppf-descent and Assumption B.1 ii), one also gets:
$$
H^1(X_{\rm fppf},\cO^*_X)={\rm Pic}(X), $$
$$
H^2(X_{\rm fppf},\tau_{\geq 1}\Omega_{X/S}^\cdot)
=H^2(X_{\rm Zar},\tau_{\geq 1}\Omega_{X/S}^\cdot),
$$
and
\[
H^1(X_{\rm fppf},\tau_{\geq 1}\Omega_{X/S}^\cdot)
=\ker\bigl(H^0(X_{\rm fppf},\Omega_{X/S}^1) \longrightarrow H^0(X_{\rm fppf},\Omega_{X/S}^2)\bigr)
=\Gamma(S,f_*\Omega_{X/S}^1).
\]
Sheafification of the exact sequence (\ref{H12}) and the injectivity of its first map yields an exact sequence of fppf-sheaves of abelian groups over $S$:
\[
0\longrightarrow f_*\Omega_{X/S}^1\longrightarrow {\rm Pic}^\#_{X/S}
\longrightarrow {\rm Pic}_{X/S}
\stackrel{c}{\longrightarrow} R^2f_{*}\tau_{\geq 1}\Omega_{X/S}^\cdot.
\]
As there are no non-trivial homomorphisms from the abelian scheme
${\rm Pic}_{X/S}^0$ to the coherent sheaf
$R^2f_{*}\tau_{\geq 1}\Omega_{X/S}^\cdot$ by \cite[Lemma p.9]{mazurmessing74},
we have ${\rm Pic}_{X/S}^0\subseteq \ker (c)$. Finally we obtain an extension
of fppf-sheaves  of abelian groups over $S$
\begin{equation}\label{univext}
0\longrightarrow f_*\Omega_{X/S}^1\longrightarrow 
{\rm Pic}_{X/S}^{\#,0}
\longrightarrow {\rm Pic}_{X/S}^0\longrightarrow 0
\end{equation}
where 
\[
{\rm Pic}_{X/S}^{\#,0}
:={\rm Pic}_{X/S}^\#\times_{{\rm Pic}_{X/S}}{\rm Pic}_{X/S}^0.
\]

\smallskip

\noindent{\bf B.5.} 
The {\it universal vector extension} of the abelian scheme ${\rm Pic}_{X/S}^0$
is a group scheme $E_{X/S}$ which fits into an  exact sequence of fppf-sheaves
\begin{equation}\label{univvectext}
0\longrightarrow \mathbb{E}_{A/S}\longrightarrow E_{X/S}
\longrightarrow {\rm Pic}_{X/S}^0\longrightarrow 0
\end{equation}
where $\mathbb{E}_{A/S}$ denotes the Hodge bundle of the
dual abelian scheme $$A:=({\rm Pic}_{X/S}^0)^\vee \stackrel{\pi_A}{\longrightarrow} S,$$
namely
$$\mathbb{E}_{A/S} := \pi_{A \ast}\Omega^1_{A/S}.$$
The universal vector extension may be characterized by its universal property:
given an abelian fppf-sheaf $E'$ and a vector group scheme $M$ which 
fit into an extension of fppf-sheaves of abelian groups
\begin{equation}\label{testext}
0\longrightarrow M\longrightarrow E'
\longrightarrow {\rm Pic}_{X/S}^0\longrightarrow0,
\end{equation}
there exists a unique $\cO_S$-linear morphism $\phi:\E_{A/S}\rightarrow M$ 
such that (\ref{testext}) is isomorphic to the pushout of (\ref{univvectext})
along $\phi$.

By the universal property
there exist unique morphisms $\alpha$ and $\beta$ (of $\cO_S$-modules and $S$-group schemes respectively) such that 
\begin{equation}\label{podxa}
\begin{array}{ccccccccc}
0&\rightarrow &\mathbb{E}_{A/S}&\rightarrow &E_{X/S}&
\rightarrow &{\rm Pic}_{X/S}^0&\rightarrow& 0\\
&&\downarrow{\scriptstyle  \alpha} &&\downarrow{\scriptstyle  \beta}&& \| \\
0&\rightarrow &f_*\Omega_{X/S}^1&\rightarrow &{\rm Pic}_{X/S}^{\#,0}&
\rightarrow &{\rm Pic}_{X/S}^0&\rightarrow& 0
\end{array}
\end{equation}
is a pushout diagram.
The biduality of abelian schemes 
$$  {\rm Pic}_{X/S}^0 \simeq \left( {\rm Pic}_{X/S}^0\right)^{\vee \vee} = A^\vee := {\rm Pic}_{A/S}^0$$
(see for instance  \cite[8.1, Theorem 5]{boschetal90}) yields a canonical isomorphism
\begin{equation}\label{bidualisom}
E_{X/S}\stackrel{\sim}{\longrightarrow}E_{A/S}.
\end{equation}
It is furthermore shown in \cite{mazurmessing74} and \cite{messing73} 
that (\ref{podxa}) with  $X$ replaced by $A$ 
induces a canonical isomorphism 
\[
E_{A/S}\stackrel{\sim}{\longrightarrow}{\rm Pic}_{A/S}^{\#,0}.
\]

Assume that $X/S$ admits a section $\epsilon$. 
There exists a canonical morphism of $S$-schemes, the Albanese morphism of $X$ 
over $S$ relative to the ``base point" $\epsilon$,
\[
\varphi:X\longrightarrow A,
\]
that is characterized by the fact that the pullback of a Poincar\'e bundle for $A$ over $S$ (rigidified along $0$) is isomorphic to a Poincar\'e bundle
 for $X$ (rigidified along $\epsilon$).
The pullback along $\varphi$ induces morphisms
\[
\varphi^*:\mathbb{E}_{A/S}\longrightarrow f_*\Omega_{X/S}^1\,,
\,\,\, \gamma \longmapsto \varphi^*\gamma
\] 
and (using description (\ref{piccoker}))
\[ 
\varphi^*:{\rm Pic}_{A/S}^{\#,0}\longrightarrow {\rm Pic}_{X/S}^{\#,0}
\,,\,\,\,[L,\nabla]\longmapsto [\varphi^*L,\varphi^*\nabla]
\]
such that the diagram
\begin{equation}\label{podxa2}
\begin{array}{ccccccccc}
0&\longrightarrow &\mathbb{E}_{A/S}&\longrightarrow &{\rm Pic}_{A/S}^{\#,0}&
\longrightarrow &{\rm Pic}_{X/S}^0&\longrightarrow& 0\\
&&\downarrow{\scriptstyle \varphi^*} &&\downarrow{\scriptstyle \varphi^*}&& \| \\
0&\longrightarrow &f_*\Omega_{X/S}^1&\longrightarrow &{\rm Pic}_{X/S}^{\#,0}&
\longrightarrow &{\rm Pic}_{X/S}^0&\longrightarrow& 0
\end{array}
\end{equation}
is commutative.
The uniqueness assertion in the universal property implies that the maps 
$\alpha$ and $\beta$ in (\ref{podxa}) are 
given under the canonical identifications
\[
E_{X/S}\stackrel{\sim}{\longrightarrow}E_{A/S}\stackrel{\sim}{\longrightarrow}{\rm Pic}_{A/S}^{\#,0}
\]
by pullback along $\varphi$.

\smallskip

\noindent{\bf B.6.} 
Let $S$ be the spectrum of a field $k$ of characteristic zero.
For a projective, smooth, geometrically connected $S$-scheme, 
our assumptions \ref{univectext} i)- iii) are satisfied by 
B.1. 

Furthermore the morphism $\alpha$ becomes an isomorphism
\begin{equation}\label{figbp}
\alpha:\mathbb{E}_{A/k} :=  \Gamma(A,\Omega^1_{A/k}) \stackrel{\sim}{\longrightarrow} \Gamma(X,\Omega^1_{X/k})
\end{equation}
of $k$-vector spaces. Indeed, to establish that $\alpha$ is an isomorphism, we may replace $k$ by a finite field extension, and therefore assume that $X(k)$ is not empty. If $\varphi: X \to A$ denotes the Albanese morphism associated to some base point $\epsilon$ in $X(k)$, $\alpha$ is given by pull back along $\varphi,$ and is injective as $X$ generates $A$ as an abelian variety, and bijective for dimension reasons (compare for example 
\cite[8.4 Th. 1 b)]{boschetal90}).

It follows that $\beta$ is an isomorphism of $k$-group schemes 
\begin{equation}\label{mainuve}
\beta: E_{X/k}\stackrel{\sim}{\longrightarrow} {\rm Pic}_{X/k}^{\#,0}.
\end{equation} 
In other words, ${\rm Pic}_{X/k}^{\#,0}$ becomes canonically isomorphic to 
the universal vector extension $E_{X/k}$ of 
${\rm Pic}^0_{X/k}$.

When $X(k)$ is not empty, this isomorphism may be described as above, by means of the pull back along the Albanese map $\varphi$ associated to any base point $\epsilon$ in $X(k)$, and using   (\ref{piccoker}) we get a canonical isomorphism of abelian groups: 
\begin{equation}\label{linenabla}
E_{X/k}(k)\simeq \left.\biggl\{(L,\nabla) \left\vert
\genfrac{}{}{0pt}{}{ L \mbox{ line bundle algebraically equivalent to zero on } X }{\nabla \mbox{ integrable connection on }L } \right. \biggr\}
\right/ \sim \,\, ,
\end{equation}
where $\sim$ denotes the obvious isomorphism relations between pairs $(L,\nabla).$

 In general, when $X(k)$ is possibly empty, we may choose a Galois extension
$k'/k$ with Galois group $\Gamma$ such that $X(k')\neq \emptyset$ and
use the obvious identification
\begin{equation}
E_{X/k}(k)=E_{X_{k'}/k'}(k')^\Gamma
\end{equation}
to reduce to the previous case.

\smallskip

\noindent{\bf B.7.} 
If $k=\C$, the extension of commutative complex Lie groups 
\begin{equation}\label{univextoverc}
0\rightarrow \Gamma(X,\Omega_{X/\C}^1)\longrightarrow E_{X/\C}(\C)
\longrightarrow {\rm Pic}^0_{X/\C}(\C)\longrightarrow 0,
\end{equation}
deduced from (\ref{univvectext}) by considering the complex points, 
admits the following description in the complex analytic category (compare \cite[ex.(1.4)]{messing73}).

The Lie algebra of $\Pico_{X/\C}$, hence of the complex Lie group $\Pico_{X/\C}(\C)$, may be be identified with $H^1(X, \cO_X),$ that is, by GAGA, with
$H^1(X(\C),\cO^\an_{X(\C)})$. By considering the exact sequence of sheaves over $X(\C)$
$$0\longrightarrow 2 \pi i \Z \longrightarrow \cO_{X(\C)}^\an \stackrel{\exp}{\longrightarrow} \cO_{X(\C)}^{\an,\ast} \longrightarrow 0$$ and using GAGA, one obtains that the exponential map of $\Pico_{X/\C}$ defines an isomorphism of commutative complex Lie groups:
\begin{equation}\label{expPic}
\frac{H^1(X(\C),\cO_X^{\rm hol})}{H^1(X(\C),2\pi i\Z)} \simeq \Pico_{X/\C}(\C).
\end{equation}

The group of isomorphism classes of pairs $(L,\nabla)$ where $L$ is an algebraic line bundle over $X$ and $\nabla$ an integrable algebraic connection on $L$ --- or equivalently by GAGA, of pairs $(L^\an,\nabla^\an)$  where $L^\an$ is 
a
holomorphic line bundle on
the complex manifold $X(\C)$ and $\nabla^\an$  an integrable, complex analytic
connection on $L^\an$ --- may be identified with $H^1(X(\C), \C^\ast)$, by sending $[(L^\an,\nabla^\an)]$ to the class of the rank one local system $\Ker (\nabla^\an).$ By considering the exponential sequence 
\[
0 \longrightarrow 2 \pi i \Z \longrightarrow \C
\stackrel{\exp}{\longrightarrow} \C^* \longrightarrow 0,
\]
one sees that the group  of classes of such pairs $(L,\nabla)$ with $L$ algebraically equivalent to zero may be identified with the subgroup of $H^1(X(\C), \C^\ast)$ that is the isomorphic image under the exponential map of
$$\frac{H^1(X(\C),\C)}{H^1(X(\C),2 \pi i \Z)}.$$
Using the identification (\ref{linenabla}), we finally obtain an isomorphism
\begin{equation}\label{expE}
\frac{H^1(X(\C),\C)}{H^1(X(\C),2 \pi i \Z)} \simeq E_{X/\C}(\C).
\end{equation}

The analytic de Rham isomorphism
$$H^1(X(\C),\C) \simeq H^1(X(\C), \Omega^{\cdot \, \an}_{X/\C})$$
and the Hodge filtration give rise to a short exact sequence of finite dimensional $\C$-vector spaces
$$
0 \longrightarrow \Gamma(X(\C), \Omega^{1\, \an}_{X/\C})
\longrightarrow H^1(X(\C),\C) \longrightarrow
H^1(X(\C), \cO^\an_{X(\C)})
\longrightarrow 0,
$$
and then, by quotienting its second and third terms by $H^1(X(\C),2 \pi i \Z)$, to a short exact sequence 
of commutative complex Lie groups:
\begin{equation}
0 \longrightarrow \Gamma(X(\C), \Omega^{1\, \an}_{X/\C})
\longrightarrow \frac{H^1(X(\C),\C)}{H^1(X(\C),2 \pi i \Z)} \longrightarrow
\frac{H^1(X(\C), \cO^\an_{X(\C)})}{H^1(X(\C),2 \pi i \Z)}
\longrightarrow 0,
\end{equation}

It turns out that it coincides with the short exact sequence (\ref{univextoverc}) when we take the GAGA isomorphism 
$\Gamma(X,\Omega_{X/\C}^1) \simeq \Gamma(X(\C), \Omega^{1\, \an}_{X/\C})$ and the ``exponential" isomorphisms (\ref{expPic}) and (\ref{expE}) into account.

Observe that the maximal compact subgroup of the Lie group $E_{X/\C}(\C)$ is precisely \begin{equation}\label{maxcptsg}
\frac{H^1(X(\C),2\pi i\R)}{H^1(X(\C),2\pi i\Z)}
\hookrightarrow
\frac{H^1(X(\C),\C)}{H^1(X(\C),2 \pi i \Z)} \simeq E_{X/\C}(\C).
\end{equation}
It is a ``real torus", of dimension the first Betti number of $X(\C).$  Moreover, as a consequence of Hodge theory, the canonical morphism  
$E_{X/\C}(\C)
\rightarrow {\rm Pic}^0_{X/\C}(\C)$ in (\ref{univextoverc})
maps this subgroup isomorphically (in the category of real Lie groups) onto $\Pico_{X/\C}(\C)$.

In this way, we define a canonical splitting 
\begin{equation}\label{sectoverc}
\varsigma \colon {\rm Pic}^0_{X/\C}(\C) \longrightarrow E_{X/\C}(\C).
\end{equation}of (\ref{univextoverc})
in the category of commutative real Lie groups, characterized by the fact that its image lies in --- or equivalently, is --- the maximal compact subgroup of $E_{X/\C}(\C).$

The injection $U(1) \hookrightarrow \C^\ast$ determines an injective morphism $H^1(X(\C),U(1)) \hookrightarrow H^1(X(\C),\C^\ast)$, and the maximal compact group (\ref{maxcptsg}) coincides with the preimage of $H^1(X(\C),U(1))$ under the exponential map. Consequently this group classifies the pairs $(L,\nabla)$ as above, with $L$ algebraically equivalent to zero, such that the monodromy of $\nabla^\an$ lies in $U(1).$ This shows that the real analytic splitting $\varsigma$ may also be described as follows:
 for any  line bundle $L$ over $X$ that is algebraically equivalent
to zero, we may equip $L_\C^{\rm hol}$ with its unique integrable, holomorphic connection  $\nabla_L^u$
with unitary monodromy (\cf   \ref{remcomplex} \emph{supra}); it algebraizes uniquely by GAGA, 
and the assignment
\[
[L]\longmapsto \bigl[(L,\nabla_L^u)\bigr]
\]
defines the group homomorphism (\ref{sectoverc}).

\smallskip

\noindent{\bf B.8.} 
If $k=\R$, the extension
\begin{equation}\label{univextoverr}
0\longrightarrow \Gamma(X,\Omega_{X/\R}^1)\longrightarrow E_{X/\R}(\R)
\longrightarrow {\rm Pic}_{X/\R}^0(\R)\longrightarrow 0
\end{equation}
is obtained from the extension (\ref{univextoverc})
by taking invariants under complex conjugation.
We obtain again a canonical splitting 
\[
\varsigma_\R \colon {\rm Pic}_{X/\R}^0(\R) \longrightarrow E_{X/\R}(\R)
\]
since the splitting (\ref{sectoverc}) is invariant under complex 
conjugation.
The image of $\varsigma_\R$ is the unique maximal compact subgroup 
of $E_{X/\R}(\R)$.

\end{appendix}

\backmatter

\bibliographystyle{alpha}

\end{document}